\documentclass[french]{article} 
\usepackage[utf8]{inputenc}
\usepackage[T1]{fontenc}
\usepackage{lmodern}
\usepackage[a4paper]{geometry}
\usepackage[english]{babel}

\usepackage{fullpage}

\usepackage{amssymb,amsmath,mathtools,amsthm,dsfont}

\usepackage[all]{xy}

\newtheorem{theorem}{Theorem}[section]
\newtheorem{prop}[theorem]{Proposition}

\newtheorem{lemma}[theorem]{Lemma}
\newtheorem{conjecture}[theorem]{Conjecture}

\newcommand{\E}{\mathbb{E}}

\newcommand{\N}{\mathbb{N}}

\renewcommand{\P}{\mathbb{P}}
\newcommand{\Q}{\mathbb{Q}}
\newcommand{\R}{\mathbb{R}}

\newcommand{\Z}{\mathbb{Z}}

\newcommand{\cB}{\mathcal{B}}

\newcommand{\cF}{\mathcal{F}}
\newcommand{\cG}{\mathcal{G}}

\newcommand{\cK}{\mathcal{K}}
\newcommand{\cL}{\mathcal{L}}

\newcommand{\cN}{\mathcal{N}}

\newcommand{\cS}{\mathcal{S}}
\newcommand{\cT}{\mathcal{T}}

\renewcommand{\1}{\mathds{1}}
\renewcommand{\d}{\, \mathrm{d}}

\renewcommand{\epsilon}{\varepsilon}
\renewcommand{\phi}{\varphi}

\newcounter{numeroexo}

\newcommand{\interieur}[1]{\text{Int}(#1)}

\newcommand{\var}{\text{var}}

\newcommand{\parent}[1]{\overset\leftarrow{#1}}
\newcommand{\bad}{\text{BadGap}}
\newcommand{\rejection}{\text{Interference}}
\newcommand{\rhorejection}{\text{Interference}_\rho}
\newcommand{\rhobad}{\text{BadGap}_\rho}

\newcommand{\Overpopulation}{\text{Overpopulation}}
\newcommand{\autre}{\text{Else}}
\newcommand{\smallrho}[3]{\text{Small}_{#3}\left(#1,#2\right)}

\newcommand{\BRW}{BRW}
\renewcommand{\root}{\varnothing}
\newcommand{\survival}{\text{S}}
\newcommand{\troncature}{\Lambda}
\newcommand\gaussienne{\cN}

\title{Percolation in the Boolean model with convex grains in high dimension}
\author{
Jean-Baptiste Gouéré\footnote{IDP. Université de Tours. France. E-Mail:jean-baptiste.gouere@univ-tours.fr},
Florestan Labéy\footnote{IDP. Université de Tours. France. E-Mail:florestan.labey@univ-tours.fr}
}

\begin{document}

\selectlanguage{english}

\maketitle

\begin{abstract}
We investigate percolation in the Boolean model with convex grains in high dimension.
For each dimension $d$, one fixes a compact, convex and symmetric set $K \subset \R^d$ with non empty interior.
In a first setting, the Boolean model is a reunion of translates of $K$.
In a second setting, the Boolean model is a reunion of translates of $K$ or $\rho K$ for a further parameter $\rho \in (1,2)$.
We give the asymptotic behavior of the percolation probability and of the percolation threshold in the two settings.
\end{abstract}

\section{Introduction and main results}

\subsection{Setting}
\label{s:setting}

\paragraph{Boolean model.} 
For any $d \ge 1$, we denote by $\cK(d)$ the set of subsets $K$ of $\R^d$ such that: 
\begin{itemize}
\item The set $K$ is compact, convex and symmetric (that is, for all $x \in K, -x \in K$).
% \item The interior set $\interieur{K}$ is non empty.
\item The Lebesgue measure of $K$ satisfies $|K|=1$.
\end{itemize}
Note that for all $K \in \cK(d)$,  the interior set $\interieur{K}$ is non empty.
Otherwise $K$ would be included in a proper affine subspace of $\R^d$ and its Lebesgue measure would be $0$.

Let $d \ge 1$ and $K \in \cK(d)$.
Let $\nu$ be a finite measure on $(0,+\infty)$ with positive mass.
For simplicity we assume that the support of $\nu$ is bounded.
Let $\lambda > 0$.
Let $\xi$ be a Poisson point process on $\R^d \times (0,+\infty)$ with intensity measure $\d x \otimes \lambda\nu$.
Consider
\[
\Sigma = \Sigma(\lambda,\nu,d,K) = \bigcup_{(c,r) \in \xi} c+r  K.
\]
This is the Boolean model with parameters $\lambda, \nu, d$ and $K$.

We call the sets $c+rK$ the {\em grains} of the Boolean model.
We call $c$ the {\em center} and $r$ the {\em radius} of the grain $c+rK$.

We can write
\begin{equation}\label{e:xichi}
\xi = \{(c,r(c)), c \in \chi\}
\end{equation}
where $\chi$ is a Poisson point process on $\R^d$ with intensity measure $\lambda \nu[(0,\infty)] \d x$.
Given $\chi$, $(r(c))_{c \in \chi}$ is a family of i.i.d.r.v.\ with common distribution $\nu[(0,\infty)]^{-1} \nu$.
As we shall not need this result, we do not give a more formal statement.
We nevertheless think that it can provide some intuition.

%The following alternate description may be more intuitive.
%Let $\chi$ be a Poisson point process on $\R^d$ with intensity measure $\lambda \nu[(0,\infty)] \d x$.
%Condition to $\chi$, consider a family $(R(c))_{c \in \chi}$ of 
%Then
%\[
%\bigcup_{c \in \chi} c+R(c)K
%\]
%is a Boolean model with parameters $\lambda, \nu, K$ and $d$.

We denote by $B(d)$ the Euclidean closed ball of $\R^d$ centered at the origin and such that $|B(d)|=1$.
Note that $B(d)$ belongs to $\cK(d)$.
Most of the times, we will write simply $B$.
We refer to $\Sigma(\lambda,\nu,d,B)$ as a Euclidean Boolean model.
Note that, with our choice of terminology, $B$ is a grain of radius $1$ whereas this is not a Euclidean ball of radius $1$.
This should create no confusion as we will use the term "radius" only in the above sense.

We refer to the books by Kingman \cite{Kingman-livre} and Last and Penrose \cite{Last-Penrose-livre} for background on Poisson point processes.

\paragraph{Percolation in the Boolean model.} 
Fix $r>0$. 
Recall \eqref{e:xichi} and set $\chi^0=\chi \cup \{0\}$ and $r(0)=r$.
We define an unoriented graph structure on $\chi^0$ by putting an edge between $x$ and $y$ if $x+r(x)K$ touches $y+r(y)K$.
As $K$ is convex and symmetric,  
\[
x+r(x)K \text{ touches }y+r(y)K \iff x-y \in [r(x)+r(y)]K  \iff y-x \in [r(x)+r(y)]K.
\]
Note (see \eqref{e:pouf} in Appendix \ref{s:detail-lambda_c}) for all $\lambda>0$,
\begin{align}
& \P[\text{the connected component of the graph }\chi^0\text{ that contains }0\text{ is unbounded}] \nonumber \\
& = \P[\text{the connected component of }\Sigma \cup rK \text{ that contains }0\text{ is unbounded}]. \label{e:ahouicetruc}
\end{align}
We define the percolation threshold as usual (see Appendix \ref{s:detail-lambda_c} for details):
\begin{align}
& \lambda_c(\nu,d,K)  \nonumber \\
 & = \inf \{\lambda >0 : \P[\text{the connected component of the graph }\chi^0\text{ that contains }0\text{ is unbounded}]>0\} 
 \label{e:lambda_c_graphe} \\
 & = \inf \{\lambda >0 : \P[\text{the connected component of }\Sigma \cup rK \text{ that contains }0\text{ is unbounded}]>0\} 
 \label{e:lambda_c_composante}\\
 & = \inf \{\lambda >0 : \P[\text{one of the connected components of }\Sigma\text{ is unbounded}]=1\}. 
 \label{e:lambda_c_01}
\end{align}
In particular this does not depend on the choice of $r>0$.
%For further reference, we also note that, for any $r>0$, we also have
%\begin{equation} \label{e:perco-reformulation}
%\lambda_c = \inf \{\lambda : \P[\text{the connected component of }\Sigma \cup rB \text{ that contains }0\text{ is unbounded}]>0\}.
%\end{equation}
We refer to the book by Meester and Roy \cite{Meester-Roy-livre} for background on percolation in the Boolean model.

As $K$ is symmetric and as the interior set $\interieur{K}$ is non empty, there exists $r^->0$ such that $r^- B \subset K$.
As $K$ is compact, there exists $r^+>0$ such that $K \subset r^+ B$.
Therefore, by a natural coupling (and with an obvious generalization in our notations as $r^\pm B$ may not belong to $\cK(d)$),
\[
\Sigma(\lambda,\nu,r^-B,d) \subset \Sigma(\lambda,\nu,d,K) \subset \Sigma(\lambda,\nu,r^+B,d).
\]
By standard results on percolation in the Euclidean Boolean model, we deduce that $\lambda_c$ is always positive and finite.
%and is positive as soon as
%\[
%\int_{(0,+\infty)} r^d \nu(dr)<\infty.
%\]
%See \cite{G-perco-boolean-model}.
%In this paper we will only deal with $\nu$ of the form $m_a\delta_a$ or $m_a\delta_a+m_b\delta_b$ and therefore the above conditions will be fulfilled.
Actually, we will provide lower and upper bounds on $\lambda_c$ with independent arguments.

\paragraph{Scale invariant thresholds.} 
As in \cite{GM-grande-dimension}, we define
\footnote{In \cite{GM-grande-dimension} there is an extra factor $v_d$ which is the volume of the unit Euclidean ball.
This is due to the fact that, in \cite{GM-grande-dimension}, the unit grain is the unit Euclidean ball whereas here the unit grain is 
a set $K$ of volume $1$.}
\[
c_c(\nu,d,K) = 1-\exp\left(-\lambda_c(\nu,d,K) \int_{(0,+\infty)} r^d \nu(dr)\right) 
\]
and
\[
\tilde\lambda_c(\nu,d,K) = -2^d\ln(1-c_c(\nu,d,K)) = \lambda_c(\nu,d,K) \int_{(0,+\infty)} (2r)^d \nu(dr).
\]
The quantity $c_c(\nu,d,K)$ is the probability that a given point (say $0$) belongs to the critical Boolean model $\Sigma(\lambda_c,\nu,d,K)$.
By ergodicity, this is also the density of the critical Boolean model.
It has therefore a clear geometrical meaning.
It is moreover invariant by scaling: if all radii are multiplied by a constant, $\lambda_c$ changes but $c_c$ remains unchanged.
For these reasons $c_c$ - and the related quantity $\tilde\lambda_c$ - are more convenient when comparing the threshold for different measure $\nu$.
We refer to \cite{GM-grande-dimension} for a more detailed discussion on these topics.

\subsection{Known results in the Euclidean setting}

\subsubsection{The case of a constant radius}

\paragraph{Framework.} 
We are interested here in the case $\nu=\delta_{1/2}$ and $K=B$.
In other words, each grain is a translate of $\frac 1 2 B$.
For $\lambda>0$ and $d \ge 1$, let
\[
C^0 = C^0(\lambda,\delta_{1/2},d,B)
\]
be the connected component of 
\[
\Sigma(\lambda,\delta_{1/2},d,B) \cup \frac 1 2 B
\]
that contains the origin.
As mentioned in \eqref{e:lambda_c_composante}, 
\[
\lambda_c(\delta_{1/2},d,B) = \inf \{\lambda>0 : \P[C^0(\lambda,\delta_{1/2},d,B) \text{ is unbounded}]>0\}.
\]

\paragraph{Link with a Galton-Watson process.} 
In order to provide some intuition, let us describe the link between $C^0$ and some Galton-Watson process.
A more precise description will be given later.

Fix $\lambda>0$ and $d \ge 1$.
Denote by $N_d$ the number of balls that belongs to $C^0$.
The set $C^0$ is unbounded if and only if $N_d$ is  infinite.
Call $\frac 1 2 B$ the ball of generation $0$.
Define the balls of generation $1$ as the random balls of $\Sigma$ that touch $\frac 1 2 B$.
Define the balls of generation $2$ as the random balls of $\Sigma$ that touch at least of ball of generation $1$ 
without being a ball of generation $0$ or $1$
and so on.

Note that $x+\frac 1 2 B$ touches $x'+\frac 1 2 B$ if and only if $x' \in x + B$.
Therefore the number of random balls that intersect a given deterministic ball $x+\frac 1 2 B$ is a Poisson random variable with parameter $\lambda$.
Indeed, this is the number of points of a Poisson point process on $\R^d$ with intensity $\lambda \d x$ that belongs to $x+B$ and $|B|=1$.
If there where no geometrical interference between children (a ill defined term) of different balls, 
then $N_d$ would be the total population $Z_d$ of a Galton-Watson process with offspring distribution Poisson$(\lambda)$.
See Section \ref{s:cluster-brw-basic} for details.
Taking into account interaction, one can see that $N_d$ is actually stochastically dominated by $Z_d$.

Let $\survival(\lambda)$ denote the survival probability of a Galton-Watson process with Poisson$(\lambda)$ offspring.
Turning the above intuition into a proof, one can easily prove
\[
\P[C^0(\lambda,\delta_{1/2},d,B) \text{ is unbounded}] \le \survival(\lambda)
\]
and therefore
\[
\lambda_c(\delta_{1/2},d,B) \ge 1.
\]

\paragraph{Result.}
In \cite{Penrose-high-dimensions}, Penrose proves that these bounds are asymptotically sharp when $d$ tends to $\infty$.

\begin{theorem}[\cite{Penrose-high-dimensions}]  \label{t:penrose}
%\begin{enumerate}
%\item For any $\lambda>0$ and any $d \ge 1$,
%\[
%\P[\text{the connected component of }\Sigma(\lambda,\delta_{1/2},d,B) \cup \frac 1 2 B \text{ that contains }0\text{ is unbounded}] \le \survival(\lambda).
%\]
%\item 
For any $\lambda>0$,
\[
\lim_{d \to \infty} \P[C^0(\lambda,\delta_{1/2},d,B) \text{ is unbounded}] = \survival(\lambda).
\]
Moreover,
\[
\lim_{d \to \infty} \lambda_c(\delta_{1/2},d,B) = 1.
\]
\end{theorem}

The idea is to prove that, when $d$ tends to infinity, the geometrical interference vanish and the genealogy of the balls becomes closer and closer
to the genealogy of the Galton-Watson process.

\subsubsection{The case of radii that can take two values $1$ and $\rho$ with $1<\rho<2$.}

\paragraph{Framework.} 
Let $\rho>1$ and $d \ge 1$.
We consider here the case 
\begin{equation}\label{e:nu}
\nu = \nu_{d,\rho} = \delta_{1/2} + \frac{1}{\rho^d} \delta_{\rho/2}
\end{equation}
and $K=B$.
We refer to \cite{GM-grande-dimension} for the motivation of the choice of $\nu_{d,\rho}$.
Basically, the idea is to keep constant in $d$ the relative influence of grains of different radii.

For all $\lambda>0$, let 
\[
C^0 = C^0(\lambda,\nu_{d,\rho},d,B)
\]
be the connected component of 
\[
\Sigma(\lambda,\nu_{d,\rho},d,B) \cup \frac \rho 2 B
\]
that contains the origin.
As before, we are interested in the percolation probability, that is the probability that $C^0$ is unbounded, and the percolation threshold $\lambda_c$.

\paragraph{Link with a two-type Galton-Watson process.} 
As in the constant radius case, there is a natural genealogy for the balls of $C^0$
and this genealogical structure is stochastically bounded from above by a Galton-Watson process.
In this setting, the Galton-Watson is a two-types Galton-Watson process: there are $\frac 1 2 B$ grains and $\frac \rho 2 B$ grains.
One can check easily that the probability that $C^0$ is unbounded is bounded from above by the survival probability of this Galton-Watson process.
This gives a lower bound on $\lambda_c(\nu_{d,\rho},d,B)$ which is equivalent to $\kappa^c(\rho)^d$, when $d$ tends to $\infty$, where
\begin{equation}\label{e:kappa_c}
\kappa_c(\rho) = \frac{2\sqrt\rho}{1+\rho} < 1.
\end{equation}

%
%We will call them $1$-particles and $\rho$-particles.
%Let
%\begin{equation}\label{e:M-lambda}
%M = 
%\lambda
%\begin{pmatrix}
%1 & \left(\frac{1+\rho}{2\rho}\right)^d \\
%\left(\frac{1+\rho}2\right)^d  & 1
%\end{pmatrix}.
%\end{equation}
%This is the mean matrix of the Galton-Watson process.
%Denote by $N_1$ the random number of children of a $1$-particle which are $1$-particle.
%Denote by $N_\rho$ the random number of children of a $1$-particle which are $1$-particle.
%Then $N_1$ and $N_\rho$ are independent Poisson random variable with parameters given (in that order) by the first row of $M$.
%The children of a $\rho$ behave in a similar way described by the second row of $M$.

\paragraph{Result.} In \cite{GM-grande-dimension}, 
Gouéré and Marchand prove that if $1<\rho<2$, when $d$ tend to infinity, geometrical interference vanish 
and the behavior of the percolation threshold is roughly given
by the behavior of the critical parameter of the two type Galton-Watson process.

\begin{theorem}[\cite{GM-grande-dimension}] \label{t:GM}
For all $\rho \in (1,2)$,
\begin{equation}\label{e:t:GM}
\lim_{d \to \infty} \frac 1 d \ln\left(\lambda_c(\nu_{d,\rho},d,B)\right)  = \ln\left(\kappa_c(\rho)\right).
\end{equation}
\end{theorem}
Note 
$
\tilde\lambda_c(\delta_{1/2},d,B) = \lambda_c(\delta_{1/2},d,B) 
$
and 
$
\tilde\lambda_c(\nu_{d,\rho},d,B) = 2\lambda_c(\nu_{d,\rho},d,B).
$
Therefore, by Theorems \ref{t:penrose} and \ref{t:GM}, for $d$ large enough,
\[
\tilde\lambda_c(\nu_{d,\rho},d,B) < \tilde\lambda_c(\delta_{1/2},d,B)
\]
and then
\[
c_c(\nu_{d,\rho},d,B) < c_c(\delta_{1/2},d,B).
\]
This result, which refutes a conjecture based on numerical estimations in low dimension and some heuristics in any dimension, 
was one of the motivations of \cite{GM-grande-dimension}.
We refer to \cite{GM-grande-dimension} for more details.

\subsubsection{The case of radii that can take two values $1$ and $\rho$ with $\rho>2$: geometry still plays a role in high dimension}

In \cite{GM-grande-dimension-2}, Gouéré and Marchand prove that \eqref{e:t:GM} does not hold when $\rho>2$.
In that case, when $d$ tend to infinity, geometrical interference do not vanish and the behavior of the percolation threshold is given
by a competition between geometrical effects (dependencies due to the lack of space in $\R^d$) and genealogical aspects (given by
the associated Galton-Watson process).

To sum up, geometrical interference do not always vanish in high dimension and geometry can still play a role.
This is one of the motivation of this work where we investigate percolation in the Boolean model in high dimensions beyond the Euclidean case.

\subsection{Our main results.}

\subsubsection{The case of a constant radius}
\label{s:penrose-convexe-statement}

We consider the case $\nu=\delta_{1/2}$.
For $d \ge 1$,  $K \in \cK(d)$ and $\lambda>0$, let 
\[
C^0 = C^0(\lambda,\delta_{1/2},d,K)
\]
be the connected component of 
\[
\Sigma(\lambda,\delta_{1/2},d,K) \cup \frac 1 2 K
\]
that contains the origin.
As mentioned in \eqref{e:lambda_c_composante}, 
\[
\lambda_c(\delta_{1/2},d,K) = \inf \{\lambda>0 : \P[C^0(\lambda,\delta_{1/2},d,K) \text{ is unbounded}]>0\}.
\]
Exactly as in the Euclidean case, one can check that the probability that $C^0(\lambda,\delta_{1/2},d,K)$ is unbounded is bounded from above
by $\survival(\lambda)$, the survival probability of a Galton-Watson process with progeny Poisson$(\lambda)$.
This yields $\lambda_c(\delta_{1/2},d,K) \ge 1$.
This is stated as Proposition \ref{p:penrose-convexe-upper}.

Our first main result is a generalization of Theorem \ref{t:penrose}.

\begin{theorem} \label{t:penrose-convex}
\begin{itemize}
\item For any $\lambda>0$,
\[
\lim_{d \to \infty, K \in \cK(d)} 
\P[C^0(\lambda,\delta_{1/2},d,K) \text{ is unbounded}] = \survival(\lambda).
\]
The convergence is uniform in $K \in \cK(d)$
\footnote{The statement "$f(d,K)$ converges to $\ell$ uniformly in $K \in \cK(d)$ when $d$ tends to $\infty$" means:
\[
\forall \epsilon>0, \exists d_0 \ge 1, \forall d \ge d_0, \forall K \in \cK(d), |f(d,K)-\ell| \le \epsilon.
\].}.
\item Moreover,
\[
\lim_{d \to \infty, K \in \cK(d)} \lambda_c(\delta_{1/2},d,K) = 1.
\]
The convergence is uniform in $K \in \cK(d)$.
\end{itemize}
\end{theorem}

%\begin{theorem} \label{t:penrose-convex}
%\begin{enumerate}
%\item For any $\lambda>0$, any $d \ge 1$ and any $K \in \cK(d)$,
%\[
%\P[\text{the connected component of }\Sigma(\lambda,\delta_{1/2},d,K) \cup \frac 1 2 K \text{ that contains }0\text{ is unbounded}] 
%\le 
%\survival(\lambda).
%\]
%\item For any $\lambda>0$ and any $K \in \cK(d)$,
%\[
%\lim_{d \to \infty} 
%\P[\text{the connected component of }\Sigma(\lambda,\delta_{1/2},d,K) \cup \frac 1 2 K \text{ that contains }0\text{ is unbounded}] 
%= 
%\survival(\lambda).
%\]
%The convergence is uniform in $K \in \cK(d)$.
%\item As a consequence,
%\[
%\lim_{d \to \infty} \lambda_c(\delta_{1/2},d,K) = 1.
%\]
%The convergence is uniform in $K \in \cK(d)$.
%\end{enumerate}
%\end{theorem}

\subsubsection{The case of radii that can take two values $1$ and $\rho$ with $1<\rho<2$.}

\paragraph{Framework.} 
Let $\rho>1$ and $d \ge 1$.
We are interested in the case $\nu=\nu_{d,\rho}$ defined in \eqref{e:nu}.
With $\beta>0$ we associate the intensity 
\[
\lambda = \beta \kappa_c(\rho)^d
\]
where $\kappa_c(\rho)$ is defined in \eqref{e:kappa_c}.
This choice is natural because of Theorem \ref{t:GM}.
Let 
\[
C^0 = C^0(\beta \kappa_c(\rho)^d,\nu_{d,\rho},d,K)
\]
be the connected component of 
\[
\Sigma(\beta \kappa_c(\rho)^d,\nu_{d,\rho},d,K) \cup \frac \rho 2 K
\]
that contains the origin.
As before, we are interested in the percolation probability, that is the probability that $C^0$ is unbounded, and the percolation threshold
\[
\beta_c(\rho,d,K) = \inf \left\{\beta>0: \P\left[ C^0\big(\beta \kappa_c(\rho)^d,\nu_{d,\rho},d,K\big) \text{ is unbounded}\right]>0\right\}.
\]

\paragraph{Link with a two-type Galton-Watson process.} 
As explained above, there is a natural genealogy for the balls of $C^0$.
This genealogical structure is bounded from above by a Galton-Watson process.
In this setting, the Galton-Watson is a two-types Galton-Watson process: there are $\frac 1 2 K$ grains and $\frac \rho 2 K$ grains.
We will call them $1$-particles and $\rho$-particles.
Let 
\begin{equation}\label{e:M}
M=M(\beta,\rho,d) = 
\beta
\begin{pmatrix}
\left(\frac{2\sqrt{\rho}}{1+\rho}\right)^d & \left(\frac{1}{\sqrt\rho}\right)^d \\
\left(\sqrt\rho\right)^d  & \left(\frac{2\sqrt{\rho}}{1+\rho}\right)^d
\end{pmatrix}.
\end{equation}
This is the mean matrix of the  Galton-Watson process.
Denote by $N_1$ the random number of children of a $1$-particle which are $1$-particle.
Denote by $N_\rho$ the random number of children of a $1$-particle which are $\rho$-particle.
Then $N_1$ and $N_\rho$ are independent Poisson random variables with parameters given (in that order) by the first row of $M$.
The children of a $\rho$ behave in a similar way described by the second row of $M$.
Note that the Galton-Watson process does not depend on $K$.
However, its link with the cluster of the origin depends on $K$.
The Galton-Watson process starts with one $\rho$-particle.
Denote by $\survival(\beta,\rho,d)$ its survival probability.
As before, one can easily check that the probability that $C^0$ is unbounded is bounded from above by $\survival(\beta,\rho,d)$.
This gives a lower bound on $\beta_c(\rho,d,K)$.
This is formalized in Proposition \ref{p:t-GM-convex-upper}.

\paragraph{Result.}  In our second main result we prove that, if $1<\rho<2$, the above mentioned inequalities are asymptotically sharp when 
$d$ tends to $\infty$. As above, $\survival(\beta^2)$  denotes the survival probability of  Galton-Watson process with Poisson$(\beta^2)$ progeny.

\begin{theorem} \label{t:GM-convex}
\begin{itemize}
\item Let $\beta>0$ and $\rho \in (1,2)$. Then
\[
\lim_{d \to \infty, K \in \cK(d)}
\P\left[  C^0\big(\beta \kappa_c(\rho)^d,\nu_{d,\rho},d,K\big) \text{ is unbounded}\right] = S(\beta^2).
\]
The convergence is uniform in $K \in \cK(d)$.
\item Let $\rho \in (1,2)$. 
\[
\lim_{d \to \infty, K \in \cK(d)} \beta_c(\rho,d,K) = 1.
\]
The convergence is uniform in $K \in \cK(d)$.
\end{itemize}
\end{theorem}

The second item can be rephrased as follows:
\[
\lambda_c(\rho,d,K) \sim \kappa_c(\rho)^d \text{ as } d \to \infty.
\]
This is a strengthening and a generalization of Theorem \ref{t:GM} which only provides a logarithmic equivalent of $\lambda_c(\rho,d,B)$.
This is the more delicate result of the article.

\paragraph{Open questions and conjectures.} 
There is no hope to generalize such a result for any $\rho$.
Indeed, by \cite{GM-grande-dimension-2}, 
we know that the behavior of the critical threshold of percolation is not given 
by the critical threshold of the associated Galton-Watson process when $\rho>2$ and $K=B$.
However, for $K$ belonging to some families of convex, the result of Theorem \ref{t:GM-convex} could hold for $\rho$ in a larger interval.
Actually, 
the proof suggests that it is in the Euclidean case that the link between the Galton-Watson process and the cluster of the origin is the weakest.
We give more details in Section \ref{s:open_question} and a state a related conjecture in Section \ref{s:conjecture}.

\paragraph{Organization of the paper.} In Section \ref{s:tools-notations} we gather some notations and some results from analysis and
high dimension geometry. 
In Section \ref{s:preuve-penrose-convexe} we prove Theorem \ref{t:penrose-convex}.
In Section \ref{s:preuve-GM-convexe} we prove Theorem \ref{t:GM-convex}.
In both cases, the difficult part is to establish the lower bound on percolation probabilities.
The plans of the proofs are given in Section \ref{s:plan} and \ref{s:plan-GM} once the objects are defined.

\section{Some tools and notations}
\label{s:tools-notations}

\subsection{A couple of notations for random variables}
\label{s:notations}

When $d \ge 1$ and $K \in \cK(d)$ are given, we shall denote by $X_K, X'_K, X''_K$ independent random variables with uniform distribution on $K$.
In the whole of this paper, $\gaussienne$ will denote a standard Gaussian random vector in $\R^2$.

\subsection{Log-concavity}
\label{s:log-concave}

In this section, we gather some well known facts about log-concave functions. 
A map $f:\R^d \to \R_+$ is log-concave if, for all $x,y \in \R^d$ and all $\lambda \in (0,1)$, the following inequality holds:
\[
f(\lambda x + (1-\lambda) y) \ge f(x)^\lambda f(y)^{1-\lambda}.
\]
For example, the indicator of a convex set is log-concave.

If $g,h:\R^d \to \R_+$ are log-concave and measurable and if their convolution is defined everywhere, then $g*h$ is log-concave.
This is a consequence of Prékopa-Leindler inequality (see for example \cite{Gardner-Brunn-Minkowski} or \cite{Prekopa}).

In particular, for all $d \ge 1$ and $K \in \cK(d)$, $X_K$ and $X_K+X_K'$ (see Section \ref{s:notations} for notations) have a log-concave density.

Let $f:\R^d \to \R_+$ be log-concave.
Let us also assume that $f$ is symmetric, that is, for all $x \in \R^d$, $f(x)=f(-x)$.
Then, for all $x \in \R^d, f(x) \le f(0)$.
Indeed, for all $x \in \R^d$,
\begin{align*}
f(0) 
 & \ge f(x)^{1/2}f(-x)^{1/2} \text{ by log-concavity} \\
 & = f(x) \text{ by symmetry.}
\end{align*}

\subsection{A central limit theorem for random variables with log-concave density}

The total variation distance between two random variables $X$ and $Y$ with values in $\R^2$ is
\[
d_{TV}(X,Y) = \sup_{A \in \cB(\R^2)} |\P[X \in A] - \P[Y \in A]|.
\]
The following result is a rewriting of a weak version of Theorem 1.3 in \cite{Klartag-CLT}.

\begin{theorem} \label{t:Klartag-CLT}
There exists a sequence $(\epsilon_{CLT}(d))_d$ which tends to $0$ such that, 
for any centered random variable $X$ in $\R^d$ with log-concave density,
there exists a linear map $L:\R^d \to \R^2$ such that
\[
d_{TV}\left(L\left(X\right),\gaussienne\right) \le \epsilon_{CLT}(d).
\]
\end{theorem}

We will apply this result to random variables $X_K$ or $X_K+X'_K$ (see Section \ref{s:notations} for notations).

\begin{proof}
In \cite{Klartag-CLT}, the result is stated for an isotropic random vector $X$ with log-concave density.
One says that $X$ is isotropic if $\E(X)=0$ and $\var(X)=I_d$.
The random variable $X$ may not be isotropic.
But as $X$ has a density, $\var(X)$ is positive definite.
As moreover $X$ is centered, there there exists an inversible linear map $T:\R^d\to\R^d$ such that $T(X)$ is isotropic.
Moreover $T(X)$ still have a log-concave density.
Therefore Theorem 1.3 of \cite{Klartag-CLT} applies to $T(X)$ and provides,
for any $\epsilon>0$ and any $d$ large enough depending only on $\epsilon$, an orthogonal projection $\pi$ on a plane $P \subset \R^d$
such that $\pi(T(X))$ is close in total variation distance to a standard Gaussian random vector $\gaussienne_P$ on $P$: 
$d_{TV}(\pi(T(X)),\gaussienne_P) \le \epsilon$.
From this result, one gets Theorem \ref{t:Klartag-CLT} with $L=\phi \circ \pi \circ T$ where $\phi$ is an isometry between $P$ and $\R^2$.
\end{proof}

\subsection{A concentration result for random variables with uniform distribution on convex sets}

Let $d \ge 1$.
One says that a random variable on $\R^d$ is isotropic if $\E(X)=0$ and $\var(X)=I_d$.
Let $K \in \cK(d)$.
Let $X_K$ be a random variable with uniform distribution on $K$.
We will need the following definition:
\begin{equation}\label{e:adapted}
T\text{ is adapted to }X_K\text{ if }T\text{ is an invertible linear map from }\R^d\text{ to }\R^d\text{ such that }T(X_K)\text{ is isotropic}.
\end{equation}
There always exists such maps.
The following result is a rewriting of a weak version of Theorem 1.4 in \cite{Klartag-CLT}.

\begin{theorem} \label{t:Klartag-concentration}
Let $\epsilon>0$.
Let $\ell \ge 1$.
There exists $d_0$ such that, 
for all $d \ge d_0$, all $K \in \cK(d)$ and all $T$ adapted to $X_K$,
\[
\P\left[\left|\frac{\left\| T (X_K^{1}+\dots+X_K^\ell)\right\|_2}{\sqrt{d}} - \sqrt{\ell} \right|\ge \epsilon \right]\le \epsilon
\]
where $X_K^{1},\dots,X_K^\ell$ are independent random variables with uniform distribution on $K$.
\end{theorem}

\begin{proof} By Theorem 1.4 in \cite{Klartag-CLT}, there exists absolute constants $c,C >0$ such that the following holds.
For all $d \ge 1$, all isotropic random variable $X$ on $\R^d$ with a log-concave density and all $\eta \in [0,1]$,
\[
\P\left[\left|\frac{\left\| X\right\|_2}{\sqrt{d}} - 1 \right|\ge \eta \right]\le Cd^{-c\eta^2}.
\]

Let $d \ge 1, \ell \ge 1, K \in \cK(d)$ and $T$ a map adapted to $X_K$.
The random variable
\[
X=\frac{T(X_K^1+\dots+X_K^\ell)}{\sqrt \ell}
\]
is isotropic.
Furthermore, as $K$ is convex, the common density of the $X^i_K$ is log-concave.
Therefore the common density of the $T(X_K^i)$ is log-concave.
By stability by convolution (see Section \ref{s:log-concave}), the density of $T(X_K^1+\dots+X_K^\ell)$ is log-concave.
Therefore, the density of $X$ is log-concave.

Let $\epsilon \in (0,\sqrt{\ell})$. 
We apply the result stated at the beginning of the proof with $X$ defined as above and $\eta=\epsilon/\sqrt{\ell}$.
We get
 \[
\P\left[\left|\frac{\left\| T(X_K^1+\dots+X_K^\ell)\right\|_2}{\sqrt d \sqrt \ell} - 1 \right|\ge \frac{\epsilon}{\sqrt \ell} \right]
\le Cd^{-c\epsilon^2\ell^{-1}}.
\]
The result follows.
\end{proof}

\subsection{Rearrangement inequalities}

We will need to following version of Riesz's rearrangement inequality.
This is Theorem 3.7 in \cite{Lieb-Loss} in the simple setting of indicator function.
When $A$ is a Borel subsets of $\R^d$ with finite Lebesgue measure, we denote by $A^*$ the Euclidean ball centered at the origin such that $|A|=|A^*|$.

\begin{theorem} \label{t:riesz} Let $d \ge 1$. Let $A_1, A_2, A_3$ be Borel subsets of $\R^d$ with finite Lebesgue mesure.
Then
\[
\int_{\R^d} \1_{A_1} * \1_{A_2}(x)\1_{A_3}(x) \d x 
\le 
\int_{\R^d} \1_{A_1^*}*\1_{A_2^*}(x)\1_{A_3^*}(x) \d x.
\]
\end{theorem}

\subsection{A conjecture}
\label{s:conjecture}

This section in not necessary for the main results of the article.
Let $d \ge 1$ and $K \in \cK(d)$. Let $\|\cdot\|_K$ be the norm defined by $\|x\|_K = \inf\{r >0 : x \in r K\}$. 
In particular $\|\cdot\|_B=\|\cdot\|_2$.
For any $r>0$, by Theorem \ref{t:riesz},
\[
\P[\|X_K+X_K'\|_K \le r] = \int_{\R^d} \1_K * \1_K(s)\1_{rK}(s) \d s \le \int_{\R^d} \1_B * \1_B(s)\1_{rB}(s) \d s = \P[\|X_B+X_B'\|_2 \le r].
\]
In other words, $\|X_B+X_B'\|_B$ is stochastically dominated by $\|X_K+X_K'\|_K$.
Similar results holds for the sum of more copies of $X_B$ or $X_K$.
Numerical simulations suggest the following related conjecture.
For any $p \in [1,+\infty]$ we denote by $\|\cdot\|_p$ the usual $\ell^p$ norm on $\R^d$ and by $B_p \in \cK(d)$ 
the associated ball of volume $1$ centered at $0$.
In particular, $B_p=B$.

\begin{conjecture} For any $r>0$, the map from $[1,+\infty]$ to $[0,1]$ defined by 
\[
p \mapsto \P\left[\|X_{B_p}+X'_{B_p}\|_p \le r\right]
\]
is increasing on $[1,2]$ and decreasing on $[2,+\infty]$.
\end{conjecture}

When $d$ tends to $\infty$, $\|X_{B_p}+X'_{B_p}\|_p$ converges in probability to a constant $N(p)$.
Using the representation of uniform random variables on $B_p$ given in \cite{Barthe-al-lp-ball} one easily gets, for $p \in [1,\infty)$,
\[
N(p) = \left[\frac{p}{4\Gamma^2(1+p^{-1})} \int_{\R^d \times \R^d} |x+y|^p e^{-|x|^p-|y|^p} \d x \d y\right]^{\frac 1 p}
\]
where $\Gamma$ is the Gamma function.
One also easily check that $N(\infty)=2$.
Here is a related easy looking conjecture: $N$ is decreasing on $[1,2]$ and increasing on $[2,\infty]$.
We have not been able to prove this result. 
However, as a consequence of the above discussion, on can check that one always has $N(p) \ge N(2)=\sqrt 2$.

We formulate similar conjecture for the sum of a higher number of copies of $X_{B_p}$.

\subsection{Sum of independent random variable uniformly distributed on an Euclidean ball}

We state and proof here a well known result for which we have no ready reference.
Recall that $B(d)$ denotes the Euclidean closed ball of $\R^d$ centered at the origin such that $|B(d)|=1$.

\begin{lemma} \label{l:somme-unif-boule}
For all $\epsilon>0$,
$
\P[ X_{B(d)} + X'_{B(d)} \in (\sqrt{2}-\epsilon) B(d)] \to 0 \text{ as } d \to \infty.
$
\end{lemma}

\begin{proof}
Let $X(d)$ and $X'(d)$ be independent random variable with uniform distribution on the unit ball $\widetilde B(d)$ of $\R^d$.
We will prove that $\|X(d)+X'(d)\|_2$ tends to $\sqrt{2}$ in probability when $d$ tends to $\infty$.
By a scaling argument, this yields the required result.

By independence and isotropy, 
\begin{equation} \label{e:course}
\E\big[\big|(X(d),X'(d))\big|\big] = \E\big[ \big|(X(d), \|X'(d)\|_2 e_1)\big| \big] \le \E\big[ \left|(X(d), e_1)\right| \big]
\end{equation}
where $(\cdot,\cdot)$ denotes the scalar product in $\R^d$ and $e_1$ the first vector of the canonical basis.
But 
\[
 (X(d), e_1)^2 + \cdots	 +  (X(d), e_d)^2 = \|X(d)\|_2^2 \le 1.
\]
Therefore, by symmetry, 
$
\E\big[(X(d), e_1)^2\big] \le d^{-1} \to 0
$
and then
$
\E\big[ \left|(X(d), e_1)\right| \big] \to 0.
$
By \eqref{e:course} and Markov inequality, we deduce from the above limit that $(X(d),X'(d))$ tends to $0$ in probability.
As moreover $\|X(d)\|^2_2$ and $\|X'(d)\|^2_2$ tends $1$ in probability, 
we get that $\|X(d)+X'(d)\|^2_2$ tends to $2$ in probability as $d$ tends to $\infty$.
\end{proof}

\subsection{Notations for branching random walks (\BRW)}
\label{s:brw}

Let $\lambda>0$ and $d \ge 1$.
Let $\Delta$ be a random variable on $\R^d$.
Let $\cS$ be a finite subset of $\R^d$.
Let $k$ be an integer.
Let $L$ be a map from $\R^d$ to $\R^2$.
Let $M>0$.
	
\begin{itemize}
\item $\tau^\lambda$ denotes a Galton-Watson tree with Poisson$(\lambda)$ progeny.
The generation $|x|$ of a particle $x$ of $\tau^\lambda$ is the graph distance from $x$ to the root.
If $x$ is not the root we denote by $\parent{x}$ the parent of $x$, that is the unique particle of generation $|x|-1$
on the shortest path from the root to $x$.
\item We see a \BRW\ as a a random tree (or forest) where each node $x$ is called a particle and possesses a location $V(x) \in \R^d$.
The notation $\overline\tau^{\lambda, d,\Delta ; \cS}$ denotes a \BRW\ such that:
\begin{itemize}
\item The process starts with one particle located at each point of $\cS$.
\item Each particle $x$ has $N(x)$ children located at $V(x)+\Delta(x,1), \dots, V(x)+\Delta(x,N(x))$
where the distribution of $N(x)$ is Poisson$(\lambda)$, where the $\Delta(x,\cdot)$ have the same distribution as $\Delta$ and 
where all variables $N(\cdot)$ and $\Delta(\cdot,\cdot)$ are independent.
\end{itemize}
If $A$ is a subset of $\R^d$, we write
\[
\overline\tau^{\lambda,d,\Delta ; \cS}(A)
\]
as a short notation for
\[
\sum_{x \in \overline\tau^{\lambda,d,\Delta ; \cS}	} \1_A(V(x)),
\]
that is the number of particles of the \BRW\ located in $A$.
\item $\overline\tau^{\lambda,d,\Delta}$ is a short notation for $\overline\tau^{\lambda,d,\Delta;\{0\}}$.
\item If $\overline\tau$ is a \BRW, then $\overline\tau_{\le k}$ denotes the restriction of $\overline\tau$ to the $k$ first generations
and $\overline\tau_k$ denotes its restriction to the $k$-th generation. 
%We moreover define the event 
%\begin{align*}
%\nicebrw{\overline\tau_{\le k}}{M}{\troncature}=\{
%& \text{The Euclidean length of each step of }\overline\tau_{\le k}\text{ is at most  }\troncature\\
%& \text{ and there are at most }M\text{ individual in }\overline\tau_{\le k}\}.
%\end{align*}

%\item $\overline\tau^{\lambda,\Delta,L,\troncature ; \cS}$ is a pruned version of $\overline\tau^{\lambda,\Delta ; \cS}$
%where we removed any particle $y$ of generation at least $1$ and its progeny as soon as $\|L(V(y)-V(\parent{y}))\|_2 > \troncature$.
%It has the same distribution as $\overline\tau^{\lambda',\Delta' ; \cS}$ where 
%\[
%\lambda' = \lambda \P[\|L(\Delta)\|_2 \le \troncature] \text{ and } \P[\Delta' \in \cdot]  = \P[\Delta \in \cdot | \|L(\Delta)\|_2 \le \troncature].
%\]
\end{itemize}

Let moreover $\beta>0$ and $\rho>1$. 
In Section \ref{s:preuve-GM-convexe}, we will use a \BRW\ which will alternate $\rho$-particles for even generations and $1$-particles for odd generations.
\begin{itemize}
\item $\overline\tau^{\beta,\rho,d,\Delta;\cS}$ denotes a two-types \BRW.
There are $1$-particles and $\rho$-particles.
It starts with one $\rho$-particle located at each point of $\cS$.
\begin{itemize}
\item Each $1$-particle $x$ has $N(x)$ children which are $\rho$-particles located at $V(x)+\Delta(x,1), \dots, V(x)+\Delta(x,N(x))$
where the distribution of $N(x)$ is Poisson$(\beta \sqrt\rho^{-d})$, where the $\Delta(x,\cdot)$ have the same distribution as $\Delta$ and 
where all variables $N(\cdot)$ and $\Delta(\cdot,\cdot)$ are independent.
\item Each $\rho$-particle $x$ has $N(x)$ children which are $1$-particles located at $V(x)+\Delta(x,1), \dots, V(x)+\Delta(x,N(x))$
where the distribution of $N(x)$ is Poisson$(\beta \sqrt\rho^d)$, where the $\Delta(x,\cdot)$ have the same distribution as $\Delta$ and 
where all variables $N(\cdot)$ and $\Delta(\cdot,\cdot)$ are independent.
\end{itemize}
%\item $\overline\tau^{\beta,\rho,d,\Delta,L,\troncature ; \cS}$ is a pruned version of $\overline\tau^{\beta,\rho,d,\Delta;\{\cS\}}$
%where we removed any particle $y$ of generation at least $1$ and its progeny as soon as $\|L(V(y)-V(\parent{y}))\|_2 > \frac{1+\rho}2\troncature$.
\item As above we will omit $\cS$ when $\cS=\{0\}$.
\item We define the event 
\begin{align*}
\smallrho{\overline\tau^{\beta,\rho,d,\Delta;\cS}_{\le k}}M\rho
=\{ & \text{the number of }\rho\text{ particles of } \overline\tau^{\beta,\rho,d,\Delta;\cS}_{\le k} \text{ is at most }M\\
& \text{ and the number of children of any }\rho\text{ particle is at most }M \sqrt\rho^d\}.
\end{align*}
%and the event
%\begin{align*}
%\nicebrweven{\overline\tau_{\le 2k}}{M}{\troncature}=\nicebrw{\left(\overline\tau_{2n}\right)_{n \le k}}{M}{\troncature}.
%\end{align*}
%The latter event can be rephrased as follows. 
%Start with $\overline\tau_{\le 2k}$.
%Sample at even times.
%We get the first $k$ generations of a one-type \BRW\ $\overline\tau'$.
%We then wonder whether the event $\nicebrw{\overline\tau'_{\le k}}{M}{\troncature}$ occur.
\end{itemize}

We will need some ordering on the nodes of a tree.
We can for example formalize trees using Neveu formalism.
In this formalism, nodes are finite sequences of positive integers (the idea is that $(4,2)$ is the second child of the forth child of the root $\root$).
See for example Section 2.2 of \cite{Shi-Saint-Flour}.
We can then order the nodes by lexicographic order.
We will refer to this order as Neveu order.

\section{Proof of Theorem \ref{t:penrose-convex}}
\label{s:preuve-penrose-convexe}

\subsection{Framework and $\widehat C^0$}
\label{s:framework:t:penrose-convex}

Let $\lambda>0$, $d \ge 1$ and $K \in \cK(d)$.
Let $\chi$ be a Poisson point process on $\R^d$ with intensity measure $\lambda \d x$.
Note that $\xi=\{(c,1/2), c \in \chi\}$ is a Poisson point process on $\R^d \times (0,+\infty)$ 
with intensity measure $\d x \otimes \lambda \delta_{1/2}$.
Set $\chi^0 = \chi \cup \{0\}$.
As in Section \ref{s:setting}, we define an unoriented graph structure on $\chi^0$ by putting an edge between $x,y \in \chi^0$ if $y-x \in K$.
Let $\widehat C^0$ be the connected component of the graph $\chi^0$ that contains $0$.
We are interested (see \eqref{e:ahouicetruc} with $r=1/2$) in
\[
\P[\#\widehat C^0=\infty] 
\]
and in
\[
\lambda_c(\delta_{1/2},d,K) = \inf \{\lambda>0 : \P[\#\widehat C^0=\infty] > 0\}.
\]

%Consider
%\[
%\Sigma = \bigcup_{c \in \chi} \left(c+\frac 1 2 K\right).
%\]
%This is a version of the Boolean model with parameters $\lambda, \delta_{1/2}, K$ and $d$.
%and consider
%\[
%\Sigma^0 = \bigcup_{c \in \chi^0} \left(c+ \frac 1 2 K\right) = \frac 1 2 K \cup \Sigma.
%\]
%Let $C^0$ be the connected component of $\Sigma^0$ that contains the origin.
%Set 
%\[
%\widehat C^0 = C^0 \cap \chi^0.
%\]
%
%In other words, $\widehat C^0$ is the set of centers of the grains that are connected to the origin in $\Sigma^0$.
%Note that
%\[
%\frac 1 2 K - \frac 1 2 K = K
%\]
%by symmetry and convexity of $K$.
%Therefore, for all $x,y \in \R^d$, 
%\begin{align*}
%\left( x + \frac 1 2 K\right) \cap \left(y + \frac 1 2 K\right)
% & \iff x-y \in \frac 1 2 K - \frac 1 2 K \\
% & \iff x \in y + K \\
% & \iff y \in x+K \text{ by symmetry of } K.
%\end{align*}
%
%As moreover $\chi^0$ is locally finite and $K$ is compact and connected,
%$\widehat C^0$ is also the component of the origin 
%in the undirected graph with vertex set $\chi^0$ and with and edge between $x$ and $y$ if and only if $y \in x + K$.
%
%Note that, as $\chi^0$ is locally finite, $C^0$ is unbounded in and only if $\widehat C^0$ is infinite.

\subsection{Branching random walk and cluster of the origin}
\label{s:cluster-brw}

The content of this section (with the exception of Section \ref{s:open_question}) is essentially contained in Section 5 of \cite{Penrose-high-dimensions}.
Roughly, the aim is to explain than $\widehat C^0$, or a subset of $\widehat C^0$, can be seen as the set of positions of a pruned \BRW.
The framework is the same as in Section \ref{s:framework:t:penrose-convex}.

\subsubsection{Basic construction}
\label{s:cluster-brw-basic}

\paragraph{Exploring a subset of $\widehat C^0$.} 
We explore some part of $\widehat C^0$ by revealing successively parts of the point process $\chi$.
We will define inductively a tree with root $0$ and where each node is a point of $\widehat C^0$.
We will call the graph distance from a node $x$ to the root $0$ the generation of $x$.
Start with $A=\{0\}$ and $B=\emptyset$.
At each stage of the algorithm, perform the following steps.
\begin{enumerate}
\item Select one of the elements of $A$ according to any given rule and call it $x$.
\item Add each point of 
\[
\chi \cap \big((x+K) \setminus (B+K)\big)
\]
to $A$ where 
\[
B+K = \bigcup_{b \in B} b+K.
\]
Put an arrow from $x$ to each of the points $y$ we have just added to $A$.
We think about each such $y$ as a child of $x$.

The idea is the following. Before performing Step 2:
\begin{itemize}
\item The points $y \in \chi$ such that there is an edge between $x$ and $y$ are the points of $\chi \cap (x + K)$.
\item All the points of  $\chi \cap (B+K)$ have been revealed. All of them are in $A \cup B$ and in $\widehat C^0$. 
\item None of the points of $\chi \setminus (B+K)$ have been revealed.
\end{itemize}
\item Move $x$ from $A$ to $B$.
\item If $A$ is non-empty, go back to Step 1.
\end{enumerate}
There are two cases.
\begin{itemize}
\item The algorithm terminates. In that case it provides a tree whose set of nodes is $B$. We have $B=\widehat C^0$.
\item The algorithm does not terminate. In that case it provides a sequence of growing trees. We can consider the limit tree.
Denote its set of nodes by $B$. We have $B \subset \widehat C^0$ and $\#B = \#\widehat C^0 = \infty$.
\end{itemize}
In all cases, we have 
\[
B \subset \widehat C^0 \text{ and } \#B = \#\widehat C^0.
\]
Note that, when $\widehat C^0$ is infinite, it may happen that $B \neq \widehat C^0$.
This depends on the rule used to select $x$ in Step 1.
If we use 
the rule "Select one of the elements of $A$ of minimal generation according to any given rule"
or the rule "Select the element of $A$ of minimal Euclidean distance to $0$",
then $B=\widehat C^0$.
We shall not need this result and, on the contrary, we will find it convenient to have some freedom in the choice of the point $x$ in Step 1.

\paragraph{Building a subset of $\widehat C^0$ by pruning a \BRW.}
The key is the following. Before Step 2 in the above exploration process, condition to what we have revealed, 
\[
\chi \cap \big((x+K) \setminus (B+K)\big)
\]
is a Poisson point process of intensity 
\[
\lambda \1_{(x+K) \setminus (B+K)}(y) \d y.
\]
Furthermore, generating such a Poisson point process can be done as follows.
\begin{itemize}
\item Generate $N$ a random variable with Poisson$(\lambda)$ distribution. Note that $\lambda = \lambda |x+K|$ as $|K|=1$.
\item Generate $Y_1, \dots, Y_N$ i.i.d.\ random variables with uniform distribution on $x+K$.
\item The random set $\{Y_i : Y_i \not\in B+K\}$ has the required distribution.
\end{itemize}

This enables us to build a set which, under a suitable coupling (which we will assume implicitly henceforth) is a subset of $\widehat C^0$.
Consider the \BRW\ 
\[
\overline\tau=\overline\tau^{\lambda,d,X_{K}}
\]
where $X_{K}$ is uniformly distributed on $K$ (see Section \ref{s:brw} for notations on \BRW).
Start with $A=\{\root\}$ where $\root$ is the root of $\overline\tau$ (which is located at $0$) and $B=\emptyset$.
At each stage of the algorithm, perform the following steps.
\begin{enumerate}
\item Select one of the elements of $A$ according to any given rule and call it $x$.
\item Consider successively the children (in the \BRW) $y$ of $x$ in any order and for each of them, do the following: 
\[
\text{if there does not exist }x' \in B\text{ such that }V(y) \in V(x')+K\text{, then add }y\text{ to }A.
\]
Let us introduce some vocabulary for future reference.
We say that the other children of $x$ are rejected
\begin{equation}
\label{e:interference}
\text{because of {\em interference} between } x \text{ and } x' \text{ or because of {\em interfence} with }x'.
\end{equation}
We say that 
\begin{equation}
\label{e:regino_interference}
V(x')+K \text{ is the {\em region of interference} of } x'.
\end{equation}
In other words, we reject $x$ because of interference with $x'$ when $x$ belongs to the region of interference of $x'$.
A large part of our proof will be devoted to establish the fact that, in high dimension, there is not too much interference.
\item Move $x$ from $A$ to $B$.
\item If $A$ is non-empty, go back to Step 1.
\end{enumerate}
If the algorithm terminates it provides a set $B$.
Otherwise, it provides an increasing sequence of $B$ (one $B$ for each stage) and we define $B$ as the union of all those $B$ at different stages.
In any case,
\[
\{V(x), x \in B\} \subset \widehat C^0 \text { and } \#B = \#\widehat C^0.
\]
Note that $\{V(x), x \in B\}$ is the set of positions of a pruned version of $\overline\tau$.

\subsubsection{An open question}
\label{s:open_question}

This section is not necessary for the main result of the article.
Let $d \ge 1$ and $K \in \cK(d)$. By Theorem \ref{t:riesz},
\[
\P[X_K+X'_K \in K] = \int_{\R^d} \1_{K} * \1_{K}(s)\1_{K}(s) \d s \le  \int_{\R^d} \1_B * \1_B(s)\1_B(s) \d s = \P[X_B+X_B' \in B].
\]
We now use the setting and notations of Section \ref{s:cluster-brw-basic}.
To construct $\widehat C^0$, we have to reject because of interference the grandchildren of the root of $\overline\tau^{\lambda,d,X_K}$
that belong to $K$.
If we condition by the underlying tree $\tau^{\lambda,d,X_K}$, the probability that a given grandchildren is rejected for this reason is
$\P[X_K+X'_K \in K]$. 
As seen above, this probability is maximal when $K=B$.
This remark and further similar considerations may suggest that the percolation probability could be minimal when $K=B$
and therefore the percolation threshold $\lambda_c$ could be maximal when $K=B$.
This would be coherent with numerical simulations in low dimension. 
In particular, by reducing percolation to a local criteria and then using Monte Carlo methods, Balister, Bollob\'{a}s and Walters 
provided in \cite{Balister-al-square-disc}
the following $99.99\%$ confidence intervals for $\lambda_c(\delta_{1/2},K,2)$:
$[4.508,4.515]$ when $K=B$ and $[4.392,4.398]$ when $K=[-1/2,1/2]^2$.

\subsubsection{Constructing a smaller set by over-pruning}
\label{s:cluster-brw-over}

We introduce here a variant of the previous constructions.
The general idea is that, by rejecting more than necessary at some stage, we may be able to have a better control on interference at a later stage.

\paragraph{Exploring a smaller subset of $\widehat C^0$.} 
Start with $A=\{0\}$ and $B=\emptyset$.
At each stage of the algorithm, perform the following steps.
\begin{enumerate}
\item Select one of the elements of $A$ according to any given rule and call it $x$.
\item Let $K(x)$ be a subset of $x+K$ defined according to any given rule. It can depend on $A, B$ and $x$.
Consider successively the points of
\[
\chi \cap \left(K(x) \setminus \bigcup_{b \in B} K(b)\right)
\]
in any order and add some of them to $A$ according to any given rule.
Put an arrow from $x$ to each of the points we have just added.

With the vocabulary introduced above, the region of interference of $x$ is $K(x)$.

\item Move $x$ from $A$ to $B$.
\item If $A$ is non-empty, go back to Step 1.
\end{enumerate}
We get in the end some set that we denote by $B^-$.
We have
\[
B^- \subset \widehat C^0 \text{ and } \#B^- \le \widehat C^0.
\]

\paragraph{Building a smaller subset of $\widehat C^0$ by over-pruning a \BRW.}  
As in Section \ref{s:cluster-brw-basic}, consider the \BRW\ 
$
\overline\tau=\overline\tau^{\lambda,d,X_{K}}
$
and start with $A=\{\root\}$ and $B=\emptyset$.
At each stage of the algorithm, perform the following steps.
\begin{enumerate}
\item Select one of the elements of $A$ according to any given rule and call it $x$.
\item Let $K(x)$ be a subset of $V(x)+K$ defined according to any given rule. 
We say that this is the {\em region of interference} of $x$.
It can depend on $A, B$ and $x$.
Consider successively in any given order the children $y$ of $x$ whose position belongs to
\begin{equation}\label{e:keller}
K(x) \setminus \left(\bigcup_{b \in B} K(b)\right)
\end{equation}
and, for each of them, decide according to any given rule whether one adds it to $A$ or not.

We thus distinguish (somehow artificially) two kinds of over-pruning:
\begin{enumerate}
\item By taking $K(x)$ smaller than $V(x)+K$ we may reduce the number of children of $x$ in the pruned \BRW.
However, as we reduced the region of interference of $x$, we may reject less points at a later stage of the algorithm.
\item Similarly, by adding to $A$ only some of the children of $x$ which belongs to \eqref{e:keller}, 
we further reduce the number of children of $x$ in the pruned \BRW.
However, we can gain some properties on the positions of the points of the pruned \BRW\ which may help us controlling the interference at a
later stage of the algorithm.
\end{enumerate}
Both over-pruning amounts to rejecting some of the children $y$ of $x$ belonging to 
\[
(V(x)+K) \setminus  \bigcup_{b \in B} K(b)
\]
according to some give rule and we will describe it this way in the specific over-pruning we will use in this work.
%
%\footnote{Equivalently, we can consider successively in any given order the children $y$ of $x$ which belongs to
%\[
%x+K \setminus \left(\bigcup_{b \in B} K(b)\right)
%\]
%and, for each them, decide according to any given new rule whether one adds it to $A$ or not.
%The new rule is the same as the former rule except that we add the requirement "$y$ belongs to $K(x)$".
%
%
%We have two level of over-pruning. Among all children $y$ of $x$ which belongs to $x+K$ we reject:
%\begin{enumerate}
%\item The children who does not belongs to $K(x)$.
%\item The children who does not fulfill some further properties.
%\end{enumerate}
%This distinction is quite artificial. We could
%Indeed the first over-pruning is a special case of the second one ("not belonging to $K(x)$" is a property).
%However we want to emphasize the fact that. use only the second level of over-pruning by replacing th
%one by reducing the interference region and one by further rejecting some children of $x$ which belongs to \eqref{e:keller}.
%One could have always set $K(x)=x+K$ and only reject more children 
%}.

\item Move $x$ from $A$ to $B$.
\item If $A$ is non-empty, go back to Step 1.
\end{enumerate}
If the algorithm terminates it provides a set $B$.
Otherwise, it provides an increasing sequence of $B$ (one $B$ for each stage) and we define $B$ as the union of all those $B$ at different stages.
In any case,
\[
\{V(x), x \in B\} \subset \widehat C^0 \text { and } \#B \le \#\widehat C^0.
\]
The set $\{V(x), x \in B\}$ is the set of positions of an over-pruned version of $\overline\tau$.

\paragraph{Intuitive rephrasing.} Let us rephrase one of the key ideas at a more intuitive level.
When we look for the children of a particle $x$, we reveal the relevant Poisson point process in a (subset of) the 
interference region $K(x)$. Therefore at any later stage of the construction we have to reject any particle which fall in $K(x)$.
However, we do not have to care about any rejected particle or about any particle whose children we never consider: they generate
no interference.

\subsection{Proof of the upper bound on percolation probability}
\label{s:penrose-convexe-upper}

The aim is to prove the following proposition, which is the easy part of Theorem \ref{t:penrose-convex}.
We refer to Section \ref{s:penrose-convexe-statement} for notations.

\begin{prop} \label{p:penrose-convexe-upper}
\begin{itemize}
\item For all $\lambda>0, d \ge 1$ and $K \in \cK(d)$,
$
\P[C^0(\lambda,\delta_{1/2},d,K) \text{ is unbounded}] \le S(\lambda).
$
\item For all $d \ge 1$ and $K \in \cK(d)$, $\lambda_c(\delta_{1/2},d,K) \ge 1$.
\end{itemize}
\end{prop}
\begin{proof}
We use the framework of Section \ref{s:framework:t:penrose-convex} 
and the \BRW\ $\overline\tau=\overline\tau^{\lambda,d,X_K}$ of Section \ref{s:cluster-brw}.
By the discussion in Section \ref{s:cluster-brw-basic},
we know that $\#\widehat C^0$ has the same distribution as the total population of a pruned version of $\overline\tau$.
Therefore $\#\widehat C^0$ is stochastically dominated by the total population of $\overline\tau$,
that is by the population of a Galton-Watson tree $\tau^\lambda$.
Using \eqref{e:ahouicetruc} for the first equality, we thus get
\[
\P[C^0(\lambda,\delta_{1/2},d,K) \text{ is unbounded}]  
= \P[\#\widehat C^0(\lambda,\delta_{1/2},d,K) =\infty]
\le \P[\#\tau^\lambda=\infty]
= \survival(\lambda).
\]
This is the first part of Proposition \ref{p:penrose-convexe-upper}.
The second part follows as $S(\lambda)=0$ for $\lambda \le 1$ and as $\lambda_c(\delta_{1/2},d,K)$ is the infimum 
of all $\lambda>0$ such that $\P[C^0(\lambda,\delta_{1/2},d,K) \text{ is unbounded}]$ is positive.
\end{proof}

\subsection{Proof of  the lower bound on percolation probability}
\label{s:t:penrose-convexe-preuve-lower}

We use the framework of Subsection \ref{s:framework:t:penrose-convex}.
Our aim is to prove the following result.

\begin{theorem} \label{t:penrose-convex-lower}
Let $\lambda>0$ and $\epsilon>0$. There exists $d_0 \ge 1$ such that, for all $d \ge d_0$ and all $K \in \cK(d)$,
\[
\P[ \#\widehat C^0(\lambda,\delta_{1/2},d,K) = \infty] \ge \survival(\lambda)- \epsilon.
\]
\end{theorem}

\begin{proof}[Proof of Theorem \ref{t:penrose-convex} using Theorem \ref{t:penrose-convex-lower}]
Let $\lambda>0$ and $\epsilon>0$.
Let $d_0$  be as given by Theorem \ref{t:penrose-convex-lower}.
Let $d \ge d_0$ and $K \in \cK(d)$.
We then have,
\[
\survival(\lambda) 
\ge \P[C^0(\lambda,\delta_{1/2},d,K) \text{ is unbounded}] 
= \P[\#\widehat C^0(\lambda,\delta_{1/2},d,K) =\infty]
\ge \survival(\lambda)- \epsilon.
\]
The first inequality is Proposition \ref{p:penrose-convexe-upper}. 
The equality is \eqref{e:ahouicetruc}.
The second inequality is due to our choice of $d_0$.
We have proved the first part of Theorem  \ref{t:penrose-convex}.

Let $\lambda>1$. 
Then $\survival(\lambda)>0$. 
Set $\epsilon=\survival(\lambda)/2>0$.
By the first part of Theorem \ref{t:penrose-convex}, there exists $d_0$ such that, for all $d \ge d_0$ and $K \in \cK(d)$,
\[
\P[C^0(\lambda,\delta_{1/2},d,K) \text{ is unbounded}] \ge \survival(\lambda)- \epsilon = \survival(\lambda)/2>0.
\]
Therefore, for all $d \ge d_0$ and all $K \in \cK(d)$, $\lambda_c(\delta_{1/2},d,K) \le \lambda$.
Combined with the second part of Proposition \ref{p:penrose-convexe-upper}, this gives the second part of Theorem \ref{t:penrose-convex}.
\end{proof}
	
\subsubsection{Good gaps} 
\label{s:good-gaps}

\paragraph{Result.}
Let $d \ge 1$.
For any $K \in \cK(d)$, denote as usual by $X_K, X'_K$ i.i.d.r.v.\ uniformly distributed on $K$. 
For all $\eta>0$, set
\[
G(d,K,\eta) = \{z \in \R^d : \P(z + X'_K \not\in K) \ge 1-\eta\}
\]
where $G$ stands for "good gap".
Note that $G$ is symmetric because $K$ is symmetric.

\begin{lemma} \label{l:good_gap} 
There exists a sequence $(\epsilon_G(d))_d$ that tends to $0$ such that 
for all $d \ge 1$, all $K \in \cK(d)$, all $a \in \R^d$ and all $\eta>0$ :
\[
\P[a + X_K \not\in G(d,K,\eta)] \le \eta^{-1}\epsilon_G(d).
\]
\end{lemma}

\begin{proof}
Let $d \ge 1$.
For all $K \in \cK(d)$ and all $a \in \R^d$ we have
\begin{align*}
\P(a+X_K+X'_K \in K)
 & = \int_{\R^d}  \1_K*\1_K(x)\1_{K-a}(x) \d x \\
 & \le \int_{\R^d} \1_B*\1_B(x)\1_B(x) \d x  \text{ by Theorem \ref{t:riesz} (rearrangement inequality)}\\
 & = \P(X_B+X'_B \in B).
\end{align*}
So, for all $\eta>0$,
\begin{align*}
\P[a + X_K \not\in G(d,K,\eta)] 
 & = \P[ \P(a+X_K+X'_K \not\in K | X_K) < 1-\eta] \\
 & = \P[ \P(a+X_K+X'_K \in K | X_K) >  \eta] \\
 & \le \eta^{-1}\E[ \P(a+X_K+X'_K \in K | X_K)] \\
 & = \eta^{-1}\P(a+X_K+X'_K \in K ) \\
 & \le \eta^{-1}\P(X_B+X'_B \in B) \text{ by the above discussion.}
\end{align*}
But $\P(X_B+X'_B \in B) \to 0$ when $d$ tends to $\infty$.
This is (up to a scaling) (21) of Lemma 3 in \cite{Penrose-high-dimensions}.
This is also a consequence of Lemma \ref{l:somme-unif-boule}.
This concludes the proof.
\end{proof}

\subsubsection{Embedding of a two-dimensional lattice in $\R^d$ - oriented percolation}
\label{s:embedding}

\paragraph{Embedding of a two-dimensional lattice in $\R^d$.} 
Set $\cL = \{(i,j) \in \N \times \Z : i+j \text{ odd } |j| < i\}$ and $\overline{\cL}=\cL \cup\{(0,0)\}$.
We see $\overline{\cL}$ as an oriented graph by putting and edge from $(0,0)$ to $(1,0)$ and,
for every $(i,j) \in \cL$, one edge from $(i,j)$ to $(i+1,j+1)$ and one from $(i,j)$ to $(i+1,j-1)$.
We consider on $\overline{\cL}$ the lexicographical order.
Thus, the first vertices of $\overline\cL$ are $(0,0),(1,0),(2,-1),(2,1),(3,-2),\dots$

When $d \ge 1$ and $K \in \cK(d)$ are given, one fixes a linear map $L:\R^d\to\R^2$ given by Theorem \ref{t:Klartag-CLT} for $X_K$.
With each $(i,j) \in \overline\cL$ we associate the sets $A(i,j) \subset \R^2$ and $A_L(i,j) \subset \R^d$ defined by
\[
A(i,j) = (i,j) + 4^{-1} D  \text{ and } A_L(i,j)=L^{-1}(A(i,j))
\]
where $D$ the Euclidean unit ball of $\R^2$ (not to be confused with $B=B(d)$, which is the Eulidean ball of $\R^d$ of volume $1$).
The sets $A_L(i,j)$ are pairwise disjoint.
Moreover $0$ belongs to $A_L(0,0)$.

Using to this embedding, we will compare the cluster of the origin to a supercritical percolation process on $\overline\cL$.

\paragraph{Oriented percolation on $\cL$.}
Let $\theta(u)$ be the probability that 
there exists an infinite open path originating from $(1,0)$ in a Bernoulli site percolation on the graph $\cL$ with parameter $u$.
We will need\footnote{By the first inequality of (1) of Section 10 of \cite{Durrett-oriented} with $N=0$, 
one has (thanks to a contour argument):
$1-\theta(u) \le \sum_{m \ge 4} 3^m(1-u)^{m/4}$.}:
\begin{equation}\label{e:percolimite}
\lim_{u \to 1} \theta(u)=1.
\end{equation}
We refer to \cite{Durrett-oriented} for background on oriented percolation.

\subsubsection{An estimate about \BRW}

The aim of this section is to prove Lemma \ref{l:briqueG} and its consequence Lemma \ref{l:brique}.
Recall the notations from Section \ref{s:brw}.
Recall in particular that $\survival(\lambda)$ is defined as the survival probability of a Galton-Watson process with progeny Poisson$(\lambda)$.

\begin{lemma} \label{l:brique}
Let $\lambda>1$ and $\epsilon>0$.
There exists $m,d_0,k,M \ge 1$ such that, for all $d \ge d_0$ and all centered random variable $X$ in $\R^d$ with log-concave density,
the following properties hold where $L$ is any map given by Theorem \ref{t:Klartag-CLT} for $X$.
\begin{itemize}
\item For all $z \in A_L(0,0)$, 
\[
\P\left[\overline\tau^{\lambda,d,X;\{z\}}_k (A_L(1,0)) \ge m 
\text{ and } \overline\tau^{\lambda,d,X;\{z\}}_{ \le k}(\R^d) \le M\right] \ge \survival(\lambda)-\epsilon.
\]
\item For all $(i,j) \in \cL$ and all subset $\cS \subset A_L(i,j)$ of cardinality $m$,  
\begin{align*}
& \P\Big[\overline\tau^{\lambda,d,X;\cS}_{k} (A_L(i+1,j+1)) \ge m 
\text{ and } \overline\tau^{\lambda,d,X;\cS}_{k} (A_L(i+1,j-1)) \ge m 
 \text{ and } \overline\tau^{\lambda,d,X;\cS}_{ \le k}(\R^d) \le M\Big] \ge 1-\epsilon.
\end{align*}
\end{itemize}
\end{lemma}

Lemma \ref{l:brique} is a consequence of the following result (the proof is given below).
In Section \ref{s:preuve-penrose-convexe} we will prove and use another consequence of Lemma \ref{l:briqueG}.
Recall that $\gaussienne$ denotes a standard Gaussian random vector in $\R^2$.

\begin{lemma} \label{l:briqueG}
Let $\lambda>1$ and $\epsilon>0$.
There exists $m,k,M \ge 1$ such that the following properties hold.
\begin{itemize}
\item For all $z \in A(0,0)$, 
\[
\P\left[\overline\tau^{\lambda,2,\gaussienne;\{z\}}_k (A(1,0)) \ge m 
\text{ and } \overline\tau^{\lambda,2,\gaussienne;\{z\}}_{\le k}(\R^2) \le M \right] \ge \survival(\lambda)-\epsilon.
\]
\item For all $(i,j) \in \cL$ and all subset $\cS \subset A(i,j)$ of cardinality $m$,  
\begin{align*}
\P\Big[
 & \overline\tau^{\lambda,2,\gaussienne;\cS}_{k} (A(i+1,j+1)) \ge m \text{ and }  
 \overline\tau^{\lambda,2,\gaussienne;\cS}_{k} (A(i+1,j-1)) \ge m  
  \text{ and }  
 \text{ and } \overline\tau^{\lambda,2,\gaussienne;\cS}_{\le k}(\R^2) \le M \Big] \ge 1-\epsilon.
\end{align*}
\end{itemize}
\end{lemma}

Let us start by the following lemmas. 

\begin{lemma} \label{l:continuity-psi} The map $\survival$ is continuous on $[0,+\infty)$. 
%anciennement (1,+\infty) ce qui suffisait en dehors d'un item du lemme 
\end{lemma}
\begin{proof} For any $\lambda>1$, $1-\survival(\lambda)$ is the only real $u \in (0,1)$ such that $u = \exp(\lambda(u-1))$ which we write
\[
\frac{\ln(u)}{u-1} = \lambda.
\]
But $u \mapsto \ln(u)/(u-1)$ defines a decreasing homeomorphism $f$ from $(0,1)$ to $(1,+\infty)$\footnote{
This is for example a consequence of the following facts:
$f(0+)=+\infty$; $f(1-)=1$; for all $v \in (0,1)$, $f(1-v) = 1+v/2+v^2/3+v^3/4+\dots$}.
As moreover $\survival$ vanishes on $[0,1]$, the result follows.
\end{proof}

\begin{lemma} \label{l:prebriqueG}
Let $\lambda>1$ and $\alpha,\beta,\epsilon>0$. 
Let $m \ge 1$.
There exists $k \ge 1$ such that, for all $x \in \alpha D$,
\[
\P\left[\overline\tau^{\lambda,2,\gaussienne}_k ( x + \beta D) \ge m\right] \ge \survival(\lambda)-\epsilon.
\]
\end{lemma}

\begin{proof} 
Let us first define a few constants.
Let $\lambda_1 \in ]1,\lambda[$ be such that
\[
\survival(\lambda_1) \ge \survival(\lambda)-\epsilon.
\]
Such a $\lambda_1$ exists by continuity of $\survival$, see Lemma \ref{l:continuity-psi}.
Let $\troncature_1>0$ be such that
\[
\lambda_1 := \lambda\P[\|\gaussienne\|_2 \le \troncature_1].
\]
Let $C>0$ be such that, for all $n \ge 1$ and all $z \in (\alpha + \troncature_1 n) D$,
\begin{equation}\label{e:zarb}
\P\left[\gaussienne  \in \frac{1}{n}z+\frac{\beta}{n} D\right] \ge \frac{C}{n^2}.
\end{equation}

Let us show the existence of $\eta>0$ and $n_0 \ge 1$ such that, 
\begin{equation}\label{e:truc1}
\forall n \ge n_0, \P\left[\overline\tau^{\lambda,2,\gaussienne}_n(n\troncature_1 D) \ge \eta \lambda_1^n\right] \ge \survival(\lambda) - 3\epsilon.
\end{equation}
Denote by $\overline\tau^{\lambda,2,\gaussienne,\troncature_1}$ the \BRW\ obtained from $\overline\tau^{\lambda,2,\gaussienne}$ by pruning one particle and its progeny as soon as it makes a step
whose Euclidean norm is larger than $\troncature_1$.
The new \BRW\ has then the same distribution as $\overline\tau^{\lambda_1,2,\gaussienne_1}$ where $\gaussienne_1$ has the distribution of $\gaussienne$ condition to
$\|\gaussienne\|_2 \le \troncature_1$.
We have (see for example \cite{Athreya-ney-book} page 9)
\[
\frac{\overline\tau^{\lambda_1,2,\gaussienne_1}_n(\R^2)}{\lambda_1^n} \to W \text{ a.s.}
\]
where $W$ is a random variable which is positive on the event $\{\overline\tau^{\lambda_1,2,\gaussienne_1} \text{ survives}\}$
whose probability is $S(\lambda_1)$.
We can then chose $\eta>0$ such that $\P[W \ge 2\eta] \ge S(\lambda_1)-\epsilon$.
We can now fix $n_0 \ge 1$ such that, for all $n \ge n_0$,
\[
\P\left[\overline\tau^{\lambda_1,2,\gaussienne_1}_n(\R^2) \ge \eta \lambda_1^n\right] \ge \survival(\lambda_1)-2\epsilon \ge \survival(\lambda) - 3\epsilon
\]
and then
\[
\P\left[\overline\tau^{\lambda_1,2,\gaussienne_1}_n(n\troncature_1 D) \ge \eta \lambda_1^n\right] \ge \survival(\lambda) - 3\epsilon.
\]
One deduces \eqref{e:truc1}.

Now, let us show
\begin{equation}\label{e:truc2}
\forall n \ge 1, \forall x \in \alpha D, \forall y \in n\troncature_1 D,
\P\left[\overline\tau^{\lambda,2,\gaussienne;\{y\}}_{n^2}(x+\beta D) \ge 1\right] \ge \frac{\survival(\lambda)C}{n^2}.
\end{equation}
Let $n \ge 1, x \in \alpha D$ and $y \in n\troncature_1 D$.
We have
\[
\P\left[\overline\tau^{\lambda,2,\gaussienne;\{y\}}_{n^2}(x+\beta D) \ge 1\right] \ge \survival(\lambda)\P\left[y+\sum_{i=1}^{n^2} \gaussienne_i  \in x+\beta D\right]
\]
where the $\gaussienne_i$ are independent copies of $\gaussienne$.
To prove this, it is sufficient to consider, on the event $\{\overline\tau^{\lambda,2,\gaussienne;\{y\}} \text{ survives}\}$,
the position of a given particle of generation $n^2$.
As $\sum_{i=1}^{n^2} \gaussienne_i$ has the same distribution as $n \gaussienne$, we deduce
\[
\P\left[\overline\tau^{\lambda,2,\gaussienne;\{y\}}_{n^2}(x+\beta D) \ge 1\right] \ge \survival(\lambda)\P\left[\gaussienne  \in \frac{1}{n}(x-y)+\frac{\beta}{n} D\right].
\]
Thanks to \eqref{e:zarb} we deduce \eqref{e:truc2}.

We now combine \eqref{e:truc1} and \eqref{e:truc2} and get, for all $x\in \alpha D$ and all $n \ge n_0$,
\[
\P\left[\overline\tau^{\lambda,2,\gaussienne}_{n+n^2}(x+\beta D) \ge m\right] 
\ge 
(\survival(\lambda)-3\epsilon)\P\left[\text{binomial}\left(\lfloor \eta \lambda_1^n\rfloor,\frac{\survival(\lambda)C}{n^2}\right) \ge m\right].
\]
This can be proven by first conditioning with respect to the $n$ first generations of the \BRW,
working on the event $\{\overline\tau^{\lambda,2,\gaussienne}_n(n\troncature_1 D) \ge \eta \lambda_1^n\}$
(whose probability is controlled by \eqref{e:truc1}) and
using \eqref{e:truc2} with $\lfloor \eta \lambda_1^n\rfloor$ independent \BRW\ originating from different positions of $\overline\tau^{\lambda,2,\gaussienne}_n$
in $n \Lambda_1 D$.
But for $n$ large enough we have
\[
\lfloor \eta \lambda_1^n\rfloor\frac{\survival(\lambda)C}{n^2} - m \ge  \frac 1 2 \eta\lambda_1^n\frac{\survival(\lambda)C}{n^2}
\]
and then, by Chebyshev's inequality,
\[
\P\left[\text{binomial}\left(\lfloor \eta \lambda_1^n\rfloor,\frac{\survival(\lambda)C}{n^2}\right) \le m\right]
\le 
\eta \lambda_1^n\left(\frac 1 2 \eta\lambda_1^n\frac{\survival(\lambda)C}{n^2}\right)^{-2} \to 0 \text{ as } n \to \infty.
\]
As a consequence, there exists $n \ge n_0$ such that, for all $x\in \alpha D$,
\[
\P\left[\overline\tau^{\lambda,2,\gaussienne}_{n+n^2}(x+\beta D) \ge m\right] 
\ge 
\survival(\lambda)-4\epsilon.
\]
We fix such a $n$.
The lemma is proven with $k=n+n^2$.
\end{proof}

\begin{proof}[Proof of Lemma \ref{l:briqueG}]
Fix $m$ such that
\[
\big(1-\survival(\lambda)+\epsilon\big)^m \le \epsilon. 
\]
This is possible if $\epsilon>0$ is small enough, which we can assume.
We apply Lemma \ref{l:prebriqueG} with $\alpha=3$ and $\beta=1/4$.
We get $k$ such that, for all $x \in 3D$,
\[
\P\left[\overline\tau^{\lambda,2,\gaussienne}_{k} ( x + 4^{-1} D) \ge m \right] \ge \survival(\lambda)-\epsilon.
\]
By natural couplings between the involved \BRW\ we get, for all $z \in \R^2$;
\begin{align*}
 \overline\tau^{\lambda,2, \gaussienne ; \{z\}}_{  k} ( A(1,0) )
  & = \overline\tau^{\lambda,2,\gaussienne ; \{z\}}_{  k} ( (1,0)+4^{-1} D ) \\
  & = \overline\tau^{\lambda,2,\gaussienne}_{ k} ( (1,0)-z+ 4^{-1} D).
\end{align*}
But if $z \in A(0,0)$, then $(1,0)-z \in 3D$. 
Therefore, for all $z \in A(0,0)$,
\begin{equation}\label{e:plage1}
\P\left[\overline\tau^{\lambda,2,\gaussienne ; \{z\}}_{  k} ( A(1,0) ) \ge m \right] \ge \survival(\lambda)-\epsilon.
\end{equation}
Let $(i,j) \in \cL$. 
Let $\cS \subset A(i,j)$ such that $\#\cS=m$.
Using the independence between the \BRW\ originating from the different points of $\cS$, an argument similar to the above one and the definition of $m$,
we get
\[
\P\left[\overline\tau^{\lambda,2,\gaussienne;\cS}_{ k} (A(i+1,j+1)) \ge m \right] \ge 1-\big(1-\survival(\lambda)+\epsilon\big)^m \ge 1-\epsilon.
\]
With the same argument, we get
\[
\P\left[\overline\tau^{\lambda,2,\gaussienne;\cS}_{ k} (A(i+1,j-1)) \ge m \right] \ge 1-\big(1-\survival(\lambda)+\epsilon\big)^m \ge 1-\epsilon.
\]
Therefore
\begin{equation}\label{e:plage2}
\P\left[\overline\tau^{\lambda,2,\gaussienne;\cS}_{ k} (A(i+1,j+1)) \ge m 
\text{ and } \overline\tau^{\lambda,2,\gaussienne;\cS}_{ k} (A(i+1,j-1)) \ge m \right] \ge 1-2\epsilon.
\end{equation}
Let $M$ be large enough to ensure $\P\left[\overline\tau^{\lambda,2,\gaussienne}_{\le k}(\R^2) \ge M/m\right] \le \epsilon/m$.
By a natural coupling, we get, for all $\cS \subset \R^2$ of cardinality at most $m$,
\begin{equation}\label{e:plage3}
\P\left[\overline\tau^{\lambda,2,\gaussienne,\cS}_{\le k}(\R^2) \ge M\right]  \le m\epsilon/m = \epsilon.
\end{equation}
The lemma follows from \eqref{e:plage1}, \eqref{e:plage2} and \eqref{e:plage3}.
\end{proof}

\begin{proof}[Proof of Lemma \ref{l:brique}]
Let $m,k,M$ be given by Lemma \ref{l:briqueG}.
Let $d_0$ be such that $\epsilon_{CLT}(d) \le \epsilon/M$ for all $d \ge d_0$ where $\epsilon_{CLT}$ appears in Thereom \ref{t:Klartag-CLT}.
%By Theorem \ref{t:Klartag-CLT}, we can choose $d_0$ such that for all $d \ge d_0$ and random variable $X$ in $\R^d$ with a log-concave density, 
%there exists a linear map $L:\R^d \to \R^2$ and coupling between $L(X)$ and $\gaussienne$ such tat,
%\[
%\P[L(X) \neq \gaussienne] \le \epsilon/M.
%\]
Let $d \ge d_0$, $X$ be a centered random variable in $\R^d$ with log-concave density and $L$ be any map given by Theorem \ref{t:Klartag-CLT}.
Let $\cS$ be a finite subset of $\R^2$.
With an appropriate coupling,
\[
\P\left[
\left\{\overline\tau^{\lambda,2,\gaussienne; \cS}_{\le k}(\R^2) \le M \right\}
\setminus
\left\{\overline\tau^{\lambda,2,L(X); \cS}_{\le k} = \overline\tau^{\lambda,2,\gaussienne; \cS}_{\le k}\right\}
\right] 
\le M\P[L(X) \neq \gaussienne] \le M\epsilon_{CLT}(d) \le \epsilon.
\]
Therefore, for any $z \in A_L(0,0)$,
\begin{align*}
& \P\left[\overline\tau^{\lambda,d,X;\{z\}}_k (A_L(1,0)) \ge m 
\text{ and } \overline\tau^{\lambda,d,X;\{z\}}_{\le k}(\R^d) \le M \right] \\
& = \P\left[\overline\tau^{\lambda,2,L(X);\{L(z)\}}_k (A(1,0)) \ge m 
\text{ and } \overline\tau^{\lambda,2,L(X);\{L(z)\}}_{\le k}(\R^2) \le M \right] \\
& \ge \P\left[\overline\tau^{\lambda,2,\gaussienne;\{L(z)\}}_k (A(1,0)) \ge m 
\text{ and } \overline\tau^{\lambda,2,\gaussienne;\{L(z)\}}_{\le k}(\R^2)\le M
\text{ and } \left\{\overline\tau^{\lambda,2,L(X); \{L(z)\}}_{\le k} = \overline\tau^{\lambda,2,\gaussienne; \{L(z)\}}_{\le k}\right\}
\right] \\
& \ge 
\P\left[\overline\tau^{\lambda,2,\gaussienne;\{L(z)\}}_k (A(1,0)) \ge m 
\text{ and } \overline\tau^{\lambda,2,\gaussienne;\{L(z)\}}_{\le k}(\R^2)\le M \right] - \epsilon \\
& \ge \survival(\lambda)-2\epsilon.
\end{align*}
This gives the first item. The second item is proven in exactly the same way.
\end{proof}

\subsubsection{Plan and intuition}
\label{s:plan}

\paragraph{Setup.} Let $\lambda>0$, $\epsilon>0$, $d \ge 1$ and $K \in \cK(d)$. 
Recall the definition of $\widehat C^0=\widehat C^0(\lambda,\delta_{1/2},d,K)$ in Section \ref{s:framework:t:penrose-convex}
and the notation $S(\lambda)$ for the a Poisson$(\lambda)$ offspring Galton-Watson process.
The aim is to prove that the inequality
\[
\P[\#\widehat C^0=\infty] \ge S(\lambda)-\epsilon
\]
holds for any $d$ large enough, uniformly in $K \in \cK(d)$.
Recall that $\widehat C^0$ can be built as the set of positions of a pruned version of the \BRW\ $\overline\tau^{\lambda, d, X_K}$
where $X_K$ denotes a random variable with uniform distribution on $X_K$.
Recall in particular the notion of interference defined in \eqref{e:interference}.
See Section \ref{s:brw} for notations on \BRW\ and Section \ref{s:cluster-brw-basic} for the construction of $\widehat C^0$ from 
$\overline\tau^{\lambda, d, X_K}$.
The basic idea is that, up to an event whose probability vanishes when $d$ tends to $\infty$, 
$\widehat C^0$ is infinite when $\overline\tau^{\lambda, d, X_K}$ is infinite.

\paragraph{The underlying Galton-Watson tree.}
The underlying Galton-Watson tree $\tau$ of the \BRW\ $\overline\tau^{\lambda, d, X_K}$ does not depend on $d$ nor on $K$.
It only depends on $\lambda$. 
This is a Galton-Watson process with Poisson$(\lambda)$ offspring starting from one particle.

\paragraph{Control of the interference up to a given generation.}
Consider the case where the root $\root$ has a child $x$ which itself has a child $y$.
In the construction of $\widehat C^0$ we have to reject $y$ because of interference with the root $\root$ if  
$V(y) \in V(\root)+K$\footnote{Actually $V(\root)=0$ but the argument is clearer if we keep writing $V(\root)$.} that is if 
\begin{equation}\label{e:boulette}
[V(x)-V(\root)] + [V(y)-V(x)]  \in K.
\end{equation}
Recall that $X_K$ and $X'_K$ are independent random variables with uniform distribution on $K$.
Condition to the tree $\tau$, the probability of \eqref{e:boulette} is $\P[X_K+X'_K \in K]$. 
By the rearrangement inequality (Theorem \ref{t:riesz}) this probability is at most $\P[X_B+X'_B \in B]$ where, as usual, $B$ is the Euclidean
ball of unit volume.
It is moreover easy to check that $\P[X_B+X'_B \in B]$ tends to $0$ when $d$ tends to infinity (see Lemma \ref{l:somme-unif-boule}).
Therefore $\P[X_K+X'_K \in K]$ tends to $0$ uniformly in $K$ as $d$ tends to infinity.
With these ideas it is quite easy to show that for any given generation $k \ge 1$, 
\[
\P[\text{no particle of }\overline\tau^{\lambda, d, X_K}_{\le k}\text{ is rejected because of interference}] \to 1 \text{ as } d \to \infty \text{ uniformly in } K.
\]

\paragraph{Control of the position of the particles.} 
For various reasons, we need to control the position of the particles.
Fix a linear map $L:\R^d\to\R^2$ given by Theorem \ref{t:Klartag-CLT} for $X_K$.
Recall that $\cN$ denotes a standard Gaussian random vector $\cN$ on $\R^2$.
The map $L$ (which depends on $X_K$ and thus on $d$) fulfills the following property.
The total variation distance between $L(X_K)$ and $\cN$ tends to $0$ when $d$ tends to $\infty$, uniformly in $K$.
Thanks to this property, we can control the value of the image by $L$ of the position of the particles, 
uniformly in $K$ when $d$ tends to $\infty$.
This will be sufficient for our purpose.

\paragraph{Comparison with a super-critical oriented two-dimension percolation process.}
The difficulty is to get a control over all generations of $\overline\tau^{\lambda,d,X_K}$.
Recall the definition of the sets $A_L(i,j)$, $(i,j) \in \overline\cL$ in Section \ref{s:embedding}.
Combining the previous ideas, we can prove the following results which are the basic steps of a renormalization scheme.
Here $m$ (number of seeds) and $k$ (number of generations) are to be suitably chosen.
See Lemma \ref{l:brique}.
The following results hold for $d$ large enough, uniformly in $K$.
\begin{itemize}
\item With a probability close to $S(\lambda)$, no particle of $\overline\tau_{\le k}^{\lambda,d,X_K}$ is rejected by interference 
(and thus all of them belong to the cluster $\widehat C_0$ in our construction) and $m$ particles of $\overline\tau_k^{\lambda,d,X_K}$ 
are located in $A_L(1,0)$. If this is the case, we say that stage $(0,0)$ is a success.
\item With a probability close to $1$, for any $(i,j) \in \cL$, if we start with a set of $m$ particles whose set of position is 
$\cS(i,j) \subset A_L(i,j)$,
then no particle of $\overline\tau_{\le k}^{\lambda,d,X_K ; \cS(i,j)}$ (a \BRW\ with initial set of particles located at $\cS(i,j)$) is rejected by interference 
(when considering only interference within $\overline\tau_{\le k}^{\lambda,d,X_K ; \cS(i,j)}$),
$m$ particles of $\overline\tau_k^{\lambda,d,X_K ; \cS(i,j)}$ belongs to $A_L(i+1,j+1)$ and 
$m$ particles of $\overline\tau_k^{\lambda,d,X_K ; \cS(i,j)}$ belongs to $A_L(i+1,j-1)$.
If this is the case, we say that stage $(i,j)$ is a success (this is not well defined for the moment as it depends on $\cS(i,j)$).
\end{itemize}
If stage $(0,0)$ is a success (which occurs with probability close to $S(\lambda)$), then we can use the position of $m$ particles of 
$\overline\tau_k^{\lambda,d,X_K}$ located in $A_L(1,0)$ 
(recall that all of them belongs to $\widehat C_0$ as none of them was rejected by interference) as a set of seeds $\cS(1,0)$ for stage $(1,0)$.
If stage $(1,0)$ is a success (which occurs with probability close to $1$) we can use the position of $m$ particles of 
$\overline\tau_k^{\lambda,d,X_K ; \cS(1,0)}$
located in $A_L(2,\pm 1)$ as a set of seeds $\cS(2,\pm 1)$ for stage $(2,\pm 1)$ and so on.
With the exception of stage $(0,0)$, we thus have a natural coupling with a super-critical oriented percolation process on $\cL$.

If stage $(0,0)$ is a success, if the oriented percolation process percolates and if there were no interference between \BRW\ of different stages,
then $\widehat C_0$ would be infinite and the proof would be over.
It remains to deal with interference between the \BRW\ of different stages.

\paragraph{Control of the interference between \BRW\ of different stages.}
This is actually the main difficulty of the proof.
Let us mention that in the actual proof we will handle interference between particles of a given \BRW\ and particles of different \BRW\ in a unified
way, based on the following ideas. 
We perform over-pruning (see Section \ref{s:cluster-brw-over}) to build a subset of $\widehat C^0$.
Concretely this means that, when exploring the \BRW, if a particle does not fulfill one of the required properties,
we reject the particle and its progeny.
This depends on the order in which we explore the \BRW, but this is not an issue for our purpose.

\begin{enumerate}
\item We fix a large $M$ and do not explore more than $M$ particles at each stage.
If $M$ is large enough, this does not modify significantly the probability of success at each stage.
Thus there is no drawbacks.
However, this gives a bound on the number of particles at each stage that can generate interference at a later stage.
\item We reject a particle and its progeny if it makes a step whose image by $L$ is too large.
More precisely, for a large $\Lambda$, if $y$ is a child of $x$, we reject $y$ and its progeny if $\|L(V(y)-V(x))\|_2 \ge \Lambda$.
As above, if $\Lambda$ is large enough, there is no drawbacks.
However, there are two advantages:
\begin{enumerate}
\item This reduces the interference region. Recall that $D$ denotes the unit disk of $\R^2$.
The interference of $x$ is $x + K \cap L^{-1}(\Lambda D)$ instead of $x+K$\footnote{Equivalently,
we could have worked from the beginning with the \BRW\ 
\[
\overline\tau^{\lambda\P[\|L(X_K)\|_2 \le \Lambda],d,X_K^\Lambda} 
\]
where $X_K^\Lambda$ is distributed as $X_K$ condition to $\|L(X_K)\|_2 \le \Lambda$.}.
\item This allow to localize the particles at each stage and then to identify the particles that can interfere at a later stage.
As the seeds are in $(i,j)+D$ and as we explore $k$ generations,
all the particles considered at stage $(i,j)$ belongs to $(i,j) + (1+\Lambda k) D$.
\end{enumerate}
\item We reject a particle $x$ and its progeny if we previously examined without rejecting 
a particle $x'$ such that the following condition does {\em not} hold:
\begin{equation}\label{e:planutilisationgood}
\|L(V(x)-V(x'))\|_2 \ge 2 \Lambda \text{ or } V(x)-V(x') \in G(d,K,\eta).
\end{equation}
Here $\eta>0$ is a suitable parameter and $G(d,K,\eta)$ is defined in Section \ref{s:good-gaps}.
By the localization properties, if $x$ is examined at stage $(i,j)$ and if $x'$ has been examined at stage $(i',j')$, then
$\|L(V(x)-V(x'))\|_2 \ge 2 \Lambda$ holds as soon as $(i,j)$ and $(i',j')$ are far enough from each other.
Moreover the number of particles examined at each stage is bounded.
Therefore, for a given $x$, we just have to check the property $V(x)-V(x') \in G(d,K,\eta)$ for a bounded number of particles $x'$.
By Lemma \ref{l:good_gap}, condition to everything but $V(x)-V(\parent x)$
\footnote{In the proof, we will indeed consider conditional probabilities. We therefore have to be careful with these aspects.}
(recall that $\parent x$ denotes the parent of $x$), this holds with high probability for all $d$ large enough, uniformly in $K$.
\end{enumerate}

Thanks to \eqref{e:planutilisationgood}, the probability of interference is small.
Let us explain this point.
Let $x, y, x'$ be three distinct particles where $y$ is a child of $x$.
We have examined and not rejected $x$ and $x'$.
We are examining $y$ and we already now that $\|L(V(x)-V(y)\|_2 \le \Lambda$ holds.
We wonder whether $y$ has to be rejected because of interference with $x'$.
As the interference region of $x'$ is $V(x')+K \cap L^{-1}(\Lambda D)$, we wonder whether $V(y)$ belongs to $V(x')+K \cap L^{-1}(\Lambda D)$.
We condition by everything but $V(y)-V(x)$, which is a random variable uniformly distributed on $K$.
\begin{itemize}
\item If $\|L(V(x)-V(x'))\|_2 \ge 2\Delta$, then $y$ can not be rejected because of interference with $x'$.
Indeed $y$ belongs to $V(x)+K \cap L^{-1}(\Lambda D)$ and therefore $y$ can not belong to $V(x')+K \cap L^{-1}(\Lambda D)$.
\item Otherwise, by \eqref{e:planutilisationgood}, we have $V(x)-V(x') \in G(d,K,\eta)$. 
The probability (because of our conditioning this is a probability on $V(y)-V(x)$) that $V(y)-V(x') = (V(x)-V(x'))+(V(y)-V(x))$ 
belong to $K$ is at most $\eta$ (by definition of $G(d,K,\eta)$).
Thus the probability that $y$ is rejected because of interference with $x'$ is at most $\eta$.
\end{itemize} 
Because of our control on localization and number of particles at each stage, this is sufficient.

\paragraph{A two-step approach to handle interference.}
Let us emphasize that, in addition to the control of the number of particles and the length of the projections of the steps,
the main ingredient is thus the following two-step approach:
\begin{enumerate}
\item First we ensure that all relevant relative positions are good (this is the content of \eqref{e:planutilisationgood}).
\item Then we use this control to control the interference (the set $G(d,K,\eta)$ is designed for this task).
\end{enumerate}

\paragraph{Comparison with the proof by Penrose in the Euclidean case.}
The plan of the proof is the same.
Thanks to results of analysis and high dimension geometry (see Section 2) most parts of the proof of Penrose
can actually be adapted to the non Euclidean case.
One of the main difference is due to the lack of isotropy in our setting.
In particular, the equivalent of the set of good gaps $G(d,K,\eta)$ in the Euclidean setting is simply  $\R^d \setminus \frac 3 4 B$.
In our setting we had to provide an alternative description of this set of good gaps.
Fortunately, an abstract definition was sufficient thanks to rearrangement inequalities (see Theorem \ref{t:riesz}).

\subsubsection{Construction of a subset of $\widehat C^0$ related to an oriented percolation on $\cL$}
\label{s:cluster-perco-orientee}

\paragraph{Parameters.} Fix $\lambda > 1$ and $\epsilon>0$. 
Fix $m,d_1,k,M \ge 1$ as provided by Lemma \ref{l:brique} for the parameters $\lambda$ and $\epsilon$.	
Fix $\troncature \ge 1$ such that $\P[\|\gaussienne\|_2 \ge \troncature ] \le \epsilon/M$.
Let $d_2$ be such that, for all $d \ge d_2$, $\epsilon_{CLT}(d) \le \epsilon/M$ where $\epsilon_{CLT}$ appears in Theorem \ref{t:Klartag-CLT}.
Fix $\eta>0$ such that
\begin{equation}\label{e:eta}
400k^2\troncature^2M^2 \eta \le \epsilon.
\end{equation}
Let $d_3$ be such that, for all $d \ge d_3$,
\begin{equation}\label{e:conditionepsilonG}
400k^2\troncature^2M^2 \eta^{-1}\epsilon_G(d)  \le \epsilon/3.
\end{equation}
Set $d_0=\max(d_1,d_2,d_3)$.

\paragraph{Setting and aim.}
Let $d \ge d_0$ and $K \in \cK(d)$.
Fix $L : \R^d \to \R^2$ given by Theorem \ref{t:Klartag-CLT} for $X_K$.
By definition of $\troncature$ and as $d \ge d_0 \ge d_2$, we get (under an appropriate coupling)
\begin{equation}\label{e:troncature}
\P[\|L(X_K)\|_2 \ge \troncature ] \le \P[L(X_K) \neq \gaussienne] + \P[\|\gaussienne\|_2 \ge \troncature ] \le  2\epsilon/M.
\end{equation}
%Set
%\[
%K = K \cap L^{-1}(\troncature D).
%\]
We aim at proving
\[
\P[\#\widehat C^0=\infty] \ge (\survival(\lambda)-5\epsilon)\theta(1-5\epsilon)
\]
where $\theta$ is defined in Section \ref{s:embedding}. 
Theorem \ref{t:penrose-convex-lower} will follow easily.
Thanks to the choice of parameters, the following properties hold.
\begin{itemize}
\item For all $z \in A_L(0,0)$, 
\begin{equation} \label{e:applilbrique1}
\P\left[\overline\tau^{\lambda,d,X_{K};\{z\}}_k (A_L(1,0)) \ge m 
\text{ and } \overline\tau^{\lambda,d,X_{K};\{z\}}_{ \le k}(\R^d) \le M \right] \ge \survival(\lambda)-\epsilon.
\end{equation}
\item For all $(i,j) \in \cL$ and all $\cS \subset A_L(i,j)$ with cardinality $m$,  
\begin{align} \label{e:applilbrique2}
& \P\left[\overline\tau^{\lambda,d,X_{K};\cS}_{k} (A_L(i+1,j+1)) \ge m 
\text{ and } \overline\tau^{\lambda,d,X_{K};\cS}_{k} (A_L(i+1,j-1)) \ge m 
\text{ and } \overline\tau^{\lambda,d,X_{K};\cS}_{\le k}(\R^d) \le M\right] \nonumber \\
& \ge 1-\epsilon.
\end{align}
\item  \eqref{e:eta}, \eqref{e:conditionepsilonG}, \eqref{e:troncature}.
\end{itemize}

\paragraph{Randomness and $\sigma$-fields.} 
Recall that $X_K$ stands for a random variable with uniform distribution on $K$.
Let
\[
\big( \overline\tau^{i,j,n} )_{(i,j,n) \in \overline\cL \times \{1,\dots,m\}}
\]
be a family of independent copies of $\overline\tau^{\lambda,d,X_{K}}$. 
Let $(\alpha^{i,j})_{(i,j) \in \overline\cL}$ be a family of i.i.d.\ Bernoulli random variables with parameter $1-\epsilon$.
For all $(i,j) \in \cL$ we denote by $\cF_{i,j}$ (resp. $\cF_{i,j}^-$) the $\sigma$-field generated by 
the $\overline\tau^{i',j',n}$ and the $\alpha^{i',j'}$ for $(i',j',n) \in \overline\cL \times \{1,\dots,m\}$ such that $(i',j')$ is smaller 
(resp. strictly smaller) than $(i,j)$
for the lexicographic order.

We will not formalize it but we will merge appropriately pruned versions of the \BRW\ $\overline\tau^{i,j,n}$ to get a unique \BRW\ starting
from one particle located at $0$. This latter \BRW\ is the \BRW\ explored in order to build a subset of $\widehat C^0$.

\paragraph{Further notations and remarks.}
At the beginning, the site $(0,0)$ is active and each site $(i,j) \in \cL$ is inactive. 
Moreover, with $(0,0)$ is associated the singleton $\cS(0,0)=\{0_{\R^d}\} \subset A_L(0,0)$.
We then enumerate the sites $(i,j)$ of $\overline{\cL}$ by lexicographic order.
Some sites $(i,j) \in \cL$ will be activated.
With each active site $(i,j)$ will be associated a subset $\cS(i,j) \subset A_L(i,j)$ of cardinality $m$.
This set of $m$ points will always be, in a coupling with the Boolean model, a subset of $\widehat C^0$.
If at the end of the construction there exists in the graph $\overline\cL$ an infinite path of active sites, 
then $\widehat C^0$ is infinite and percolation occurs in the Boolean model.

To simplify some arguments, we will also associate with each site of $\overline\cL$ a state: open or closed.
It will be done in such that a way that if $\pi$ is an infinite open path in $\overline\cL$ from $(0,0)$,
then $\pi$ only contains active sites and therefore percolation occurs in the Boolean model.

We will use the over-pruning algorithm described in Section \ref{s:cluster-brw-over}.
We will often use over-pruning implicitly by rejecting more particles than necessary
and by not considering the children of some particles (which amount to define their interference region as the empty set).

Set 
\[
G=G(d,K,\eta)
\]
where $G(d,K,\eta)$ is introduced in Section \ref{s:good-gaps}.

\paragraph{Site $(0,0)$.}  We consider the \BRW\ $\overline\tau=\overline\tau^{0,0,1}$.
\begin{enumerate}
\item We examine successively the particles of generation at most $k$ of this \BRW\ in any admissible order.
The last requirement means that:
\begin{itemize}
\item Children are examined after their parents.
\item Children of a given parent are examined in a row: once we start examining one of them, we then examine all the children of this parent.
\end{itemize}
We stop as soon as we have examined all the particles 
or as soon as the particle $x$ under examination fulfills one of the following conditions:
\begin{enumerate}
\item Overpopulation. The particle is the $M$-th particle examined.
\item Bad gap. This occurs in any of the following conditions.
\begin{itemize}
\item $x$ is not the root and $V(x) \not\in V(\parent x)+K \cap L^{-1}(\troncature D)$.
This ensures that the interference region of $\parent x$ is indeed $V(\parent x)+K \cap L^{-1}(\troncature D)$ 
(see Section \ref{s:cluster-brw-over} and the remarks in the paragraph above).
\item There exists a particle $y$ examined strictly before $x$ such that $V(x) \not\in V(y)+ G$.
\end{itemize}
\item Interference.
%\footnote{This is actually a second kind of over-pruning. Indeed we do reject only because of interference,
%otherwise we would have to be more restrictive on $y$ (this should not be a sibling of $x$).}. 
There exists a particle $y$ examined strictly before $x$ which is not the parent $\parent{x}$
nor a sibling of $x$
and which is such that $V(x) \in V(y)+K \cap L^{-1}(\troncature D)$
\footnote{With the notations of Section \ref{s:cluster-brw-over}, we want to reject any particle which belongs to
\[
\bigcup_{y \in B} V(y) + K \cap L^{-1}(\troncature D).
\]
Let $B'$ be the set of particles examined strictly before $x$ and which are neither $\parent x$ nor a sibling of $x$.
As $B \subset B'$, we are performing over-pruning at this step.}.
\end{enumerate}
One defines the following two sets.
\begin{itemize}
\item The set $\cG(0,0)$ ($\cG$ stands for generated) of all the particles examined
except the last one if it caused Overpopulation or Bad Gap or Interference.
By some abuse of notation, we will sometimes see $\cG(0,0)$ (and other similar sets) as the set of positions of the particles.
We will call them  the particles {\em generated} at stage $(0,0)$. 
Note that each particle whose children have been examined is a generated particle
(this is due to the fact that we examine the particles in an admissible order).
This is a key remark when considering interference later in the construction.
%We include in $\cG(-1,0)$  the point $0$, the position of the root.
\begin{itemize}
\item In the coupling with the Boolean model, all the points of $\cG(0,0)$ (seen as the set of positions of the particles) belong to $\widehat C^0$.
Indeed, none of them caused a stop by Interference.
\item All the gaps between any two distinct particles of $\cG(0,0)$ are good:
\[
\forall x,y \in \cG(0,0), x \neq y \implies V(y) \in V(y) + G.
\]
Indeed, none of them caused a stop by Bad gap.
\item 
\begin{equation}\label{e:besoindelabel0}
\text{The image by }L\text{ of }\cG(0,0)\text{ is included in }k\troncature D.
\end{equation}
Here we see $\cG(0,0)$ as the set of positions of the particles.
This is due to the fact that the position of the root is $0$, the fact that the image by $L$ of each step belongs to $\troncature D$
and the fact that we did not explore the \BRW\ beyond generation $k$.
\item $\cG(0,0)$ contains at most $M$ points.
\end{itemize}
\item The set $\cG_k(0,0) \subset \cG(0,0)$ of articles of generation $k$ examined
except the last one if it caused Overpopulation or Bad Gap or Interference.
\end{itemize}
\item If
\begin{equation}\label{e:okpasinitial}
\#\big(\cG_k(0,0) \cap A_L(1,0)\big) \ge m
\end{equation}
we say that the site $(0,0)$ open, that the site $(1,0)$ is active and we define $\cS(1,0)$, where $\cS$ stands for seeds,
as the $m$ first points of 
\[
\cG_k(0,0) \cap A_L(1,0)
\]
in Neveu ordering (see Section \ref{s:brw}).
In the coupling with the Boolean model, all the points of $\cS(1,0)$ belong to $\widehat C^0$.
This holds because $\cS(1,0) \subset \cG_k(0,0) \subset \cG(0,0)$ and all points of $\cG(0,0)$ belongs to  $\widehat C^0$.
If \eqref{e:okpasinitial} does not hold, we say that the site $(0,0)$ is closed and $(1,0)$ remains inactive.
\end{enumerate}

\paragraph{Stage $(i,j)$.} Recall that we now consider successively each $(i,j)\in\cL$ by lexicographic order.
If $(i,j)$ is inactive, then we decide as follows: it is open if $\alpha^{i,j}=1$ ; it is closed otherwise.
Therefore, in this case, it is open independently of everything else with probability $1-\epsilon$.

Thereafter, we consider the case where $(i,j)$ is active.
The set $\cS(i,j)$ is well defined. 
It's a subset of cardinal $m$ of $A_L(i,j)$.
List the point of $\cS(i,j)$ in an arbitrary order: $\cS(i,j)=\{x^1,\dots,x^m\}$.
We consider the $m$ \BRW\ $\overline\tau^n=x^n+\overline\tau^{i,j,n}$ 
where $x^n+\overline\tau^{i,j,n}$ designates the \BRW\ $\overline\tau^{i,j,n}$  in which $x^n$ was added to the position of all the particles.
We gather these $m$ \BRW\ into a single \BRW\ originating from $\cS(i,j)$.
We denote it by $\overline\tau^{\cS(i,j)}$.
\begin{enumerate}
\item We examine successively the particles of generation between $1$ and $k$ of $\overline\tau^{\cS(i,j)}$ in any admissible order
(see Stage $(0,0)$).
%We consider successively the particles of generation $1$ to $k$ of $\overline\tau^{\cS(i,j)}$ in the following order:
%we go through the particles of generation $1$ to $k$ of $x^1+\overline\tau^1$ in any order that respects the generation,
%We then go through the particles of generation $1$ to $k$ of $x^1+\overline\tau^2$ in any order that respects the generation and so on.
In particular, we never examine the roots of this \BRW\ 
(note that the positions of the roots are also positions of particles of \BRW\ examined during one of the previous stages).
We stop as soon as we have examined all the particles or as soon as the particle $x$ under examination fulfills one of the following conditions:
\begin{enumerate}
\item Overpopulation. The particle is the $M$-th particle examined during this stage $(i,j)$.
\item Bad gap. One of the following conditions occurs.
\begin{itemize}
\item $V(x) \not\in V(\parent x)+K \cap L^{-1}(\troncature D)$.
This ensure that the interference region of $\parent x$ is indeed $V(\parent x)+K \cap L^{-1}(\troncature D)$ 
(see stage $(0,0)$ for more explanations).
\item There exists a particle $y$ examined strictly before $x$ during this stage $(i,j)$
or generated during one of the previous stages 
(its position then belongs to $\cG(i',j')$ for some $(i',j') < (i,j)$ for the lexicographic order) such that:
\[
\|L(V(x))-L(V(y))\|_2 \le 2\troncature \text{ and } V(x) \not\in V(y)+G.
\]
\end{itemize}
\item Interference.  There exists a particle $y$ examined strictly before $x$ during this stage $(i,j)$
or generated during one of the previous stages 
(its position then belongs to $\cG(i',j')$ for some $(i',j') < (i,j)$ for the lexicographic order) such that:
this is not the parent $\parent{x}$
nor a sibling of $x$
and we have:
\[
V(x) \in V(y)+K \cap L^{-1}(\troncature D). 
\]
\end{enumerate}

One defines the following two sets.

\begin{itemize}
\item The set $\cG(i,j)$ of all particles examined except the last one it it caused
Overpopulation or Bad Gap or Interference.
We will call them  the particles {\em generated} at stage $(i,j)$. 
As before, each particle whose children have been examined is a generated particle.
\begin{itemize}
\item In the coupling with the Boolean model, all the points of $\cG(i,j)$ belong to $\widehat C^0$.
Indeed, none of them caused a stop by Interference.
\item All the gap between two distinct points $x$ and $y$ of
\[
\bigcup_{(i',j') \le (i,j)} \cG(i',j')
\]
satisfies
\begin{equation}\label{e:hoho}
\|L(V(x))-L(V(y))\|_2 > 2\troncature \text{ or } V(x) \in V(y)+G.
\end{equation}
Indeed, none of them caused a stop by Bad Gap.
\item 
\begin{equation}\label{e:besoindelabel}
\text{The image by }L\text{ of }\cG(i,j)\text{ is included in }(i,j)+2k\troncature D.
\end{equation} 
This is due to the fact that the positions of the roots belong to $(i,j)+D$,
the fact that the image by $L$ of each step belongs to $\troncature D$, 
the fact that we only explored the first $k$ generations
and the fact that $1+k\troncature \le 2k\troncature$.
\item $\cG(i,j)$ containts at most $M$ points.
\end{itemize}
\item The set $\cG_k(i,j) \subset \cG(i,j)$ of 	particles of generation $k$ examined, 
except the last one if it caused a stop by Overpopulation or Bad Gap or Interference.
\end{itemize}
\item If
\begin{equation} \label{e:okpassuivant}
\#\big(\cG_k(i,j) \cap A_L(i+1,j+1)\big) \ge m \text{ and } \#\big(\cG_k(i,j) \cap A_L(i+1,j-1)\big) \ge m
\end{equation}
then:
\begin{itemize}
\item We say that the site $(i,j)$ is open.
\item If the site $(i+1,j+1)$ is inactive then we say that it is henceforth active and we define $\cS(i+1,j+1)$ as the $m$ first point
by Neveu ordering (see Section \ref{s:brw}) of $\cG_k(i,j) \cap A_L(i+1,j+1)$.
In the coupling with the Boolean model, all the points of $\cS(i+1,j+1)$ belong to $\widehat C^0$.
\item If the site $(i+1,j-1)$ is inactive then we say that it is henceforth active and we define $\cS(i+1,j-1)$ as the $m$ first point
by Neveu ordering of $\cG_k(i,j) \cap A_L(i+1,j-1)$.
In the coupling with the Boolean model, all the points of $\cS(i+1,j-1)$ belong to $\widehat C^0$.
\end{itemize}
Otherwise, we say that the site $(i,j)$ is closed.
\end{enumerate}

\subsubsection{Bounds on conditional probabilities}
\label{s:proba-briques}

The aim of this section is to prove the following lemmas.
Parameters have been fixed in Section \ref{s:cluster-perco-orientee}.

\begin{lemma} \label{l:mesurabilite}
For all $(i,j) \in \overline{\cL}$, 
$
\{ (i,j) \text{ active}\} \in \cF(i,j)^- \text{ and } \{ (i,j) \text{ open}\}\in \cF(i,j).
$
\end{lemma}

\begin{lemma} \label{l:comparaisoninitiale}
We have
$
\P[(0,0) \text{ open}] \ge \survival(\lambda)-5\epsilon.
$
\end{lemma}

\begin{lemma} \label{l:comparaisonsuite}
For all $(i,j) \in \cL$,
$
\P[(i,j) \text{ open }| \cF^-(i,j)] \ge 1-5\epsilon.
$
\end{lemma}

\begin{proof}[Proof of Lemma \ref{l:mesurabilite}] This is straightforward by construction.
\end{proof}

\begin{proof}[Proof of Lemma \ref{l:comparaisoninitiale}]  
We have 
\begin{equation}\label{e:casinitialbla1}
\{(0,0) \text{ closed} \} \setminus (\Overpopulation \cup \bad \cup \rejection ) \subset \autre
\end{equation}
where
\begin{align*}
\Overpopulation & = \{\text{the enumeration stops because of Overpopulation}\},\\
\bad & = \{\text{the enumeration stops because of Bad gap and not Overpopulation}\}, \\
\rejection & = \{\text{the enumeration stops because of Interference and not Overpopulation or Bad gap}\}, \\
\autre & = \{ \overline\tau_k(A_L(1,0)) < m\}
\end{align*}
and where $\overline\tau$ is the \BRW\ used in Stage $(0,0)$.
Indeed, if the event on the left-hand side of \eqref{e:casinitialbla1} occurs, then \eqref{e:okpasinitial} does not hold and 
$
\cG_k(0,0) = \{V(x), x \in \overline\tau_k\}
$
hence
$
\overline\tau_k(A_L(1,0))<m.
$
As a consequence
\begin{equation}\label{e:casinitialbla2}
\{(0,0) \text{ closed } \} \subset \bad \cup \rejection \cup \Overpopulation \cup \autre.
\end{equation}

Let us give an upper bound for the probability of $\bad$.
We have
$
\bad \subset \bad_1 \cup \bad_2
$
with 
\[
\bad_1 = \bigcup_x \{V(x) \not\in V(\parent{x})+K \cap L^{-1}(\troncature D)\} 
\text{ and } \bad_2= \bigcup_{x \neq y} \{V(x) \not\in V(y) + G\}  \setminus \bad_1
\]
where for $\bad_1$ the union is over all $x$ in the $M$ first particles (otherwise this is the event Overpopulation which occurs)
of the tree except the root  and where for $\bad_2$ the union of over the same set of $x$ and over all $y$ strictly preceding $x$ 
(for the order used for the examination).

The event $\bad_2$ is then the union of at most $M^2$ events.
Rewriting this events, we get
\[
\bad_2 \subset \bigcup_{x \neq y} \{V(\parent{x})-V(y) + V(x)-V(\parent{x}) \not\in  G\}.
\]
But, condition to the underlying Galton-Watson tree,
$V(x)-V(\parent{x})$ is independent of $(V(\parent{x}),V(y))$ (actually it is independent of the positions of all previously examined particles)
and then of $V(\parent{x})-V(y)$.
Moreover $V(x)-V(\parent{x})$ has the same distribution as $X_{K}$.
Therefore, conditioning by the underlying tree $\tau$ we have, for all $x,y$ as above,
\[
\P[V(\parent{x})-V(y) + V(x)-V(\parent{x}) \not\in  G | \tau] \le \sup_{z \in \R^d} \P[z+X_K \not\in G] \le \eta^{-1}\epsilon_G(d)
\]
where $\epsilon_G(d)$ is defined in Lemma \ref{l:good_gap}.
We then get
\begin{align*}
\P[\bad_2] \le M^2 \eta^{-1} \epsilon_G(d).
\end{align*}
Hence
$
\P[\bad_2] \le \epsilon
$
by \eqref{e:conditionepsilonG}. (We also use the fact that several of our constants are greater than $1$.)

The event $\bad_1$ is the union of at most $M$ event. 
Arguing as above and using \eqref{e:troncature}, we get $\P[\bad_1] \le M 2\epsilon/M = 2\epsilon$. Thus
\[
\P[\bad] \le 3\epsilon.
\]

Let us now provide an upper bound for the probability of the event $\rejection$.
Note $x$ and $y$ the two particles which cause the event $\rejection$.
In particular:
\begin{itemize}
\item $x$ is one of the first $M$ particles (otherwise this is the event Overpopulation which occurs) of the tree except the root 
\item $y$ is a particle of $\overline\tau$ strictly preceding $x$ (for the order of enumeration).
\item $V(x) \in V(y) + K$.
\item $V(\parent{x}) \in V(y) + G$ because otherwise $\bad$ would occur instead of $\rejection$ (recall also $G=-G$).
\end{itemize} 
Thus,
\begin{align*}
\rejection 
 & \subset \bigcup_{x,y} \{V(x) \in V(y) + K\} \cap \{V(\parent{x}) \in V(y)+G\} \\
 & \subset \bigcup_{x,y} \{V(\parent{x}) + (V(x)-V(\parent{x}))	 \in V(y) + K\} \cap \{V(\parent{x}) \in V(y)+G\} 
\end{align*}
where the union is over all $x$ in the first $M$ particles of $\overline\tau$ except the root
and $y$ in the particle of $\overline\tau$ strictly preceding $x$.
The right-hand side is then the union of at most $M^2$ events.
By definition of $G=G(d,k,\eta)$ and using a reasoning similar to the one used for the event Bad, we then deduce
\begin{align*}
\P[\rejection]
& \le M^2 \eta  \le \epsilon
\end{align*}
by \eqref{e:eta}.
Finally, by \eqref{e:applilbrique1}, we have
\[
\P[\Overpopulation \cup \autre] \le 1-\survival(\lambda)+\epsilon.
\]
As a consequence,
$
\P[(0,0) \text{ open}] \ge \survival(\lambda)-5\epsilon.
$
\end{proof}

\begin{proof}[Proof of Lemma \ref{l:comparaisonsuite}] Let $(i,j) \in \cL$.
The event $\{(i,j) \text{ is active}\}$ is $\cF^-(i,j)$ measurable.
Therefore, we have to show
\[
\P[(i,j) \text{ open }| \cF^-(i,j)] \ge 1-\epsilon \text{ on the event } \{(i,j) \text{ is active}\}
\]
and
\[
\P[(i,j) \text{ open }| \cF^-(i,j)] \ge 1-\epsilon \text{ on the event } \{(i,j) \text{ is inactive}\}.
\]
The second property is straightforward. Indeed, when $(i,j)$ is inactive, 
$(i,j)$ has been defined as open independently of everything else with probability $1-\epsilon$.
Let us prove the first property.

{\em We now work on the event $\{(i,j) \text{ is active}\}$.}

On the event $\{(i,j) \text{ is active}\}$ there is a well defined set $\cS(i,j)$ with cardinality $m$ 
whose points are the starting points of $m$ \BRW.
This set is measurable with respect to $\cF^-(i,j)$.
We also have the \BRW\ $\overline\tau^{\cS(i,j)}$ which has been used in Stage $(i,j)$.
We have, as in the proof of Lemma \ref{l:comparaisoninitiale},
\[
\{(i,j) \text{ is closed } \} \cap \{(i,j) \text{ is active}\} \subset \bad \cup \rejection \cup \Overpopulation \cup \autre
\]
where the events $\bad$, $\rejection$ and $\Overpopulation$ are defined as in the proof of Lemma \ref{l:comparaisoninitiale} and where
\[
\autre = \{ \overline\tau_k^{\cS(i,j)}(A_L(i+1,j+1)) < m \text{ or } \overline\tau_k^{\cS(i,j)}(A_L(i+1,j-1)) < m\}.
\]

Let us provide an upper bound for the conditional probability of the event $\bad$.
The proof is similar to the one in Lemma \ref{l:comparaisoninitiale}.
The proof is identical for $\bad_1$.
For $\bad_2$, the combinatorial factor is not $M^2$ anymore.
There are at most $M$ choices for $x$.
Let us fix some $x$.
There are at most $M$ choices for $y$ if $y$ has been examined during Stage $(i,j)$.
If $y \in \cG(i',j')$ for $(i',j') < (i,j')$, then $L(V(y)) \in (i',j')+2k\troncature D$ (see \eqref{e:besoindelabel0} and \eqref{e:besoindelabel}).
Similarly, $L(V(x)) \in (i,j)+2k\troncature D$ (otherwise $\bad_1$ would occured and not $\bad_2$).
As a consequence, the condition $\|L(V(x)) - L(V(y)) \|_2 \le 2\troncature$ yields 
\[
\|(i,j)-(i',j')\|_2  \le (4k+2)\troncature \le 6k\troncature.
\]
Therefore there are at most $(12k\troncature+1)^2$ choice for $(i',j')$ and therefore at most 
\[
(12k\troncature+1)^2M+M \le 200 k^2\troncature^2 M
\]
choices for $y$.
Finally, the number of choices for $(x,y)$  is at most $200 k^2\troncature^2 M^2$.
As a consequence, arguing as in Lemma \ref{l:comparaisoninitiale},
\begin{align*}
\P[\bad | \cF^-(i,j)]
 & \le 200k^2\troncature^2M^2 \sup_{z \in \R^d} \P[z + X_{K} \not\in G ] \\
 & \le 200k^2\troncature^2M^2 \eta^{-1}\epsilon_G(d).
\end{align*}
where $\epsilon_G(d)$ is defined in Lemma \ref{l:good_gap}.
Therefore,
\[
\P[\bad| \cF^-(i,j)] \le \P[\bad_1| \cF^-(i,j)]+\P[\bad_2| \cF^-(i,j)] \le 3\epsilon
\]
by \eqref{e:conditionepsilonG}.

Let us now give an upper bound for the probability of the event $\rejection$.
Note $x$ and $y$ the two particles which cause the event $\rejection$.
In particular,
\begin{itemize}
\item $x$ is one of the first $M$ particles (otherwise the event Overpopulation would occur) of $\overline\tau$ 
not belonging to generation $0$.
\item $y$ is a particle examined strictly before $x$ during this stage $(i,j)$
or generated during one of the previous stage (its position belongs in this case to $\cG(i',j')$ for some $(i',j') < (i,j)$ 
for the lexicographical order).
\item $y$ is not $\parent{x}$.
\item $V(x) \in V(y)+K \cap L^{-1}(\Lambda D)$. This property yields $\|L(V(x))-L(V(y))\|_2 \le \troncature$.
By an argument already used to give an upper bound for the event \bad, we get that there are at most $400k^2\troncature^2M^2$ choices for $(x,y)$.
By $\|L(V(x))-L(V(y))\|_2 \le \troncature$ we also get
\begin{equation}\label{e:haha}
\|L(V(\parent{x}))-L(V(y))\|_2 \le 2\troncature.
\end{equation}
Indeed, as $\bad$ does not occur, $\|L(V(x))-L(V(\parent x))\|_2 \le \troncature$. 
\item
 $V(\parent{x}) \in V(y) + G$. 
Indeed, we have
\begin{equation}\label{e:hehe}
V(\parent{x}), V(y) \in \bigcup_{(i',j') \le (i,j)} \cG(i',j').
\end{equation}
This is clear for $y$ if $y$ has been generated during some Stage $(i',j')<(i,j)$.
This is also true if $y$ has been examined during the current stage.
Indeed, $y$ has been examined before $x$.
Therefore $y$ did not stopped the enumeration of the particles and therefore $V(y) \in \cG(i,j)$.
With the same argument we check that this is also true for $V(\parent{x})$.
The required result is then a consequence of \eqref{e:hoho}, \eqref{e:haha}, \eqref{e:hehe} and of $\parent{x} \neq y$.
\end{itemize}
We then have
\begin{equation}\label{e:rejectionunion}
\rejection \subset 
\bigcup_{x,y} \{V(x) \in V(y) + K\} \cap \{V(\parent{x}) \in V(y)+G\} 
\end{equation}
where, in addition to the properties describe by the two events, $x$ and $y$ are as above.
In particular, as already mentioned, there are at most $400k^2\troncature^2M^2$ choices for $(x,y)$.
By definition of good gaps we then get, with the same arguments as before (we only use the randomness on $V(x)-V(\parent{x})$),
\begin{align*}
\P[\rejection| \cF^-(i,j)]
& \le 400k^2\troncature^2M^2 \eta \le \epsilon
\end{align*}
by \eqref{e:eta}.

Finally, by \eqref{e:applilbrique2}, we get
\[
\P[\Overpopulation \cup \autre| \cF^-(i,j)] \le \epsilon.
\]
As a consequence, (we are still working on the $\cF_{i,j}^{-}$ measurable event $\{(i,j) \text{ is active}\}$),
\[
\P[(i,j) \text{ open} | \cF^-(i,j)] \ge 1-5\epsilon.
\]
\end{proof}

\subsubsection{Proof of Theorem \ref{t:penrose-convex-lower}}
\label{s:preuve-t-penrose-convex-lower}

In Section \ref{s:cluster-perco-orientee} we fixed $\lambda>1$ and $\epsilon>0$.
We then got some integer $d_0$ and several other parameters satisfying various properties.
We then let $d \ge d_0$ and $K \in \cK(d)$ and built some process in Section \ref{s:cluster-perco-orientee}.
We studied some properties of this process in Section \ref{s:proba-briques}.
We now conclude the proof of  Theorem \ref{t:penrose-convex-lower}.

Let $(i,j), (i',j') \in \overline\cL$ be such that $(i,j) \to (i',j')$, that is there is an arrow from $(i,j)$ to $(i',j')$ in the graph $\overline\cL$.
If $(i,j)$ is active and open, then $(i',j')$ is active and $\cS(i',j')$ is a well defined subset of $\widehat C^0$.

Recall that $(0,0)$ is active. 
If there exists an infinite path $\pi$ in $\overline\cL$ originating from $(0,0)$ and containing only open sites,
by the previous discussion, we get that $\widehat C^0$ is infinite.
Therefore
\[
\P[\#\widehat C^0=\infty] \ge \P[\text{there exists an infinite open path from }(0,0)].
\]
By Lemmas \ref{l:mesurabilite}, \ref{l:comparaisoninitiale} and \ref{l:comparaisonsuite} we get
\[
\P[\text{there exists an infinite open path from }(0,0)] \ge (\survival(\lambda)-5\epsilon)\theta(1-5\epsilon)
\]
where $\theta(1-5\epsilon)$ is the probability that there exists an infinite open path originating from $(1,0)$ in a Bernoulli site percolation
on $\cL$ with parameter $1-5\epsilon$.
Therefore
\[
\P[\#\widehat C^0=\infty] \ge (\survival(\lambda)-5\epsilon)\theta(1-5\epsilon).
\]
But $\theta(1-5\epsilon)$ tends to $1$ as $\epsilon$ tends to $0$ (this is stated as \eqref{e:percolimite}).
This proves Theorem \ref{t:penrose-convex-lower} in the case $\lambda>1$.
When $\lambda \le 1$, $S(\lambda)=0$ and the required result is straightforward.

\section{Proof of Theorem \ref{t:GM-convex}}
\label{s:preuve-GM-convexe}

\subsection{Framework and the clusters $\widehat C^0$ and $\widehat A^0$}
\label{s:framework-GM}

\paragraph{Framework.} Let $\beta>0, \rho >1, d \ge 1$ and $K \in \cK(d)$.
Set
\[
\lambda = \beta \kappa_c(\rho)^d.
\]
Let $\chi_1$ and $\chi_\rho$ be independent Poisson point processus on $\R^d$ with intensity $\lambda \d x$ and $\lambda \rho^{-d} \d x$.
Set $r(c)=1/2$ for $c \in \chi_1$ and $r(c)=\rho/2$ for $c \in \chi_\rho$.
Write
$
\chi = \chi_1 \cup \chi_\rho.
$
The process $\{ (c,r(c)), c \in \chi\}$ is a Poisson point process on $\R^d \times (0,+\infty)$ with intensity measure
$\d x \otimes\lambda  \nu_{d,\rho}(\d r)$.
%Set
%\[
%\Sigma 
%= \bigcup_{c \in \chi_1} \left(c+\frac 1 2 K\right) \cup  \bigcup_{c \in \chi_\rho} \left(c+ \frac\rho 2 K\right).
%\]
%This is a realization of $\Sigma(\lambda,\nu_{d,\rho},d,K)$.

%Set
%\[
%\Sigma^0 = \Sigma \cup \frac \rho 2 K
%\]
%and write
Write
$
\chi^0 = \chi \cup \{0\}
$
and set $r(0)=\rho/2$.
%Let $C^0$ be the cluster of $\Sigma^0$ that contains the origin.
%Denote by
%\[
%\widehat C^0 = \widehat C^0\big(\beta, \rho, d, K\big) = C^0 \cap (\chi_1 \cup \chi_\rho \cup \{0\}).
%\]
%In other words, $\widehat C^0$ is the set of centers of grains that are connected to the origin in $\Sigma^0$.
%A point of $\widehat C^0$ is called a $1$-point if it belongs to $\chi_1$ and a $\rho$-point otherwise.
%In particular $0$ is a $\rho$-point.

\paragraph{The cluster $\widehat C^0$.} As in the constant radius case, we can define a relevant undirected graph structure on $\chi^0$ as follows.
We put an edge between $x$ and $y$ if the associated grains $x+r(x)K$ and $y+r(y)K$ touch each other, that is if 
%If $x$ is a $r$-point and $x'$ is a $r'$-point ($r,r' \in \{1,\rho\})$, then we put and edge between $x$ and $x'$ if 
%\[
%x+\frac r 2 K \text{ touches } x'+\frac {r'} 2 K
%\]
%that is if 
\[
y \in x + \big(r(y)+r(x)\big) K.
\]
As in the constant radius case, $\widehat C^0=\widehat C^0(\beta,\rho,d,K)$ (we simplify the notations for the parameters) 
is the connected component of the graph $\chi^0$ which contains $0$.
We are interested in the probability that $\widehat C^0$ is unbounded (see \eqref{e:ahouicetruc}) and in
\[
\beta_c(\rho,d,K) = \inf \left\{\beta>0: \P\big[ \widehat C^0 \text{ is unbounded}\big]>0\right\}.
\]

\paragraph{The cluster $\widehat A^0$.} We define a new undirected graph on $\chi^0$. 
In this new graph, we put an edge between points $x$ and $y$ if $r(x) \neq r(y)$ and if the associated grains touch each other.
In other words, we put an edge  between points $x$ and $y$ if\footnote{Note that one of the two radii equals $1/2$ and the other equals $\rho/2$,
therefore the sum is always $(\rho+1)/2$.}
\[
r(x) \neq r(y) \text{ and } y \in  x + \frac{1+\rho}2 K.
\]
We denote by $\widehat A^0=\widehat A^0(\beta,\rho,d,K)$ the connected component of $0$ in this new graph.
Clearly,
\[
\widehat A^0 \subset \widehat C^0.
\]

\subsection{Two related \BRW}
\label{s:cluster-BRW-GM}

We will not formalize the algorithms of exploration of $\widehat C^0$ and $\widehat A^0$.
This can be done as in Section \ref{s:cluster-brw} to which we refer for more details.
We will give directly the relation with two \BRW\ and provide some intuition. 
Let $\beta>0,  \rho >1, d \ge 1$ and $K \in \cK(d)$.
 
Recall that $\widehat A^0$ is a subset of $\widehat C^0$ and that we are interested in whether $\widehat C^0$ is infinite or not.
Therefore, instead if investigating $\widehat C^0$ or $\widehat A^0$, we can investigate
\[
\frac{2}{1+\rho} \widehat C^0 \text{ or } \frac{2}{1+\rho} \widehat A^0.
\]
More specifically, the \BRW\ we will study will provide subsets of $\frac{2}{1+\rho} \widehat C^0$ or  $\frac{2}{1+\rho} \widehat A^0$.
This is a very tiny change but it will simplify formulas by removing numerous $\frac{1+\rho}2$ factors.

\subsubsection{A \BRW\ for the upper bound on percolation probability.}
\label{s:cluster-BRW-GM-M}

\paragraph{The \BRW.}
Consider for example a deterministic $1$-grain 
\[
x + \frac 1 2 K.
\]
The random $\rho$-grains which touches the previous grain are the set of grains centered at points of
\[
\left(x + \frac {1 + \rho}2 K\right) \cap \chi_\rho.
\]
Therefore, this is a Poisson random variable with parameter 
\[
\beta \kappa_c(\rho)^d  \rho^{-d} \left|\frac{1+\rho}2 K\right| = \beta \left(\frac{1}{\sqrt\rho}\right)^d
\]
as $|K|=1$ and by definition of $\kappa_c(\rho)$ (see \eqref{e:kappa_c}).
Similar results hold for the other cases.
We can therefore see that $\widehat C^0$ is set of points of some pruned two-type \BRW.
We refer to Section \ref{s:cluster-brw} for details on the pruning but mention one important difference below.
It is a two-type \BRW.
There are $1$-particles and $\rho$-particles.
It starts with one $\rho$-particle located at $0$.
The progeny are independent (between the two types) and Poisson distributed.
The matrix of mean is the matrix $M(\beta,\rho,d)$ defined in \eqref{e:M}.
The steps of the \BRW\ (which we will actually not use) are i.i.d.\ with uniform distribution 
on $K$ for type $1$ to type $1$ progeny, 
on $\rho K$  for type $\rho$ to type $\rho$ progeny, 
on $\frac{\rho +1 }2 K$ otherwise.

Let us conclude with the important difference in the pruning process. 
Recall the vocabulary introduced in \eqref{e:interference}.
Here we are revealing points of two independent Poisson point processes: $\chi_1$ and $\chi_\rho$.
Therefore there can only be interference 
between two $1$-particles (when we are revealing $\chi_\rho$) or 
between two $\rho$-particles (when are revealing $\chi_1$).
In other words, a child $y$ of $\rho$ particle $x$ can only be rejected because of interference with another $\rho$-particle $x'$
and a child $y$ of $1$ particle $x$ can only be rejected because of interference with another $1$-particle $x'$.

\paragraph{The plan.} The upper bound for the percolation probability 
will follow from a simple analysis of the survival probability of the underlying Galton-Watson process.

\subsubsection{A \BRW\ for the lower bound on percolation probability.}
\label{s:cluster-BRW-GM-M'}

\paragraph{The \BRW.}
Recall
\[
\frac 2{1+\rho}\widehat A^0\subset \frac 2{1+\rho}\widehat C^0.
\]
We will prove a lower bound on the probability that the set of the left is infinite.
This will give a lower bound on the probability that the set of the right is infinite.

The set $\frac 2{1+\rho}\widehat A^0$ is the set of positions of a pruned \BRW.
We refer to Section \ref{s:cluster-brw} for details on the pruning.
Here we only describe the \BRW.
It is a two-type \BRW.
There are $1$-particles and $\rho$-particles.
It starts with one $\rho$-particle located at $0$.
The progeny are independent (between the two types) and Poisson distributed.
The matrix of mean is
\begin{equation}\label{e:M'}
M'(\beta,\rho,d) = 
\beta 
\begin{pmatrix}
0 & \left(\frac{1}{\sqrt\rho}\right)^d \\
\left(\sqrt\rho\right)^d  & 0
\end{pmatrix}.
\end{equation}
The steps of the \BRW\ are i.i.d.\ with uniform distribution on $K$.
(Recall that we are interested in $\frac 2{1+\rho}\widehat A^0$.)
Such \BRW\ will be denoted by
\[
\overline\tau^{\beta,\rho,d,X_K}.
\]
As in Section \ref{s:cluster-BRW-GM-M}, there can only be interference 
between two $1$-particles (when we are revealing $\chi_\rho$) or 
between two $\rho$-particles (when are revealing $\chi_1$).

\subsection{A result on survival probability of Galton-Watson processes}
\label{s:GW-GM}

Let $\beta>0,  \rho >1$ and $d \ge 1$.
Recall the matrices $M(\beta,\rho,d)$ and $M'(\beta,\rho,d)$ defined by \eqref{e:M} and \eqref{e:M'}.
Let $\tau(\beta,\rho,d)$ be a two-type Galton-Watson process with independent Poisson progeny with mean matrix $M(\beta,\rho,d)$ 
and starting with one $\rho$-particle.
Denote by $S(\beta,\rho,d)$ its survival probability.
Let $\tau'(\beta,\rho,d)$ be a two-type Galton-Watson process with independent Poisson progeny with mean matrix $M'(\beta,\rho,d)$ 
and starting with one $\rho$-particle.
Denote by $S'(\beta,\rho,d)$ its survival probability.
Let $\tau''(\beta^2)$ be a one type Galton-Watson process with Poisson$(\beta^2)$ progeny starting with one particle.
As usual we denote by $S(\beta^2)$ its survival probability.

The aim of this section is to prove the following lemma.

\begin{lemma} \label{l:GW} Let $\beta>0$ and  $\rho>1$. 
There exists a sequence $(\epsilon_{GW}(d))_d$ of positive real numbers which tends to $0$ such that,
\[
\forall d \ge 1,
S(\beta^2) - \epsilon_{GW}(d) \le S'(\beta,\rho,d) \le S(\beta,\rho,d) \le S(\beta^2) + \epsilon_{GW}(d).
\]
\end{lemma}

This is a straightforward consequence of the following lemma. 
We keep the notations  given at the beginning of Section \ref{s:GW-GM}.

\begin{lemma} \label{l:GW-inter} Let $\beta>0$ and  $\rho>1$. 
\begin{enumerate}
\item For all $n \ge 1$, 
$\limsup_{d \to \infty} \P[\tau(\beta,\rho,d) \text{ survives}] - \P[\tau'(\beta,\rho,d) \text{ survives up to step } n] \le 0$.
\item For all $n \ge 1$, 
$\limsup_{d \to \infty} \P[\tau'(\beta,\rho,d) \text{ survives up to step }2n] \le \P[\tau''(\beta^2) \text{ survives up to step } n]$.
\item $\lim_{n \to \infty} \P[\tau''(\beta^2) \text{ survives up to step }n]=  \P[\tau''(\beta^2) \text{ survives}]$.
\item For all $d \ge 1$, $\P[\tau'(\beta,\rho,d) \text{ survives }] \le P[\tau(\beta,\rho,d) \text{ survives}]$.
\item $\liminf_{d \to \infty} \P[\tau'(\beta,\rho,d) \text{ survives}] \ge \P[\tau''(\beta^2) \text{ survives}]$.
\end{enumerate} 
\end{lemma}

\begin{proof}[Proof of Lemma \ref{l:GW} using Lemma \ref{l:GW-inter}] 
Fix $\beta>0$ and  $\rho>1$.
Let $\epsilon>0$. 
By Item 3 we fix $n_0$ such that 
\[
\P[\tau''(\beta^2) \text{ survives up to step }n_0] \le   \P[\tau''(\beta^2) \text{ survives}] + \epsilon.
\]
Combining Item 1 (with $n=2n_0$) and Item 2 (with $n=n_0$), we get the existence of $d_2$ such that, 
\[
\forall d \ge d_2, \; \P[\tau(\beta,\rho,d) \text{ survives}] \le \P[\tau''(\beta^2) \text{ survives up to step } n_0] + \epsilon
\]
and then 
\begin{equation}\label{e:coursannule}
\forall d \ge d_2, \; \P[\tau(\beta,\rho,d) \text{ survives}] \le \P[\tau''(\beta^2) \text{ survives}] + 2\epsilon.
\end{equation}
By Item 5 of Lemma \ref{l:GW-inter}, there exists $d_1$ such that, 
\begin{equation}\label{e:coursannule2}
\forall d \ge d_1, \; \P[\tau''(\beta^2) \text{ survives }] \le P[\tau'(\beta,\rho,d) \text{ survives}] + \epsilon.
\end{equation}
The result follows from \eqref{e:coursannule2}, Item 4 of Lemma \ref{l:GW-inter} and \eqref{e:coursannule}.
\end{proof}

\begin{proof}[Proof of Item 1 of Lemma \ref{l:GW-inter}]
Let $\beta>0$ and $\rho>1$.
Hereafter, we assume that $d$ is large enough to ensure 
\begin{equation}\label{e:pasdepercodes1}
\beta \left(\frac{2\sqrt{\rho}}{1+\rho}\right)^d \le \frac 1 2.
\end{equation}
We couple in a natural way $\tau(\beta,\rho,d)$ and $\tau'(\beta,\rho,d)$.
If $\tau(\beta,\rho,d)$ survives, 
then there exists an infinite branch $0=x_0,x_1,x_2,\dots$ in $\tau(\beta,\rho,d)$.
If moreover $\tau'(\beta,\rho,d)$ does not survive at least up to step $n$, 
then there exists $k \in \{1,\dots,n\}$ such that $x_k$ and $x_{k-1}$ are two particles of the same type.
By \eqref{e:pasdepercodes1}, there exists no infinite branch of $1$-particles.
Therefore there exists $\ell \ge n$ such that $x_\ell$ is a $\rho$-particle.
Let $\ell$ be the smallest such integer.
The probability of the event
\[
\{\tau(\beta,\rho,d) \text{ survives}\} \setminus \{\tau'(\beta,\rho,d) \text{ survives up to step } n\}
\]
is thus bounded from above by the expected number of such paths $(x_0,\dots,x_\ell)$.

In order to bound this expected number of paths, we will associate a type with each such paths.
For each $k \in \{0,\dots,\ell\}$ we set $t_k=1$ if $x_k$ is a $1$-particle and $t_k=\rho$ otherwise.
We thus get a sequence $T=(t_0,\dots,t_\ell)$ of types.
This sequence belong to the set $\cT$ of finite sequence $(t'_0,\dots,t'_{\ell'})$ such that
\begin{enumerate}
\item $\ell' \ge n$.
\item $t'_0=t'_{\ell'}=\rho$. 
As a consequence, the sets 
$\{k \in \{1,\dots,\ell'\} : t_{k-1}=1 \text{ and } t_k=\rho\}$ and $\{k \in \{1,\dots,\ell'\} : t_{k-1}=\rho \text{ and } t_k=1\}$
have the same cardinality.
This is a crucial property which will cancel a large factor later in the proof.
\item There exists $k \in \{1,\dots,n\}$ such that $t_k =  t_{k-1}$.
\item If $\ell'>n$, then $t'_n=t'_{n+1}=\dots=t'_{\ell'-1}=1$. This is due to the fact that we stop, after step $n$, with the first $\rho$-particle. 
\end{enumerate}
Let us fix $T=(t_0,\dots,t_{\ell}) \in \cT$.
The expected number of paths of the tree $\tau(\beta,\rho,d)$ whose type is $T$ is $\beta^\ell\Delta^{I(T)}$ where
\[
\Delta=\left(\frac{2\sqrt{\rho}}{1+\rho}\right)^d<1
\]
and
\[
I(T) = \#\{k \in \{1,\dots,\ell\} : t_k=t_{k-1}\}.
\]
We have used the second property of types which yields some cancellations in antidiagonal coefficients of $M$.
By the third property and as $\ell \ge n$ we have $I(T) \ge 1$.
By the fourth one we have $I(T) \ge \ell-1-n$.
Therefore $I(T) \ge \max(1,\ell-1-n)$ and the expected number of paths of $\tau$ whose type is $T$ is at most 
\[
\beta^\ell\Delta^{\max(1,\ell	-1-n)}.
\]
Summing over types, we get  
\[
\P[\{\tau(\beta,\rho,d) \text{ survives}\} \setminus \{\tau'(\beta,\rho,d) \text{ survives at least up to step }n\}]  
\le \sum_{T \in \cT} \beta^\ell\Delta^{\max(1,\ell-1-n)}.
\]
The contribution of types of length $\ell=n$ is at most $2^n \beta^n \Delta$.
The contribution of types of length $\ell=n+1$ is at most $2^n \beta^{n+1} \Delta$.
The contribution of types of length $\ell \ge n+2$ is at most (we use $\beta \Delta \le 1/2$, see \eqref{e:pasdepercodes1}),
\[
2^n\sum_{\ell \ge n+2} \beta^\ell  \Delta^{\ell-1-n} = 2^n \beta^{n+1} \frac{\beta \Delta}{1-\beta \Delta} \le 2^{n+1}\beta^{n+1}\beta \Delta.
\]
Therefore
\[
\P[\{\tau(\beta,\rho,d) \text{ survives}\} \setminus \{\tau'(\beta,\rho,d) \text{ survives at least up to step }n\}] 
 \le 2^n \beta^n (1+\beta+2\beta^2)\left(\frac{2\sqrt{\rho}}{1+\rho}\right)^d
\]
and then
\[
\P[\tau(\beta,\rho,d) \text{ survives}]\le  \P[\tau'(\beta,\rho,d) \text{ survives at least up to step }n] 
 + 2^n \beta^n (1+\beta+2\beta^2)\left(\frac{2\sqrt{\rho}}{1+\rho}\right)^d.
\]
This yields the result.
\end{proof}

\begin{proof}[Proof of Item 2 of Lemma \ref{l:GW-inter}]
Let $\beta>0$, $\rho > 1$ and $n_0 \ge 1$. Let $d \ge 1$.
Write
\begin{equation} \label{e:mumustar}
\mu=\beta \sqrt{\rho}^d \text{ and }\mu^* = \beta \frac 1 {\sqrt{\rho}^d}.
\end{equation}
Let $N(\mu)$ be a Poisson$(\mu)$ random variable and $N_1(\mu^*), N_2(\mu^*), ... $ be Poisson$(\mu^*)$ random variables.
Assume that these random variables are independent.
Set
\begin{equation}\label{e:xd}
X(\beta,\rho,d) = \sum_{i=1}^{N(\mu)} N_i(\mu^*).
\end{equation}

The process $(\tau'_{2n}(\beta,\rho,d))_n$ is a Galton-Watson process with progeny distributed as $X(\beta,\rho,d)$.
When $d$ converges to $\infty$, $X(\beta,\rho,d)$ converges in distribution to a Poisson$(\beta^2)$ random variable $N(\beta^2)$.
This can for example be shown by computing the characteristic functions: for all $t \in \R$,
\[
\E\big[e^{it X(\beta,\rho,d)}\big] = \exp\left[-\beta \sqrt{\rho}^d\left(1-\exp\left(- \beta \sqrt\rho^{-d}\left(1-e^{it}\right)\right)\right)\right]
\to_{d \to \infty}
\exp\big[-\beta^2\big(1-e^{it}\big)\big]=\E\big[e^{itN(\beta^2)}\big].
\]
Thus, there exists a sequence of random variables $\widetilde X(d)$, each of which as the same distribution of $X(\beta,\rho,d)$, 
and a random variable $\widetilde N(\beta^2)$ with Poisson$(\beta^2)$ distribution such that $\widetilde X(d)$ converges almost surely to $\widetilde N$.
Let us use such coupling for all variables defining our Galton-Watson processes.
We thus get a new version of $(\tau'_{2n}(\beta,\rho,d))_n$ -- which we denote by $(\widetilde\tau'_{2n}(\beta,\rho,d))_n$ --
and a new version of $\tau''(\beta^2)$ -- which we denote by $\widetilde\tau''(\beta^2)$ --
such that $(\widetilde\tau'_{2n}(\beta,\rho,d))_{n \le n_0}$ 
converges almost surely to $(\widetilde\tau''(\beta^2)_n)_{n \le n_0}$ when $d$ tends to $\infty$.
Therefore
\begin{align}
& \P[\tau'(\beta,\rho,d) \text{ survives up to step }2n]  \nonumber \\
 & = \P[(\tau'_{2n}(\beta,\rho,d))_n \text{ survives up to step }n] \nonumber \\
 & = \P[(\widetilde\tau'_{2n}(\beta,\rho,d))_n \text{ survives up to step }n] \text{ as they have the same distribution} \nonumber\\
 & \to \P[\widetilde\tau''(\beta^2) \text{ survives up to step }n]\text{ as } d \to \infty \text{ by the discussion above} \nonumber\\
 & = \P[\tau''(\beta^2) \text{ survives up to step }n] \text{ as they have the same distribution}. \nonumber
\end{align}
This yields the result.
\end{proof}

\begin{proof}[Proof of Item 3 of Lemma \ref{l:GW-inter}] This is straightforward as the sequence of events is non-increasing. 
\end{proof}

\begin{proof}[Proof of Item 4 of Lemma \ref{l:GW-inter}] This is straightforward by a natural coupling.
\end{proof}

\begin{proof}[Proof of Item 5 of Lemma \ref{l:GW-inter}]
Let $\beta>0$, $\rho>1$ and $d \ge 1$.
We will use some remarks and notations from the proof of Item 2, in particular \eqref{e:mumustar} and \eqref{e:xd}.
The process $(\tau'(\beta,\rho,d)_{2n})_n$ is a Galton-Watson process with progeny distributed as $X(\beta,\rho,d)$.
But
\[
X(\beta,\rho,d) \ge  \sum_{i=1}^{N(\mu)} \1_{N_i(\mu^*) \ge 1}
\]
which is Poisson$(\mu(1-\exp(-\mu^*)))$ distributed.
Therefore
\begin{align*}
\P[\tau'(\beta,\rho,d) \text{ survives}] 
 & = \P[(\tau'(\beta,\rho,d)_{2n})_n \text{ survives}] \\
 & = \P[\text{a Galton-Watson process with progeny distribued as }X(\beta,\rho,d)\text{ survives}] \\
 & \ge \P[\text{a Galton-Watson process with Poisson}(\mu(1-\exp(-\mu^*)))\text{ progeny survives}] \\
 & = S(\mu(1-\exp(-\mu^*)) \\
 & \to S(\beta^2) \text{ as } d \to \infty
\end{align*}
as $\mu(1-\exp(-\mu^*) \to \beta^2$ and as $S$ is continuous by Lemma \ref{l:continuity-psi}.
\end{proof}

\subsection{Proof of the upper bound on percolation probability}
\label{s:t-GM-convex-upper}

We use the framework of Section \ref{s:framework-GM} and the first \BRW\ defined in Section \ref{s:cluster-BRW-GM}.
The aim of this section is to prove the following result.
This is the easy part in the control of percolation probability and percolation threshold.

\begin{prop} \label{p:t-GM-convex-upper}
\begin{itemize}
\item Let $\beta>0$.
For any $\rho>1$, any $d \ge 1$ and any $K\in\cK(d)$,
\begin{equation}\label{e:upper-bound-perco-proba-GM}
\P\left[ \#\widehat C^0\big(\beta, \rho, d,K\big) = \infty\right] \le S(\beta,\rho,d).
\end{equation}
\item For any $\rho>1$, any $d \ge 1$ and any $K\in\cK(d)$,
\begin{equation}\label{e:lower-bound-beta_c-GM}
\beta_c(\rho,d,K) \ge \frac{1}{1+\left(\frac{2\sqrt \rho}{1+\rho}\right)^d}.
\end{equation}
\item Let $\beta>0$ and $\rho>1$. 
There exists a sequence $(\epsilon(d))_d$ of positive real numbers which tends to $0$ such that,
\[
\forall d \ge 1, \forall K\in\cK(d),
\P\left[ \#\widehat C^0\big(\beta,\rho,d,K\big) = \infty\right] \le S(\beta^2) + \epsilon(d).
\]
\item Let $\rho>1$. 
There exists a sequence $(\epsilon'(d))_d$ of positive real numbers which tends to $0$ such that,
\[
 \forall d \ge 1, \forall K\in\cK(d), \beta_c(\rho,d,K) \ge 1-\epsilon'(d).
\]
\end{itemize}
\end{prop}

\begin{proof}
Let $\beta>0, \rho>0, d \ge 1$ and $K \in \cK(d)$.
Let $\overline\tau$ be the associated \BRW\ introduced in Section \ref{s:cluster-BRW-GM-M}.
In particular, the mean matrix of the associated two-type Galton-Watson process is $M(\beta,\rho,d)$ which is defined in \eqref{e:M}.
By the discussion of Section \ref{s:cluster-BRW-GM-M},
\[
\P\left[ \#\widehat C^0\big(\beta,\rho,d,K\big) = \infty\right] \le \P[\#\overline\tau=\infty]=S(\beta,\rho,d).
\]
This gives \eqref{e:upper-bound-perco-proba-GM}. 
The largest eigenvalue of $M(\beta,\rho,d)$ is 

\[
r = \beta\left[1 + \left(\frac{2\sqrt\rho}{1+\rho}\right)^d\right].
\]
Therefore $S(\beta,\rho,d)$ equals $0$ when $r \le 1$ (see \cite{Athreya-ney-book} page 186).
Combined with the upper bound on the percolation probability, this yields \eqref{e:lower-bound-beta_c-GM}.

Lemma  \ref{l:GW} and \eqref{e:upper-bound-perco-proba-GM} yield Item 3.
Item 4 is a straightforward consequence of \eqref{e:lower-bound-beta_c-GM}.
\end{proof}

\subsection{Proof of the lower bound on percolation probability} 
\label{s:t-GM-convex-lower}

This is the hard part in the control of percolation probability and percolation threshold.
Our aim is to prove the following theorem.

\begin{theorem} \label{t:GM-convex-lower}
\begin{itemize}
\item Let $\beta>0$ and $1<\rho<2$. 
There exists a sequence $(\epsilon(d))_d$ of positive real numbers which tends to $0$ such that,
\[
\forall d \ge 1, \forall K\in\cK(d), \;
\P\left[ \#\widehat C^0\big(\beta,\rho,d,K\big) = \infty\right] \ge S(\beta^2) - \epsilon(d).
\]
\item Let $1<\rho<2$.   
There exists a sequence $(\epsilon'(d))_d$ of positive real numbers which tends to $0$ such that,
\[
\forall d \ge 1, \forall K\in\cK(d), \; \beta_c(\rho,d,K) \le 1+\epsilon'(d).
\]
\end{itemize}
\end{theorem}

\begin{proof}[Proof of Theorem \ref{t:GM-convex} using Theorem \ref{t:GM-convex-lower}]
Recall Section \ref{s:framework-GM} for notations. 
In particular, for all $\beta>0,\rho>1,d \ge 1$ and $K \in \cK(d)$, by \eqref{e:ahouicetruc} and by our choice of notations 
(which we simplified in Section \ref{s:preuve-GM-convexe}):
\[
\P[\#\widehat C^0(\beta,\rho,d,K) = \infty]= \P[C^0(\beta \kappa_c(\rho)^d, \nu_{d,\rho},d,K) \text{ is unbounded}].
\]
Item 1 of Theorem \ref{t:GM-convex-lower}, Item 1 of Proposition \ref{p:t-GM-convex-upper} and Lemma \ref{l:GW} 
yield Item 1 of Theorem \ref{t:GM-convex}.
Item 2 of Theorem \ref{t:GM-convex-lower} and Item 4 of Proposition \ref{p:t-GM-convex-upper} 
yield Item 2 of Theorem \ref{t:GM-convex}.
\end{proof}

\subsubsection{Good gaps}
\label{s:good-gaps-GM}

Let $d \ge 1$ and $K \in \cK(d)$.
As usual, denote by $X_K, X'_K, X''_K$ i.i.d.r.v.\ uniformly distributed on $K$.
Let $L$ any map provided by Theorem \ref{t:Klartag-CLT} for $X_K$ and define $\widehat L$ by $\widehat L = 2^{-1/2}L$.
Let $T$ be any map adapted to $X_K$ (see \eqref{e:adapted}).

For all $\eta>0$ smaller than $\sqrt{2}$ set
\begin{align*}
G(d,K,\eta) & = \{z \in \R^d : \P[z + X_K \not\in K] \ge 1-\eta\}, \\
H'(d,K,\eta) & = \{z : \|T(z)\|_2 d^{-1/2} < 3/2\}, \\
G'(d,K,\eta) & = \{z \in \R^d : \P[z+ X_K+X'_K \not\in H'(d,K,\eta)] \ge 1-\eta\}, \\
H(d,K,\eta) & = \{z \in \R^d  : \P[z + X_K \in K] \le (\sqrt{2}-\eta)^{-d} \}.
\end{align*}
The family of set $G$ was already introduced in Section \ref{s:good-gaps}   with a similar purpose.
The only property of $3/2$ we are interested in is that $\sqrt 2 < 3/2 < \sqrt 3$.
When the parameters are clear from the context, we will write $G, G'$ and so on.

\begin{lemma} \label{l:good-gap-GM} 
Let $\eta \in (0,\sqrt{2})$.
There exists a sequence $(\epsilon'_G(d,\eta))_d$ that tends to $0$ such that 
for all $d \ge 1$, all $K \in \cK(d)$, all $a \in \R^d$:
\begin{align*}
\P[a + X_K \in G(d,K,\eta)] & \ge 1-\epsilon'_G(d,\eta), \\
\P[X_K+X_K' \in H'(d,K,\eta)] & \ge 1-\epsilon'_G(d,\eta), \\
\P[a + X_K \in G'(d,K,\eta)] & \ge 1-\epsilon'_G(d,\eta), \\
\P[X_K+X_K' \in H(d,K,\eta)] & \ge 1-\epsilon'_G(d,\eta). \\
\end{align*}
\end{lemma}

\begin{proof}

We have to prove that four probabilities tends to $1$, uniformly in $K \in  \cK(d)$ and $a \in \R^d$, as $d$ tends to $\infty$.
This is the content of Lemma \ref{l:good_gap} for $\P[a + X_K \in G(d,K,\eta)]$.
By Theorem \ref{t:Klartag-concentration}, we have the following convergence in probability when $d$ tends to $\infty$.
\[
\frac {\|T(X_K+X_K')\|_2}{\sqrt d} \to \sqrt{2}  \text{ and }\frac{ \|T(X_K+X_K'+X_K'')\|_2}{\sqrt d} \to \sqrt{3}.
\]
The convergence are uniform in $K \in \cK(d)$.
The first convergences gives immediately the required result for $\P[X_K+X_K' \in H'(d,K,\eta)]$ as $\sqrt{2}<3/2$.
With some further work, the second convergence will give the required result for $\P[a + X_K \in G'(d,K,\eta)]$.

Indeed, for all $d \ge 1, a \in \R^d$ and $K \in \cK(d)$,
\begin{align*}
\P[a + X_K \not\in G']
 & = \P[ \P[a+X_K+X_K'+X_K'' \in H'| X_K] \ge \eta] \\
 & \le \eta^{-1} \P[a+X_K+X_K'+X_K'' \in H'] \text{ by Markov inequality} \\
 & = \eta^{-1} \int_{\R^d} \1_{H'}(a+s) \1_K * \1_K * \1_K (s) \d s \\
 & = \eta^{-1} \1_{H'}*\1_K * \1_K * \1_K(a) \text{ by symmetry of }K.
\end{align*}
But $H'$ and $K$ are convex and symmetric.
Therefore $\1_H'$ and $\1_K$ are log-concave and symmetric.
Hence, $\1_{H'}*\1_K * \1_K * \1_K$ is log-concave (see Section \ref{s:log-concave}) and symmetric.
As a consequence, 
\[
\1_{H'}*\1_K * \1_K * \1_K(a) \le \1_{H'}*\1_K * \1_K * \1_K(0).
\]
Therefore,
\begin{align*}
\P[a + X_K \not\in G']
 & \le \eta^{-1} \1_{H'}*\1_K * \1_K * \1_K(0) \\
 & = \eta^{-1} \P[X_K+X_K'+X_K'' \in H'] \text{ by the same arguments} \\
 & = \eta^{-1} \P[\|T(X_K+X_K'+X_K'')\|_2 d^{-1/2}< 3/2] \text{ by definition of } H'.
\end{align*}
With the above mentioned consequence of Theorem \ref{t:Klartag-concentration}, we then get the required result about $\P[a + X_K \in G'(d,K,\eta)]$.

Let us now consider $\P[X_K+X_K' \in H(d,K,\eta)]$.
For all $z \in \R^d$ we have, using symmetry of $K$,
\[
\P[z+X_K \in K] = \int_{\R^d} \1_K(z+x)\1_K(x) \d x = \1_K * \1_K(z).
\]
Therefore, 
\[
|\R^d \setminus H| = |\{z \in \R^d : \1_K * \1_K(z) \ge (\sqrt 2 - \eta)^{-d} \}| \le (\sqrt 2 - \eta)^d \int_{\R^d} \1_K * \1_K =  (\sqrt 2 - \eta)^d .
\]
But then, by Theorem \ref{t:riesz} (rearrangement inequality),
\begin{align*}
\P[X_K+X_K' \not\in H] 
 & = \int_{\R^d} \1_{\R^d \setminus H}(s)\1_K * \1_K(s) \d s \\
 & \le \int_{\R^d} \1_{(\sqrt 2-\eta)B}(s)\1_B * \1_B(s) \d s \\
 & = \P[\|X_B+X_B'\|_2 \le (\sqrt 2-\eta)]
\end{align*}
which tends to $0$ by Lemma \ref{l:somme-unif-boule}.
\end{proof}

\subsubsection{Embedding of a two-dimensional lattice in $\R^d$}
\label{s:embedding-GM}

Let $d \ge 1$ and $K \in \cK(d)$.
Let $L$ be any map given by Theorem \ref{t:Klartag-CLT} for $X_K$.
Define $\widehat L$ by $\widehat L(x)=2^{-1/2}L(x)$.
Let $\gaussienne$ and $\gaussienne'$ be two independent standard Gaussian random vector in $\R^2$.
Under an appropriate coupling,
\[
\P[L(X_K) \neq \gaussienne] \le \epsilon_{CLT}(d)
\]
where $\epsilon_{CLT}$ is the sequence which appears in the statement of Theorem \ref{t:Klartag-CLT}.
Therefore
\begin{equation}\label{e:application-CLT}
\P[\widehat L(X_K) \neq 2^{-1/2}\gaussienne] \le \epsilon_{CLT}(d).
\end{equation}
Under an appropriate coupling,
\[
\P[L(X_K+X'_K) \neq \gaussienne+\gaussienne'] \le 2\epsilon_{CLT}(d).
\]
As $2^{-1/2} (\gaussienne+\gaussienne')$ has the same distribution as $\gaussienne$ we then have, under a new coupling,
\begin{equation}\label{e:application-CLT-somme}
\P[\widehat L(X_K+X'_K) \neq \gaussienne] \le 2\epsilon_{CLT}(d).
\end{equation}

Recall the definition of $\overline\cL$ and $A(\cdot,\cdot)$ in Section \ref{s:embedding}.
For any $(i,j) \in \overline\cL$, we set
\[
A_{\widehat L}(i,j) = {\widehat L}^{-1}(A(i,j)).
\]
The sets $A_{\widehat{L}}(i,j)$ are pairwise disjoint.
Moreover $0$ belongs to $A_{\widehat{L}}(0,0)$.

As in the proof of Theorem \ref{t:penrose-convex}, we will use this embedding of $\overline\cL$ to compare
the cluster of the origin to a supercritical percolation process on $\overline\cL$.

\subsubsection{An estimate about \BRW}
\label{s:brw-convex-GM}

The aim of this section is to prove the following result.
This is a consequence of Lemma \ref{l:briqueG}.
Recall the notations about \BRW\ in Section \ref{s:brw}.
In particular, note that the underlying Galton-Watson process of $\overline\tau^{\beta,\rho,d,X_K;\cS}$ 
is a two-type Galton-Watson with matrix mean $M'(\beta,\rho,d)$ and not $M(\beta,\rho,d)$.

\begin{lemma} \label{l:brique-GM}
Let $\beta, \rho>1$ and $\epsilon>0$. 
There exists $m,d_0,k,M \ge 1$ with $k$ even such that, for all $d \ge d_0$ and all $K \in \cK(d)$,
the following properties hold where $L$ is any map given by Theorem \ref{t:Klartag-CLT} for $X_K$ and where $\widehat L=2^{-1/2}L$,
\begin{itemize}
\item For all $z \in A_{\widehat L}(0,0)$, 
\[
\P\left[\overline\tau^{\beta,\rho,d,X_K;\{z\}}_k (A_{\widehat L}(1,0)) \ge m 
\text{ and } \smallrho{\overline\tau^{\beta,\rho,d,X_K;\{z\}}_{\le k}}M\rho \right] \ge \survival(\beta^2)-\epsilon.
\]
\item For all $(i,j) \in \cL$ and all subset $\cS \subset A_{\widehat L}(i,j)$ of cardinality $m$,  
\begin{align*}
& \P\Big[\overline\tau^{\beta,\rho,d,X_K;\cS}_k (A_{\widehat L}(i+1,j+1)) \ge m 
\text{ and } \overline\tau^{\beta,\rho,d,X_K;\cS}_k (A_{\widehat L}(i+1,j-1)) \ge m  \\
& \text{ and } \smallrho{\overline\tau^{\beta,\rho,d,X_K;\cS}_{\le k}}M\rho \Big] \ge 1-\epsilon.
\end{align*}
\end{itemize}
\end{lemma}

\begin{proof} Let $\beta,\rho>1$ and $\epsilon>0$.
We can and will assume that $\epsilon$ small enough to ensure $\survival(\beta^2)-\epsilon>0$.
By Lemma \ref{l:continuity-psi}, we can choose $\eta>0$ such that
\[
\survival(\beta^2-\eta) \ge \survival(\beta^2)-\epsilon.
\]

Applying Lemma \ref{l:briqueG} with $\lambda=\beta^2-\eta$ and $\epsilon$ we get constants that we denote by $m_2,k_2,M_2$.
For all $d \ge 1$ write
\[
\mu=\beta \sqrt{\rho}^d \text{ and }\mu^* = \beta \frac 1 {\sqrt{\rho}^d}.
\]
Let $d_1$ be such that, for all $d \ge d_1$,
\[
\mu(1-\exp(-\mu^*)) \ge \beta^2-\eta.
\]
Let $d_2$ be such that, for all $d \ge d_2$, $\epsilon_{CLT}(d) \le \epsilon/M_2$ where $(\epsilon_{CLT}(d))_d$ 
is the sequence given by Theorem \ref{t:Klartag-CLT}.
Let $M_3$ be such that, for all $d \ge 1$, all $K \in \cK(d)$ and all $\cS \subset \R^d$ of cardinality at most $m_2$,
\begin{equation}\label{e:lasagne}
\P[\smallrho{\overline\tau^{\beta,\rho,d,X_K ; \cS}_{\le 2k_2}}{M_3}{\rho}] \ge 1-\epsilon.
\end{equation}
Let us check that this is possible.
Denote by $N_\rho$ the number of $\rho$-particles of $\overline\tau^{\beta,\rho,d,X_K ; \cS}_{\le 2k_2}$.
Denote by $N_1$ the number of $1$-particles of $\overline\tau^{\beta,\rho,d,X_K ; \cS}_{\le 2k_2}$.
Then, for any $M_3$,
\[
\P[N_\rho \ge M_3] \le \frac{\E[N_\rho]}{M_3} \le  m_2\frac{\sum_{i=0}^{k_2} (\mu\mu^*)^i}{M_3} = m_2\frac{\sum_{i=0}^{k_2} \beta^{2i}}{M_3}
\]
and
\[
\P[N_1 \ge M_3 \sqrt{\rho}^d] \le \frac{\E[N_1]}{M_3\sqrt{\rho}^d} \le  m_2\mu \frac{\sum_{i=0}^{k_2-1} (\mu^*\mu)^i}{M_3\sqrt{\rho}^d}
= m_2\beta \sqrt{\rho}^d  \frac{\sum_{i=0}^{k_2-1}  \beta^{2i}}{M_3\sqrt{\rho}^d}=m_2\beta  \frac{\sum_{i=0}^{k_2-1}  \beta^{2i}}{M_3}.
\]
We can thus fix $M_3$ as stated above.
 
%Let $d_4$ be such that, for all $d \ge d_4$, 
%\[
%\epsilon_{CLT}(d) \le \min(\epsilon/M_3,1/4).
%\]
%Let $\troncature_4$ be such that
%\[
%\P[\|\gaussienne\|_2 \ge 2^{1/2}\troncature_4 ] \le \min(\epsilon/M_3,1/4).
%\]
%Then, for all $d \ge d_4$, all $K \in \cK(d)$ and any map $L$ associated with $X_K$ by Theorem \ref{t:Klartag-CLT} and setting
%$\widehat L = 2^{-1/2} L$ (see Section \ref{s:embedding-GM}) we have
%\begin{equation}\label{e:serpillere}
%\P[\|\widehat L(X_K)\|_2 \ge \troncature_4] 
%=
%\P[\|{\widehat L}(X_K)\|_2 \ge 2^{1/2}\troncature_4] 
%\le 
%\P[{\widehat L}(X_K) \neq \gaussienne] + \P[\|\gaussienne\|_2 \ge 2^{1/2}\troncature_4]
%\le  \min(2\epsilon/M_3,1/2).
%\end{equation}
%We used the definition of $d_4$ and the definition of $\troncature_4$.
Finally set $d_0=\max(d_1,d_2)$.

We now prove that the conclusion of the lemma holds with $m_2,d_0,2k_2,M_3$.
Let $d \ge d_0, K \in \cK(d)$.
Let $L$ by any map associated with $X_K$ by Theorem \ref{t:Klartag-CLT} and set $\widehat L = 2^{-1/2} L$.
Let $\cS$ be a finite subset of $\R^d$.

Consider the \BRW\ $\left(\overline\tau^{\beta,\rho,d,X_K ; \cS}_{2n}\right)_n$ where we sample at even steps.
Let us prove the following stochastic domination  between two \BRW\footnote{Note that we sample at even times on the left and at all times on the right.}:
\begin{equation}\label{e:stodominationbrw}
\left(\overline\tau^{\beta,\rho,d,X_K ; \cS}_{2n}\right)_n \ge \left(\overline\tau^{\beta^2 - \eta,X_K+X_K' ; \cS}_n\right)_n.
\end{equation}
Prune $\overline\tau^{\beta,\rho,d,X_K ; \cS}$ in the following way.
If a $1$-particle has strictly more that one children, keep the first one in Neveu ordering (see Section \ref{s:brw}) 
and remove all the other ones and their progeny.
Denote by $\overline\tau^-$ the new \BRW.
We have the following stochastic domination:
\[
\left(\overline\tau^{\beta,\rho,d,X_K ; \cS}_n\right)_n \ge \left(\overline\tau^-_n\right)_n.
\]
and then
\begin{equation}\label{e:stodominationbrwetape}
\left(\overline\tau^{\beta,\rho,d,X_K ; \cS}_{2n}\right)_n \ge \left(\overline\tau^-_{2n}\right)_n.
\end{equation}
Consider the \BRW\ on the right.
The steps are independent copies of $X_K+X_K'$.
The progeny is distributed as the sum of a Poisson($\mu$) number of independent Bernoulli$(1-\exp(-\mu^*))$ random variables.
Therefore, the progeny is Poisson$(\mu(1-\exp(-\mu^*))$ distributed.
As $d \ge d_0 \ge d_1$, the progeny stochastically dominates a Poisson$(\beta^2-\eta)$ random variable.
As a consequence $(\overline\tau^-_{2n})_n$ stochastically dominates the \BRW\ on the right of \eqref{e:stodominationbrw}.
With \eqref{e:stodominationbrwetape}, this yields \eqref{e:stodominationbrw}.

As $d \ge d_0 \ge d_2$ we have, under an appropriate coupling, 
\begin{equation}\label{e:aout}
\P[{\widehat{L}}(X_K+X'_K) \neq \gaussienne] \le 2\epsilon/M_2.
\end{equation}
Assume that $\cS \subset \R^d$ is as set of cardinality at most $m_2$.
Let $(i',j') \in \overline\cL$.
We have
\begin{align}
& \P\left[\overline\tau^{\beta^2 - \eta,\gaussienne ; {\widehat{L}}(\cS)}_{k_2}(A(i',j')) \ge m_2 
  \text{ and there are at most } M_2 \text{ particles in } \overline\tau^{\beta^2 - \eta,\gaussienne ; {\widehat{L}}(\cS)}_{\le k_2}\right] \nonumber \\
& \le  \P\left[\overline\tau^{\beta^2 - \eta,{\widehat{L}}(X_K+X_K') ; {\widehat{L}}(\cS)}_{k_2}(A(i',j')) \ge m_2 \right] + M_2\frac{2\epsilon}{M_2}
\nonumber \text{ by } \eqref{e:aout} \\
& = \P\left[\overline\tau^{\beta^2 - \eta,X_K+X_K' ; \cS}_{k_2}(A_{\widehat{L}}(i',j')) \ge m_2\right] + 2\epsilon \nonumber \\
& \le \P\left[\overline\tau^{\beta,\rho,d,X_K ; \cS }_{2k_2}(A_{\widehat{L}}(i',j')) \ge m_2 \right] + 2\epsilon	\text{ by } \eqref{e:stodominationbrw}.
\label{e:vicetroud}
\end{align}
Therefore,
\begin{align*}
&  \P\left[\overline\tau^{\beta,\rho,d,X_K ; \cS }_{2k_2}(A_{\widehat{L}}(i',j')) \ge m_2 
\text{ and } \smallrho{\overline\tau^{\beta,\rho,d,X_K ; \cS}_{\le 2k_2}}{M_3}{\rho} \right]   \\
& \ge \P\left[\overline\tau^{\beta,\rho,d,X_K ; \cS }_{2k_2}(A_{\widehat{L}}(i',j')) \ge m_2 \right]  - \epsilon  \text{ by }\eqref{e:lasagne}  \\
& \ge \P\left[\overline\tau^{\beta^2 - \eta,\gaussienne ; {\widehat{L}}(\cS)}_{k_2}(A(i',j')) \ge m_2 
  \text{ and there are at most } M_2 \text{ particles in } \overline\tau^{\beta^2 - \eta,\gaussienne ; {\widehat{L}}(\cS)}_{\le k_2}\right] 
  - 3\epsilon \text{ by } \eqref{e:vicetroud}.
\end{align*}

We can now conclude using the definition of $m_2,k_2,M_2$ (which come from Lemma \ref{l:briqueG} with $\lambda=\beta^2-\eta$ and $\epsilon$).
If $\cS=\{z\} \subset A_{\widehat L}(0,0)$ we get, with $(i',j')=(1,0)$ and the inequality obtained above,
\begin{align*}
& \P\Big[\overline\tau^{\beta,\rho,d,X_K ; \{z\}}_{2k_2}(A_{\widehat L}(1,0)) \ge m_2  
\text{ and } \smallrho{\overline\tau^{\beta,\rho,d,X_K ; \{z\}}_{\le 2k_2}}{M_3}{\rho} \Big]  \\
& \ge \P\left[\overline\tau^{\beta^2 - \eta,\gaussienne ; \{\widehat L (z)\}}_{k_2}(A(1,0)) \ge m_2 
  \text{ and there are at most } M_2 \text{ particles in } \overline\tau^{\beta^2 - \eta,\gaussienne ; \{\widehat{L}(z)\}}_{\le k_2}\right] -3\epsilon \\
& \ge \survival(\beta^2-\eta) - 4\epsilon  \quad\text{ by the choice of } m_2,k_2,M_2\\
& \ge \survival(\beta^2)-5\epsilon \quad \text{by the choice of }\eta.
\end{align*}
Let $(i,j) \in \cL$.
If $\cS$ is a subset of cardinality $m_2$ of $A_{\widehat L}(i,j)$ we get similarly, with $(i',j')=(i+1,j\pm 1)$,
\[
\P\Big[\overline\tau^{\beta,\rho,d,X_K ; \cS}_{2k_2}(A_{\widehat L}(i+1,j\pm 1)) \ge m_2 
\text{ and } \smallrho{\overline\tau^{\beta,\rho,d,X_K ; \cS}_{\le 2k_2}}{M_3}{\rho} \Big]  
\ge 1 - 4\epsilon.
\]
Therefore,
\begin{align*}
& \P\Big[\overline\tau^{\beta,\rho,d,X_K;\cS}_{2k_2} (A_{\widehat L}(i+1,j+1)) \ge m_2
\text{ and } \overline\tau^{\beta,\rho,d,X_K;\cS}_{2k_2} (A_{\widehat L}(i+1,j-1)) \ge m_2  \\
& \text{ and } \smallrho{\overline\tau^{\beta,\rho,d,X_K;\cS}_{\le 2k_2}}{M_3}\rho \Big] \ge 1-8\epsilon.
\end{align*}
Items 1 and 2 hold with the choice of parameters $m_2,d_0,2k_2,M_3$.
\end{proof}

\subsubsection{Plan and intuition.}
\label{s:plan-GM}

The aim of this section is to present in an informal way the plan of the proof.
This a refinement of the plan given in Section \ref{s:plan} in the constant radius case.
We assume that the reader is familiar with the content of Section \ref{s:plan}.

\paragraph{Setup.} Let $\beta>0$, $1<\rho<2$, $\epsilon>0$, $d \ge 1$ and $K \in \cK(d)$. 
As usual, we denote by $X_K$ and $X_K'$ two independent random variables with uniform distribution on $K$.
%Fix a linear map $L:\R^d\to\R^2$ given by Theorem \ref{t:Klartag-CLT} for $X_K$.
%Recall the definition of the sets $A_L(i,j)$, $(i,j) \in \overline\cL$ in Section \ref{s:embedding}.
Recall the definition of $\widehat C^0=\widehat C^0(\beta,\rho,d,K)$ and $\widehat A^0=\widehat A^0(\beta,\rho,d,K)$ in Section \ref{s:framework-GM}.
Recall that $S(\beta^2)$ denotes the survival probability of a Poisson$(\beta^2)$ offspring Galton-Watson process.
The aim is to prove that the inequality
\[
\P[\#\widehat C^0(\beta,\rho,d,K)=\infty] \ge S(\beta^2)-\epsilon
\]
holds for any $d$ large enough, uniformly in $K \in \cK(d)$.
We actually prove that for any $d$ large enough, uniformly in $K \in \cK(d)$,
\[
\P[\#\widehat A^0(\beta,\rho,d,K)=\infty] \ge S(\beta^2)-\epsilon.
\]
This is a stronger result as $\widehat A^0$ is a subset of $\widehat C^0$.
The advantage is that the structure of $\widehat A^0$ is easier.
The idea of focusing on $\widehat A^0$ instead of $\widehat C^0$ was already used in \cite{GM-grande-dimension}.
However, it was used in a cruder way. Recall that the main result in \cite{GM-grande-dimension} is a logarithmic equivalent of the critical parameter
in the Euclidean case,
stated here as Theorem \ref{t:GM}.

The set $\widehat A^0$ can be built from the \BRW\ $\overline\tau^{\beta,\rho,d,X_K}$.
See Section \ref{s:brw} for notations on \BRW\ and Section \ref{s:cluster-BRW-GM-M'} 
for the construction of $\widehat A^0$ from $\overline\tau^{\beta,\rho,d,X_K}$.
Note in particular that, in $\overline\tau^{\beta,\rho,d,X_K}$, children of $\rho$-particles are $1$-particles and
children of $1-$particles are $\rho$-particles.
The basic intuition is that, up to an event whose probability vanishes when $d$ tends to $\infty$, 
$\widehat A^0$ is infinite when $\overline\tau^{\beta,\rho,d,X_K}$ is infinite.

\paragraph{The underlying Galton-Watson tree.}
The underlying Galton-Watson tree of the \BRW\ $\overline\tau^{\beta,\rho,d,X_K}$ does not depend on $K$ but it depends on $d$.
The number of children of a $\rho$ particles tends to $\infty$ when $d$ tends to $\infty$.
Indeed, this is a Poisson random variable with parameter $\beta \sqrt\rho^d$.
Thus, the Galton-Watson process underlying $\overline\tau^{\beta,\rho,d,X_K}$ degenerates when $d$ tends to $\infty$.
This is a major difference with respect to the constant radius case and this is why the non constant case is much more involved.
It is also for this reason that Theorem \ref{t:GM-convex} does not hold for $\rho>2$ 
(the explosion of the number of children of $\rho$-particles increases when $\rho$ increase ; 
see \cite{GM-grande-dimension-2} where a logarithmic equivalent
for the critical probability is given in the Euclidean case for every $\rho$ : the behavior is not given by the sole underlying Galton-Watson
process when $\rho>2$).
There is however a nice feature.
When $d$ tends to $\infty$, 
\begin{equation}\label{e:approx}
\left(\overline\tau^{\beta,\rho,d,X_K}_{2n}\right)_n
\approx
\left(\overline\tau^{\beta^2,d,X_K+X_K'}_n\right)_n.
\end{equation}
Note that the first \BRW\ is sampled at even times.
The second \BRW\ is similar to the \BRW\ studied in Section \ref{s:preuve-penrose-convexe}.
We simply replaced $\lambda$ by $\beta^2$ and $X_K$ by $X_K+X_K'$.
We can therefore use many results from Section \ref{s:preuve-penrose-convexe}.
The underlying Galton-Watson process of the second \BRW\ only depends on $\beta^2$. 
This is a key property.
More precisely this is a Poisson$(\beta^2)$ offspring Galton-Watson process.

The plan is to prove that, asymptotically when $d$ tends to $\infty$, if the second \BRW\ survives, which occurs with probability $S(\beta^2)$, 
then $\widehat A^0$ is infinite.

\paragraph{Control of the interference between \BRW.}
The basic plan is the same as in the constant radius case.
In particular, we use a two-step approach
(see Section \ref{s:plan}: make sure that the relevant gaps are good and then using this fact to control interference).
However, the number of $1$-particles of the \BRW\ explodes when $d$ tends to $\infty$ 
and many of them (and their progeny) have to be rejected because of interference.
But most of them have no progeny at all and the number of $\rho$ particles is well behaved when $d$ tends to $\infty$.
The idea is thus to focus as much as possible on $\rho$-particles.

We perform over-pruning (see Section \ref{s:cluster-brw-over}) in order to control rejection by interference.
This means that we reject particles and their progeny if they break one the following rules.
The constants $M$ and $\Lambda$ are large and the constant $\eta$ is small.
The map $\widehat L$ is associated with $X_K$. See Section \ref{s:embedding-GM}.
\begin{enumerate}
\item If $y$ is a children of $x$, then $\|\widehat L(V(y)-V(x))\|_2 \le \Lambda$.
(Recall that this means that we reject $y$ and its progeny if $\|\widehat L(V(y)-V(x))\|_2 > \Lambda$. Similar remarks apply below.)
\item There are no more than $M$ $\rho-$particles at each stage
and each $\rho$-particle has no more than $M \sqrt{\rho}^d$ children (which are $1$-particles).
\item If $x$ and $x'$ and two distinct $\rho$-particle, then
\[
\|{\widehat L}(V(x')-V(x))\|_2 \ge 4\troncature \text{ or } V(x')-V(x) \in G(d,K,\eta) \cap G'(d,K,\eta).
\]
Good gaps are defined in Section \ref{s:good-gaps-GM}.
\item If $z$ is a $\rho$-particle and if $x$ is his grandparent, then
\[
V(z)-V(x) \in H(d,K,\eta) \cap H'(d,K,\eta).
\]
Note that if $y$ is the parent of $z$, the above property implies that the interference region of $y$ is included in 
\[
V(x)+H'(d,K,\eta).
\]
\end{enumerate}
Note that Properties 3 and 4 depends only on $\rho$-particles which are not too numerous.
This is why Lemma \ref{l:good-gap-GM} enables us to prove that the extra-pruning required to get these properties is harmless.
However these properties freely give some further properties on $1$-particles (see for example the comment after the statement of Property 4).

Thanks to the above properties, the probability of interference will be small.
Let us explain the general ideas.
Recall that there can be interference only between a $1$-particle and a $\rho$-particle
and that we want to avoid rejection of $\rho$-particles (and thus of $1$-particles with at least one child).

\begin{itemize}
\item Rejection of a $1$-particle $y$. 
Denote by $x$ its parent. This is a $\rho$-particle. Let $x'$ be another $\rho$-particle.
Let us consider rejection of $y$ because of interference with $x'$.
\begin{itemize}
\item If 
$\|{\widehat L}(V(x')-V(x))\|_2 \ge 2\troncature$,
then $x+K \cap \widehat L ^{-1}(\troncature D)$ and $x'+K \cap \widehat L ^{-1}(\troncature D)$ are disjoint.
But $y$ belongs to $x+K \cap \widehat L ^{-1}(\troncature D)$ and the interference region of $x'$ 
is included in $x'+K \cap \widehat L ^{-1}(\troncature D)$.
Therefore $y$ can not be rejected because of interference with $x'$.
\item Otherwise $V(x)-V(x') \in G(d,K,\eta)$. 
Write $V(y)-V(x')=(V(y)-V(x))+(V(x)-V(x'))$.
By definition of $G(d,K,\eta)$, we see that the probability, condition by everything but $V(y)-V(x)$,
that $y$ is rejected because of interference with $x'$ is at most $\eta$.
\end{itemize}
\item Rejection of a $\rho$-particle $z$. Let $y$ be its parent (it	 is a $1$-particle) and $x$ its grandparent (it is a $\rho$-particle).
Let $x'$ be a $\rho$-particle different from $z$ but which can be $x$.
We are wondering if $z$ can be rejected because of interference with a child $y'$ of $x'$.
The number of $1$-children of a $\rho$-individual diverges as $d$ tends to $\infty$.
Therefore we want, as much as possible, to avoid considering each child $y'$ of $x'$ individually.
In other words we want, as much as possible, to consider only $\rho$-particles.
\begin{itemize}
\item If $\|{\widehat L}(V(x')-V(x))\|_2 \ge 4\troncature$, 
then for any child $y'$ of $x'$, $\|{\widehat L}(V(y')-V(z))\|_2 \ge \troncature$ and therefore $z$ can not be rejected because of $y'$.
\item Otherwise and if $x$ and $x'$ are distinct, then $V(x)-V(x') \in G'(d,K,\eta)$.
Write $V(z)-V(x')=(V(z)-V(y))+(V(y)-V(x)) + (V(x)-V(x'))$.
Recall that the interference region of the $1$-particle $y'$ is included in $V(x')+H'(d,K,\eta)$.
Therefore, by definition of $G'(d,K,\eta)$, the probability, condition to everything but $(V(z)-V(y))$ and $(V(y)-V(x))$, 
that there exists a child $y'$ of $x'$ such that $z$ is rejected because of interference with $y'$ is at most $\eta$.
\item Otherwise, $x=x'$. This is the crucial case where the assumption $\rho<2$ is needed.
We have $V(z)-V(x) \in H(d,K,\eta)$ and thus $V(x)-V(z) \in H(d,K,\eta)$.
Therefore, by definition of $H(d,K,\eta)$, the probability, condition to everything but $V(y')-V(x')$, that a given child $y'$ of $x$ (with $y' \neq y$) 
is responsible of the rejection of $z$ is at most $(\sqrt{2}-\eta)^{-d}$.
But $x$ as no more than $M \sqrt\rho^d$ children. 
Therefore the probability that there exists a child $y'$ of $x'=x$ such that $z$ is rejected because of interference with $y'$ is 
at most $M\sqrt\rho^d(\sqrt{2}-\eta)^{-d}$ which is smaller than $\eta$ for $d$ large enough, 
provided that $\rho < \sqrt{2}-\eta$.
\end{itemize}

\end{itemize}

\subsubsection{Construction of a subset of $\frac 2 {1+\rho}\widehat A^0$ related to an oriented percolation on $\cL$}
\label{s:cluster-perco-orientee-GM}

\paragraph{Choice of parameters.} 
 Fix $\beta > 1$, $2>\rho>1$ and $\epsilon>0$. 
Fix $m,d_1,k,M\ge 1$ with $k$ even as provided by Lemma \ref{l:brique-GM} for the parameters $\beta, \rho$ and $\epsilon$.	
Let $\troncature \ge 1$ be such that $\P[\|\gaussienne\|_2>2^{1/2}\troncature] \le \epsilon/(2M)$.
Fix $\eta \in (0,\sqrt 2)$ such that 
\begin{equation}\label{e:eta-GM}
400k^2\troncature^2M^2 \eta \le \epsilon
\end{equation}
and
\begin{equation}\label{e:intermediaire}
\frac{\sqrt{\rho}}{\sqrt{2}-\eta}<1.
\end{equation}
Recall the definition of $\epsilon'_G(d,\eta)$ in Lemma \ref{l:good-gap-GM}.
Let $d_2$ be such that, for all $d \ge d_2$, 
\begin{equation}\label{e:conditionepsilonG-GM}
400k^2\troncature^2M^2 \epsilon'_G(d,\eta)  \le \epsilon
\end{equation}
and
\begin{equation}\label{e:shining}
M^2\sqrt{\rho}^d(\sqrt{2}-\eta)^{-d} \le \epsilon.
\end{equation}
Let $d_3$ be such that, for all $d \ge d_3$, for all $K \in \cK(d)$, for any $L$ associated with $X_K$ by Theorem \ref{t:Klartag-CLT},
$\P[L(X_K) \neq \gaussienne] \le \epsilon/(2M)$.
Then, with ${\widehat L} = 2^{-1/2}L$, 
\begin{align}
\P[\|{\widehat L} (X_K)\|_2 > \troncature] 
 & \le \P[L(X_K) \neq \gaussienne] + \P[\|\gaussienne\|_2>2^{1/2}\troncature] \nonumber \\
 & \le \epsilon/M. \label{e:troncature-GM}
\end{align}
Let $d_0=\max(d_1,d_2,d_3)$.

\paragraph{Setting and aim.} Let $d \ge d_0$, $K \in \cK(d)$.
As usual, we denote by $X_K$ and $X'_K$ independent random variable with uniform distribution on $K$.
Let $L$ be any map associated with $X_K$ by Theorem \ref{t:Klartag-CLT} for $X_K$.
Define $\widehat L$ by $\widehat L(x)=2^{-1/2}L(x)$ as above.
We aim at proving
\[
\P[\#\widehat A^0=\infty] \ge (\survival(\beta^2)-8\epsilon)\theta(1-\epsilon)
\]
where $\theta$ is defined in Section \ref{s:embedding}.
From this result and Proposition \ref{p:t-GM-convex-upper}, Theorem \ref{t:GM-convex-lower} will follow easily.

Thanks to the choice of parameters, the following properties hold 
(in particular the first ones are due to the choice of $m,d_1,k,M\ge 1$ using Lemma \ref{l:brique-GM}):
\begin{itemize}
\item For all $z \in \widehat A_L(0,0)$, 
\begin{equation}\label{e:applilbriqueGM1}
\P\left[\overline\tau^{\beta,\rho,d,X_K;\{z\}}_k (\widehat A_L (1,0)) \ge m 
\text{ and } \smallrho{\overline\tau^{\beta,\rho,d,X_K;\{z\}}_{\le k}}M\rho \right] \ge \survival(\beta^2)-\epsilon.
\end{equation}
\item For all $(i,j) \in \cL$ and all subset $\cS \subset \widehat A_L(i,j)$ of cardinality $m$,  
\begin{align}
& \P\Big[\overline\tau^{\beta,\rho,d,X_K;\cS}_k (\widehat A_L(i+1,j+1)) \ge m 
\text{ and } \overline\tau^{\beta,\rho,d,X_K;\cS}_k (\widehat A_L(i+1,j-1)) \ge m  \nonumber \\
& \text{ and } \smallrho{\overline\tau^{\beta,\rho,d,X_K;\cS}_{\le k}}M\rho \Big] \ge 1-\epsilon.  \label{e:applilbriqueGM2}
\end{align}
\item  \eqref{e:eta-GM}, \eqref{e:conditionepsilonG-GM}, \eqref{e:shining} and \eqref{e:troncature-GM}.
\end{itemize}

\paragraph{Randomness and $\sigma$-fields.} 
Let
\[
\big( \overline\tau^{i,j,n} )_{(i,j,n) \in \overline\cL \times \{1,\dots,m\}}
\]
be a family of independent copies of $\overline\tau^{\beta,\rho,d,X_K}$. 
Let $(\alpha^{i,j})_{(i,j) \in \overline\cL}$ be a family of i.i.d.\ Bernoulli random variables with parameter $1-\epsilon$.
For all $(i,j) \in \overline{\cL}$ we denote by $\cF_{i,j}$ (resp. $\cF_{i,j}^-$) the $\sigma$-field generated by 
the $\overline\tau^{i',j',n}$ and the $\alpha^{i',j'}$ for $(i',j',n) \in \overline\cL \times \{1,\dots,m\}$ such that $(i',j')$ is smaller 
(resp. strictly smaller) than $(i,j)$
for the lexicographic order.
Initially, the site $(0,0)$ is active and all other sites are inactive.

\paragraph{Site $(0,0)$.} 
This stage is slightly different and slightly less involved than the next stages.
The difference is similar to the difference in the corresponding construction given in Section \ref{s:preuve-penrose-convexe}.
Here, in order to avoid lengthy repetitions, 
we only give the construction for sites $(i,j) \in \cL$ and quickly explain by footnotes the modifications needed for the site $(0,0)$.
We hope that this is clear but, if this is not the case, we refer the reader to Section \ref{s:preuve-penrose-convexe}.

\paragraph{Stage $(i,j)$ for $(i,j) \in \cL$.} 
Recall that we now consider successively each $(i,j)\in\cL$ by lexicographic order.
If $(i,j)$ is inactive we decide, independently of everything else, that it is open with probability $1-\epsilon$ and closed otherwise. 
More precisely, we decide that $(i,j)$ is open when $\alpha^{i,j}=1$ and closed otherwise.

Thereafter, we consider the case where $(i,j)$ is active. 
The set $\cS(i,j)$ is well defined.
It's a subset of cardinal $m$ of $A_{\widehat L}(i,j)$.    
List the points of $\cS(i,j)$ in an arbitrary order: $\cS(i,j)=\{x^1,\dots,x^m\}$.
We consider the $m$ \BRW\ $\overline\tau^n=x^n+\overline\tau^{i,j,n}$ 
where $x^n+\overline\tau^{i,j,n}$ designates the \BRW\ $\overline\tau^{i,j,n}$ in which $x^n$ has been added to the position of all the particles.
We gather these $m$ \BRW\ into a single \BRW\ originating from $\cS(i,j)$.
We denote it by $\overline\tau^{\cS(i,j)}$
\footnote{In Stage $(0,0)$, $(0,0)$ is always active and we simply consider the \BRW\ $\overline\tau^{0,0,1}$ originating from $\{0\}$.}.
\begin{enumerate}
\item We examine successively the particles of generation between $1$
\footnote{In Stage $(0,0)$ we also examine the particle of generation $0$.} and $k$ of $\overline\tau^{\cS(i,j)}$ in any admissible order.
The last requirement means that:
\begin{itemize}
\item Children are examined after their parents.
\item Children of a given parent are examined in a row: 
once we start examining the children of one parent, we then examine all the children of this parent.
\end{itemize}
Some of the particles will be rejected.
We will reject more particles than necessary, thus performing some over-pruning as explained in Section \ref{s:cluster-BRW-GM}.
Some particles can be rejected before examination (if we have examined and rejected one of its ancestor).
If we have examined a particle and not rejected it, we say that the particle has been "generated".
If we are examining a particle $x$ and talk about a particle $y$ generated before $x$, it means that one of the following two properties occur:
\begin{itemize}
\item $y$ has been examined at stage $(i,j)$ strictly before $x$ and $y$ has not been rejected.
\item $y$ has been examined at a previous stage $(i',j')$ (thus $(i',j')<(i,j)$ in lexicographical order) and $y$ has not been rejected
\footnote{This property never occurs if $(i,j)=0$.}.
\end{itemize}

We never examine the roots of this \BRW\ 
(note that the positions of the roots are also positions of particles of \BRW\ examined during one of the previous stages
\footnote{In Stage $(0,0)$ we examine the root, even if there is nothing to examine. This is just a convention which ensures that the root 
is generated.}).
We stop as soon as we have examined all the particles or as soon as the particle examined causes a stop by "Overpopulation"
(equivalently, if "Overpopulation" occurs, we reject all the remaining particles).
When examining a particle $x$, several things can occur:

\begin{enumerate}
\item Overpopulation. This occurs if $x$ is the $M$-th $\rho$-particle examined during this stage $(i,j)$
or if $x$ is the $M \sqrt\rho^d$-th children examined of a given $\rho$-particle.
If "Overpopulation" occurs, we reject $x$ and all the particles of $\overline\tau^{\cS(i,j)}$ which have not been examined yet.
This stops the examination.
\item Bad gap. This occurs if one of the following three conditions occurs:
\begin{itemize}
\item $x$ is a $\rho$-particle and there exists a $\rho$-particle $x'$ generated before $x$ such that: 
\[
\|{{\widehat L}}(V(x')-V(x))\|_2 < 4\troncature \text{ and } V(x)-V(x') \not\in G(d,K,\eta) \cap G'(d,K,\eta).
\]
\item $x$ is a $\rho$-particle and the grand parent $x'$ of $x$\footnote{In Stage $(0,0)$ the condition is empty for the root $0$.} is such that
\[
V(x)-V(x') \not\in H(d,K,\eta) \cap H'(d,K,\eta).
\]
\item The following condition holds\footnote{In Stage $(0,0)$ the condition is empty for the root $0$.}:
\[
\|{{\widehat L}}(V(x)-V(\parent{x}))\|_2 > \troncature.
\]
\end{itemize}
In this case we do {\em not} stop the examination.
However, we reject $x$ and all its progeny
\footnote{In Section \ref{s:preuve-penrose-convexe} we stopped the examination as soon as we saw a "Bad Gap".
Here we can not avoid some particles to be rejected because of "Bad-Gap".
However, as we will see, we can avoid (with high probability) rejection of $\rho$-particle because of "Bad Gap"
(note that we do note care about rejection of $1$-particles without children).
It will be sufficient.} and we do not consider rejection of $x$ by interference (see below).

This over-pruning will help us controlling the interference. 
Let us note a few consequences right now, as this is necessary to explain the paragraph "interference" below.
\begin{itemize}
\item If $x'$ and $y'$ are two generated particles such that $x'$ is the parent of $y'$, then $\|\widehat L(V(y')-V(x))\|_2 \le \troncature$.
\item The interference region (see \eqref{e:regino_interference}) of a generated particle $y'$ is included in 
\[
V(y') + K \cap {{\widehat L}}^{-1}(\troncature D)
\]
where $D$ is the unit disk of $\R^2$.
If moreover $y'$ is a $1$-particle with parent $x'$, then the interference region of $y'$ is included in
\begin{equation}\label{e:interference-region-1}
V(x') + H'(d,K,\eta).
\end{equation}
\end{itemize}
\item Interference.
It occurs if there exists a particle $y'$ generated before $x$, such that $y'$ is not the parent $\parent x$ of $x$ and such that
one of the following conditions occurs (see the remark on the interference region in the paragraph "Bad Gap"):
\begin{itemize} 
\item $x$ is a $1$-particle, $y'$ is a $\rho$-particle and
\[
V(x) \in V(y')+K \cap {{\widehat L}}^{-1}(\troncature D).
\]
\item $x$ is a $\rho$-particle, $y'$ is a $1$-particle and
\[
V(x) \in V(y')+K \cap {{\widehat L}}^{-1}(\troncature D) \text{ and } V(x) \in V(\parent{y'}) + H'(d,K,\eta).
\]
\end{itemize}
In this case we do {\em not} stop the examination.
However, we reject the particle $x$ and all its progeny
\footnote{This is again a difference with respect to Section 3. The reason is the same as for "Bad Gap".}.
\end{enumerate}

One defines the following  sets.

\begin{itemize}
\item The set $\cG(i,j)$ of positions of all particles examined and not rejected.
We call them the particles generated at stage $(i,j)$.
By an abuse of notation, we see $\cG(i,j)$ as inheriting the genealogical structure of the \BRW.
\begin{itemize}
\item In the coupling with the Boolean model, all the points of $\cG(i,j)$ belong to $\frac2{1+\rho}A^0$.
Indeed, none of them was rejected (see also the definition and properties of the seeds $\cS(\cdot,\cdot)$ below\footnote{In stage $(0,0)$,
the set of seeds is $\{0\}$ which belong to $\frac2{1+\rho}A^0$.}).
\item 
\begin{equation}\label{e:besoindelabelGM}
\text{The image by }\widehat L\text{ of }\cG(i,j)\text{ is included in }(i,j)+2k\troncature D.
\end{equation} 
This is due to the fact that the images by $\widehat L$ of the positions of the roots belong to $(i,j)+D$
(see definition of $\cS(\cdot,\cdot)$ below\footnote{In stage $(0,0)$, the root is located at $0$.}),
the fact that the image by ${\widehat L}$ of each step leading to a non rejected particle belongs to $\troncature D$, 
the fact that we only explored the first $k$ generations and the fact that $1+k\troncature \le 2k\troncature$.
\end{itemize}
\item The subset $\cG^\rho(i,j) \subset \cG(i,j)$ of positions of $\rho$-particles generated
and the subset $\cG^1(i,j) \subset \cG(i,j)$ of positions of $1$-particles generated.
\begin{itemize}
\item All the gap between two distinct points $x$ and $x'$ of
\[
\bigcup_{(i',j') \le (i,j)} \cG^\rho(i,j)
\]
satisfies
\begin{equation}\label{e:hoho-GM}
\|{\widehat L}(V(x))-{\widehat L}(V(x'))\|_2 \ge 4\troncature \text{ or } V(x) \in V(x')+G(d,K,\eta)\cap G'(d,K,\eta).
\end{equation}
\item If moreover $x'$ is the grand-parent of $x$, then
\begin{equation}\label{e:hoho2-GM}
V(x)-V(x') \in H(d,K,\eta) \cap H'(d,K,\eta).
\end{equation}
\item $\cG^\rho(i,j)$ contains at most $M$ points.
\item Any point in $\cG^\rho(i,j)$ has at most $M\sqrt{\rho}^d$ children in $\cG^1(i,j)$.
\end{itemize}
Indeed, none of them was rejected because of "Bad Gap" or "Overpopulation".
\item The set $\cG_k(i,j) \subset \cG^\rho(i,j)$ of positions of generated particles of generation $k$.
Note that since $k$ is even and since we started with $\rho$-particles, the particles of $\cG_k(i,j)$ are $\rho$-particles.
\end{itemize}
\item If
\begin{equation} \label{e:okpassuivant-GM}
\#\big(\cG_k(i,j) \cap A_{\widehat L}(i+1,j+1)\big) \ge m \text{ and } \#\big(\cG_k(i,j)	 \cap A_{\widehat L}(i+1,j-1)\big) \ge m
\end{equation}
then\footnote{In Stage $(0,0)$, the condition is
\begin{equation} \label{e:okpas-GM}
\#\big(\cG_k(0,0) \cap A_{\widehat L}(1,0)\big) \ge m.
\end{equation}.}
:
\begin{itemize}
\item We say that the site $(i,j)$ is open.
\item If the site $(i+1,j+1)$ is inactive then we say that it is henceforth active and we define $\cS(i+1,j+1)$
as the $m$ first particles of $\cG_k(i,j) \cap A_{\widehat L}(i+1,j+1)$ in an arbitrary order.
In the coupling with the Boolean model, all the points of $\cS(i+1,j+1)$ belong to $\frac2{1+\rho}\widehat A^0$
\footnote{In Stage $(0,0)$, we say that the site $(1,0)$ is active and we define $\cS(1,0)$
as the $m$ first particles of $\cG_k(0,0) \cap A_{\widehat L}(1,0)$ in an arbitrary order.
In the coupling with the Boolean model, all the points of $\cS(1,0)$ belong to $\frac2{1+\rho}\widehat A^0$.}.
\item If the site $(i+1,j-1)$ is inactive then we say that it is henceforth active and we define $\cS(i+1,j-1)$
as the $m$ first particles of $\cG_k(i,j) \cap A_{\widehat L}(i+1,j-1)$ in an arbitrary order.
In the coupling with the Boolean model, all the points of $\cS(i+1,j-1)$ belong to $\frac2{1+\rho}\widehat A^0$
\footnote{In Stage $(0,0)$ this item does not exist.}.
\end{itemize}
Otherwise, we say that the site $(i,j)$ is closed.
\end{enumerate}

\subsubsection{Bounds on conditional probabilities}
\label{s:proba-briques-GM}

The aim of this section is to prove the following lemmas.
Parameters have been fixed in Section \ref{s:cluster-perco-orientee-GM}.

\begin{lemma} \label{l:mesurabilite-GM}
For all $(i,j) \in \overline{\cL}$, 
$
\{ (i,j) \text{ active}\} \in \cF(i,j)^- \text{ and } \{ (i,j) \text{ open}\}\in \cF(i,j).
$
\end{lemma}

\begin{lemma} \label{l:comparaisoninitiale-GM}
We have
$
\P[(0,0) \text{ open}] \ge \survival(\beta^2)-8\epsilon.
$
\end{lemma}

\begin{lemma} \label{l:comparaisonsuite-GM}
For all $(i,j) \in \cL$,
$
\P[(i,j) \text{ open }| \cF^-(i,j)] \ge 1-8\epsilon.
$
\end{lemma}

\begin{proof}[Proof of Lemma \ref{l:mesurabilite-GM}] This is straightforward by construction.
\end{proof}

The proof of Lemma \ref{l:comparaisoninitiale-GM} is less involved than the proof of Lemma \ref{l:comparaisonsuite-GM}.
The difference is similar to the difference in the proofs of Lemmas \ref{l:comparaisoninitiale} and \ref{l:comparaisonsuite}.
There are moreover similarities in the proofs of the four lemmas.
Therefore, in order to avoid lengthy repetitions, 
we only give a somewhat sketchy proof of Lemma \ref{l:comparaisonsuite-GM} 
and quickly explain in footnotes the modifications needed for the proof of Lemma \ref{l:comparaisoninitiale-GM}.
For more details, we refer the reader to the proofs of Lemmas \ref{l:comparaisoninitiale} (which is the more detailed) and Lemma \ref{l:comparaisonsuite}.

\begin{proof}[Proof of Lemma \ref{l:comparaisonsuite-GM}] Let $(i,j) \in \cL$.
The event $\{(i,j) \text{ is active}\}$ is $\cF^-(i,j)$ measurable.
Therefore, we have to show
\[
\P[(i,j) \text{ open }| \cF^-(i,j)] \ge 1-\epsilon \text{ on the event } \{(i,j) \text{ is active}\}
\]
and
\[
\P[(i,j) \text{ open }| \cF^-(i,j)] \ge 1-\epsilon \text{ on the event } \{(i,j) \text{ is inactive}\}
\]
The second property is straightforward. Indeed, when $(i,j)$ is inactive, 
$(i,j)$ has been defined as open independently of everything else with probability $1-\epsilon$.
Let us prove the first property.

Below, we implicitly work on the event $\{(i,j) \text{ is active}\}$ and probabilities are always conditional to $\cF^-(i,j)$.
Therefore, there is a well defined set $\cS(i,j)$ with cardinality $m$ 
whose points are the starting points of $m$ \BRW.
This set is measurable with respect to $\cF^-(i,j)$.
We also have the \BRW\ $\overline\tau^{\cS(i,j)}$ which has been used in Stage $(i,j)$
\footnote{In the proof of Lemma \ref{l:comparaisoninitiale-GM}, there is no need to discuss about conditional probabilities.
Moreover, the \BRW\ is simply $\overline\tau^{0,0,1}$.}.
Moreover
\[
\{(i,j) \text{ is closed } \}  
\subset \rhobad \cup \rhorejection \cup {\smallrho{\overline\tau^{\cS(i,j)}}{M}{\rho}}^c \cup \autre
\]
where
%\begin{align*}
%\notrhonice & = \{\text{the enumeration stops because of Overpopulation}\},\\
%\rhobad & = \{\text{the enumeration did not stop because of Overpopulation }\\
%&\text{ but a }\rho\text{-particle has been rejected directly or indirectly because of Bad Gap}\}, \\
%\rhorejection & = \{\text{the enumeration did not stop because of Overpopulation }\\
%&\text{ but a }\rho\text{-particle has been rejected directly or indirectly because of Rejection}\}, \\
%\end{align*}
\begin{align*}
\rhobad & = \{\text{a }\rho\text{-particle has been rejected directly or indirectly because of Bad Gap}\} 
\cap \smallrho{\overline\tau^{\cS(i,j)}}{M}{\rho}, \\
\rhorejection & = \{\text{a }\rho\text{-particle has been rejected directly or indirectly because of Interference}\} 
\cap \smallrho{\overline\tau^{\cS(i,j)}}{M}{\rho} 
\end{align*}
and where\footnote{In the proof of Lemma \ref{l:comparaisoninitiale-GM},
\[
\autre = \{ \overline\tau^{0,0,1}_k(A_{\widehat L}(1,0)) < m \}.
\]}
\[
\autre = \{ \overline\tau_k^{\cS(i,j)}(A_{\widehat L}(i+1,j+1)) < m \text{ or } \overline\tau_k^{\cS(i,j)}(A_{\widehat L}(i+1,j-1)) < m\}.
\]
By "a $\rho$-particle $x$ has been rejected directly or indirectly" we means that $x$ has been rejected or that one if its ancestors has been rejected.
The inclusion is due 
to the fact that, if $(i,j)$ is open but $\rhobad \cup \rhorejection \cup {\smallrho{\overline\tau^{\cS(i,j)}}{M}{\rho}}^c$
does not occur, then $\cG_k(i,j)$ is the set of positions of the particles of $\overline\tau_k^{\cS(i,j)}$ (which are $\rho$ particles).

Let us provide an upper bound for the probability of the event $\rhobad$.
Denote by $x$ the first particle (in the order of enumeration) that causes $\rhobad$.
This is either a $\rho$-particle rejected because of Bad Gap 
or a $1$-particle with at least one child in $\overline\tau^{\cS(i,j)}_k$ (in other word, a parent in $\overline\tau^{\cS(i,j)}_k$ of a $\rho$ particle)
rejected because of Bad Gap.
On $\smallrho{\overline\tau^{\cS(i,j)}}{M}{\rho}$ there are at most $M$ $\rho$-particles and therefore at most $M$ $1$-particles with at least one child.
This gives us a bound on the number of particles that can be the cause of $\rhobad$.

We further distinguish according to the three different types of $\rhobad$.
\begin{enumerate}
\item Suppose that $x$ satisfies the third condition of rejection by Bad Gap.
There are at most $2M$ choices for $x$ (which is either a $\rho$-particle or a parent of a $\rho$-particle).
By \eqref{e:troncature-GM} one deduces (by first conditioning by everything but $V(x)-V(\parent x)$)
that the probability that $\rhobad$ occurs because of the third type of Bad Gap is at most
\[
2M\epsilon/M=2\epsilon.
\]
If what follows we assume that $x$ does not satisfy the third condition of rejection by Bad Gap.
In particular, arguing as in \eqref{e:besoindelabelGM}, 
\begin{equation}\label{e:ahbentiens}
\widehat L (V(x)) \in (i,j) + 2\Lambda k D.
\end{equation}
\item Consider the first case.
There at most $M$ choices for the $\rho$-particle $x$.
As $\|{\widehat L}(V(x')-V(x))\|_2 \le 4\troncature$, there are at most $200 k^2\troncature^2 M$ choices for $x'$.
Let us prove this fact.
Either $x'$ is one of the at most $M$ $\rho$-particles generated at stage $(i,j)$ 
or $x'$ is one of the at most $M$ $\rho$-particles generated at an earlier stage $(i',j')$.
In the latter case, $\widehat L(V(x)) \in (i,j)+2k\troncature D$ (this is \eqref{e:ahbentiens})
and $\widehat L(V(x')) \in (i',j')+2k\troncature D$ (see \eqref{e:besoindelabelGM}).
As $\|{\widehat L}(V(x')-V(x))\|_2 \le 4\troncature$ this implies 
\[
\|(i,j)-(i',j')\|_2  \le (4k+4)\troncature \le 6k\troncature.
\]
(Recall that $k$ is even, so $k \ge 2$).
Therefore there are at most $(12k\troncature+1)^2$ choices for $(i',j')$.
Combining the previous properties, the number of choice for $x'$ is at most
\[
M+M(12k\troncature+1)^2 \le 200 k^2\troncature^2 M
\]
as announced. We used $k,\troncature \ge 1$ to simplify the upper bound\footnote{In the proof 
of Lemma \ref{l:comparaisoninitiale-GM}, there are actually at most $M$ choices for $x'$. 
There are therefore also at most $200 k^2\troncature^2 M$ choices for $x'$. Similar remarks apply several times below.}.
Finally, the number of choices for $(x,x')$  is at most $200 k^2\troncature^2 M^2$.

By Lemma \ref{l:good-gap-GM} one deduces (by first conditioning by everything but $V(x)-V(\parent x)$) that the probability that $\rhobad$
occurs because of the first type of Bad Gap is at most
\[
400 k^2\troncature^2 M^2 \epsilon'_G(d,\eta).
\]
By \eqref{e:conditionepsilonG-GM}, this is at most $\epsilon$.

\item Consider the second case. There are at most $M$ choice for the $\rho$-particle $x$. Then $x'$ is the grandparent of $x$.
By Lemma \ref{l:good-gap-GM} one deduces (by first conditioning by everything but $V(x)-V(\parent x)$ and $V(\parent x)-V(x')$) 
that the probability that $\rhobad$ occurs because of the second type of Bad Gap is at most
\[
2M\epsilon'_G(d,\eta).
\]
By \eqref{e:conditionepsilonG-GM}, as $k,\troncature,M \ge 1$, this is at most $\epsilon$.

\end{enumerate}
Finally,
\begin{equation}\label{e:majorationPbadGM}
\P[\rhobad] \le 4\epsilon.
\end{equation}

Let us now give an upper bound for the probability of the event $\rhorejection$.
Denote by $x$ the first particle (in the order of enumeration) that causes $\rhorejection$.
This is either a $\rho$-particle rejected because of Interference 
or a $1$-particle with at least one child in $\overline\tau^{\cS(i,j)}_k$ rejected because of Interference.
On $\smallrho{\overline\tau^{\cS(i,j)}}{M}{\rho}$ there are at most $M$ $\rho$-particles and therefore at most $M$ $1$-particles with at least one child.
We distinguish the two cases of rejection by interference.

\begin{enumerate}
\item In the first case, a $1$-particle $x$ is rejected because of a $\rho$-particle $y'$ which is not the parent $\parent x$ of $x$.
Note that $x$ has not been rejected because of Bad Gap, otherwise we would not consider $x$ for rejection by interference.

There are at most $M$ choices for the $1$-particle $x$ (as it has as least one child which is a $\rho$-particle).
The $\rho$-particle $y'$ is such that $V(x) \in V(y') + K \cap {{\widehat L}}^{-1}(\troncature D)$.
Therefore $\|{\widehat L}(V(x)-V(y'))\|_2 \le \troncature$.
As above, we conclude that there are at most $200 k^2\troncature^2M$ choices for $y'$ (this is a crude bound).
Therefore, there are at most $200 k^2\troncature^2M^2$ choices for $(x,y')$.

We have $\|{\widehat L}(V(\parent{x})-V(x))\|_2 \le \troncature$ (otherwise $x$ would have been rejected for Bad Gap)
and therefore $\|{\widehat L}(V(\parent{x})-V(y'))\|_2 \le 2\troncature$.
As a consequence $V(\parent{x})-V(y') \in G(d,K,\eta)$ (otherwise $y'$ or $\parent{x}$, which are distinct 
$\rho$-particles, would have been rejected for Bad Gap).
Write $V(x)-V(y')=(V(\parent{x})-V(y'))+(V(x)-V(\parent{x}))$.
By definition of $G(d,K,\eta)$, condition to everything but $V(x)-V(\parent{x})$, the probability 
that $x$ is rejected because of interference with $y'$ is at most $\eta$.

Therefore, the probability that $\rhorejection$ occurs because of the first type of Rejection is at most
\[
200 k^2\troncature^2M^2\eta.
\]
By \eqref{e:eta-GM}, this is at most $\epsilon$.

\item In the second case, $\rhorejection$ is caused by the rejection of 
a $\rho$-particle $z$ --	 with parent $y$ and grand parent $x$  --
due to interference with a $1$-particle $y'$ -- with parent $x'$ -- where $y'$ is different from $y$.
Note that $x$ and $x'$ are $\rho$-particles, that $y$ and $y'$ are $1$-particles and that $z$ is a $\rho$-particle.

As $z$ is rejected because of interference with $y'$, at some point in the examination process:
\begin{itemize}
\item $z$ is being examined and has not been rejected for Bad Gap (otherwise we would not even consider rejection of $z$ due to interference).
\item $y'$ has been examined before and has not been rejected (otherwise it would not be a generated particle and we would not consider the 
interference caused by $y'$).
\end{itemize}
This also implies that $y,x$ and $x'$ have been examined (parents are examined before their children)
and have not been rejected (because when a particle is rejected, its progeny is rejected).
As we examine in an admissible order, 
$x'$ can not be a progeny of $y$ (because we already examined the child $y'$ of $x'$ and we are currently examining $z$)
or $z$.
As $z$ is a grandchild of $x$, $z$ is different from $x$.
To sum up, among all particles $x,y,z,x',y'$ the only possible equality is $x=x'$.

We further distinguish according whether $x=x'$ or not.
\begin{enumerate}
\item Case $x \neq x'$.  

As $z$ is a grand-child of $x$ we have $\|{\widehat L}(V(z)-V(x))\|_2 \le 2 \troncature$
(otherwise $y$ or $z$ would have been rejected for Bad Gap).
As $z$ interferes with $y'$ we have $\|{\widehat L}(V(z)-V(y'))\|_2 \le \troncature$.
As $y'$ is a child of $x'$ we have $\|{\widehat L}(V(y')-V(x'))\|_2 \le \troncature$ 
(otherwise $y'$ would have been rejected for Bad Gap).
Therefore 
\begin{equation}\label{e:ile}
\|\widehat L(V(z)-V(x'))\|_2 \le 2\troncature
\end{equation}
and
$\|{\widehat L}(V(x')-V(x))\|_2 \le 4\troncature$.
As the $\rho$-particles $x$ and $x'$ have not been rejected by Bad Gap, 
\begin{equation}\label{e:du}
V(x)-V(x') \in G'(d,K,\eta).
\end{equation}
This is where we use the assumption $x \neq x'$.
Note that rejection of $z$ because of $y'$ implies (see \eqref{e:interference-region-1})
\begin{equation} \label{e:crane}
V(z)-V(x') \in H'(d,K,\eta).
\end{equation}
We thus see that if $\rhorejection$ occurs because of this type of case, then there exists a $\rho$-particle $z$ with grandparent $x$
and a $\rho$-particle $x'$ distinct of $x$ -- which is not a progeny of $y$ or $z$ -- such that \eqref{e:ile}, \eqref{e:du} and \eqref{e:crane} occurs.
Note that there is no mention of $y'$ anymore.
As $z$ and $x'$ are $\rho$-particles, by \eqref{e:ile} and arguments already used above, 
there are at most $200 k^2\troncature^2M^2$ choices for $(z,x')$.

Write $V(z)-V(x')=(V(z)-V(y))+(V(y)-V(x))+(V(x)-V(x'))$.
By definition of $G'(d,K,\eta)$ and $H'(d,H,\eta)$,
for each $(z,x')$, condition by everything but $V(z)-V(y)$ and $V(y)-V(x)$, on the event $\{\text{\eqref{e:du}} \text{ occurs}\}$
(which is measurable with respect to the conditioning $\sigma$-field because $x'$ is not a progeny of $y$),
the probability of \eqref{e:crane} is at most $\eta$.
To sum up, the probability that $\rhorejection$ occurs because of this type of case is at most 
\[
200 k^2\troncature^2M^2\eta.
\]
By \eqref{e:eta-GM}, this is at most $\epsilon$.

\item Case $x = x'$.  
There are at most $M$ choices for the $\rho$-particle $z$.
As the enumeration has not been stopped by Overpopulation, given $z$, there are at most $M\sqrt{\rho}^d$ choices for $y'$ which is a children of 
the grandfather $x$ of $z$.
Thus, there are at most $M^2\sqrt{\rho}^d$ choices for $(z,y')$.
As $z$ has not been rejected because of Bad Gap,
\begin{equation}\label{e:patate}
V(x) - V(z) \in  H(d,K,\eta).
\end{equation}
%Therefore, by definition of $H(d,H,\eta)$, given $z$ and $y'$, condition on everything but $V(y')-V(x)$,
As $z$ is rejected by interference with $y'$, $V(z) \in V(y')+K$ and therefore (using also symmetry of $K$)
\begin{equation}\label{e:frite}
(V(x)-V(z)) + (V(y')-V(x)) \in K.
\end{equation}
But by definition of $H(d,H,\eta)$, given $z$ and $y'$, condition to everything but $V(y')-V(x)$, on 
$\{\text{\eqref{e:patate} occurs}\}$, the probability of 
\eqref{e:frite} is at most $(\sqrt{2}-\eta)^{-d}$.
Therefore the probability that $\rhorejection$ occurs because of this type of case is at most 
\[
M^2\sqrt{\rho}^d(\sqrt{2}-\eta)^{-d}.
\]
By \eqref{e:shining} this is at most $\epsilon$.
\end{enumerate}
\end{enumerate}
Finally,
\begin{equation}\label{e:majorationPinterferenceGM}
\P[\rhorejection] \le 3\epsilon.
\end{equation}

By \eqref{e:applilbriqueGM2}\footnote{In the proof of Lemma \ref{l:comparaisoninitiale-GM}, we apply \eqref{e:applilbriqueGM1}
to get
\[
\P\left[{\smallrho{\overline\tau^{\cS(i,j)}}{M}{\rho}}^c \cup \autre\right] \le 1-(\survival(\beta^2)-\epsilon).
\]
}
\[
\P\left[{\smallrho{\overline\tau^{\cS(i,j)}}{M}{\rho}}^c \cup \autre\right] \le \epsilon.
\]
This concludes the proof.
\end{proof}

\subsubsection{Proof of Theorem \ref{t:GM-convex-lower}}
\label{s:preuve-t-GM-convex-lower}

\paragraph{Proof of Item 1 of Theorem \ref{t:GM-convex-lower}.}

In Section \ref{s:cluster-perco-orientee-GM} we fixed $\beta>1$, $2>\rho>1$ and $\epsilon>0$.
We then got some integer $d_0$ (depending only on $\beta,\rho$ and $\epsilon$) and several other parameters satisfying various properties.
We then let $d \ge d_0$ and $K \in \cK(d)$ and built some process in Section \ref{s:cluster-perco-orientee-GM}.
We studied some properties of this process in Section \ref{s:proba-briques-GM}.

As in Section \ref{s:preuve-t-penrose-convex-lower}, we have

\[
\P[\#\widehat A^0=\infty] \ge \P[\text{there exists an infinite open path from }(0,0)].
\]
By Lemmas \ref{l:mesurabilite-GM}, \ref{l:comparaisoninitiale-GM} and \ref{l:comparaisonsuite-GM} we get
\[
\P[\text{there exists an infinite open path from }(0,0)] \ge (\survival(\lambda)-8\epsilon)\theta(1-8\epsilon)
\]
where $\theta(1-8\epsilon)$ is the probability that there exists an infinite open path originating from $(1,0)$ in a Bernoulli site percolation
on $\cL$ with parameter $1-8\epsilon$.
Therefore
\[
\P[\#\widehat C^0=\infty] \ge \P[\#\widehat A^0=\infty] \ge (\survival(\beta^2)-8\epsilon)\theta(1-8\epsilon).
\]
But $\theta(1-8\epsilon)$ tends to $1$ as $\epsilon$ tends to $0$ (see \eqref{e:percolimite}).
This proves Item 1 of Theorem \ref{t:GM-convex-lower} in the case $\beta > 1$. 
The case $\beta \le 1$ is trivial as, in this case, $\survival(\beta^2)=0$.

\paragraph{Proof of Item 2 of Theorem \ref{t:GM-convex-lower}.}

Let $\beta>1$, $2>\rho>1$.
Let $\epsilon>0$ such that $\survival(\beta^2)-\epsilon>0$.
By Item 1, there exists $d_0$ such that, for all $d \ge d_0$ and all $K \in \cK(d)$,
\[
\P[\#\widehat C^0=\infty] \ge \survival(\beta^2)-\epsilon > 0.
\]
Therefore, for all such $\rho, d$ and $K$,
$
\beta_c(\rho,d,K) \le \beta.
$

\appendix

\section{Some details on the definition of $\lambda_c$}
\label{s:detail-lambda_c}
	
All the results in this section are very standard and simple, but we have no ready reference for them.
%Recall the notations
%\begin{align*}
%\lambda_c(\nu,d,K) 
% & = \inf \{\lambda : \P[\text{the connected component of the graph }\chi^0\text{ that contains }0\text{ is unbounded}]>0\} \\
% & = \inf \{\lambda : \P[\text{the connected component of }\Sigma \cup rK \text{ that contains }0\text{ is unbounded}]>0\} \\
% & = \inf \{\lambda : \P[\text{one of the connected components of }\Sigma\text{ is unbounded}]=1\}.
%\end{align*}
We provide a few technical details related to the definition of $\lambda_c$.
We refer to the notations used in \eqref{e:lambda_c_graphe}, \eqref{e:lambda_c_composante} and \eqref{e:lambda_c_01}.
%Let $\lambda>0$. 

\paragraph{Measurability of $\{\text{the connected component of the graph }\chi^0\text{ that contains }0\text{ is unbounded}\}$.}
This is a consequence of the following facts:
\begin{itemize}
\item There exists a sequence of random variables $(C_n,R_n)_n$ such that, on the full event $\{\xi \text{ is infinite}\}$,
\[
\xi=\{(C_n,R_n), n \in \N\}.
\]
In other words, $\chi=\{(C_n), n \in \N\}$ and for all $n$, $r(C_n)=R_n$.
\item The map defined by $(c_1,r_1,c_2,r_2) \mapsto \1_{\{c_1+r_1K \cap c_2+r_2K \neq \emptyset\}}$ is measurable because
\[
\1_{\{c_1+r_1K \cap c_2+r_2K \neq \emptyset\}} = \1_K\left(\frac{1}{r_1+r_2}(c_2-c_1)\right).
\]
The above equality is a consequence of $r_2K-r_1K=(r_2+r_1)K$ which, in turn, is a consequence of the fact that $K$ is symmetric and convex.
\end{itemize} 

\paragraph{Equality \eqref{e:lambda_c_composante}.} This is a consequence of 
\begin{align}
& \{\text{the connected component of the graph }\chi^0\text{ that contains }0\text{ is unbounded}\} \nonumber \\ 
&=\{\text{the connected component of }\Sigma \cup rK \text{ that contains }0\text{ is unbounded}\}. \label{e:pouf}
\end{align}
Let us prove this equality. 
Let $\widehat C^0$ be the connected component of the origin in the graph.
For all $x,y \in \chi^0$, there is an edge between $x$ and $y$ if and only if $x+r(x)K$ touches $y + r(y)K$.
For all $x \in \chi^0$, $x+r(x)K$ is connected. 
Using these two facts, we get that
\begin{equation}\label{e:pern}
\bigcup_{x \in {\widehat C^0}} x+r(x)K \text{ is connected}
\end{equation}
and that 
\begin{equation}\label{e:oupaspern}
\bigcup_{x \in {\widehat C^0}} x+r(x)K \text{ and } \bigcup_{x \in {\chi^0 \setminus \widehat C^0}} x+r(x)K 
\text{ are disjoint}.
\end{equation}
For all $x \in \chi^0$, $x+r(x)K$ is compact. 
The number of grains that touches a given bounded region is finite.
This is a simple consequence of the fact that $\nu$ has bounded support, that $K$ is bounded that $\chi^0$ is locally finite.
Using these two facts, we get that the two sets appearing in \eqref{e:oupaspern} are closed subsets of $\R^d$ and therefore of $\Sigma \cup rK$.

Using \eqref{e:pern}, \eqref{e:oupaspern} and the closedness of the sets 
we get that the set which appears in \eqref{e:pern} is the connected component of $\Sigma \cup rK$ that contains the origin.
This yields \eqref{e:pouf} as $K$ is bounded.

\paragraph{Equality \eqref{e:lambda_c_01}.} With the same ideas as in the previous paragraph,
we can check that "one of the connected component of $\Sigma$ is unbounded"
if and only if "one of the connected component of the graph $\chi$ is unbounded". 
Then, with the same ideas as in the first paragraph, we can check that this is measurable.

By ergodicity under spatial translations of the model, the probability of the 
event "one of the connected component of $\Sigma$ is unbounded" is $0$ or $1$.
Define
\begin{align*}
\Lambda_0 & = \{\lambda>0 : \P[\text{the connected component of }\Sigma\text{ containing }0\text{ is unbouded}]>0\}, \\
\Lambda_r & = \{\lambda>0 : \P[\text{the connected component of }\Sigma \cup rK \text{ containing }0\text{ is unbouded}]>0\}, \\
\Lambda' & = \{\lambda>0 : \P[\text{one of the connected components of }\Sigma\text{ is unbouded}]>0\}, \\
\Lambda & = \{\lambda>0 : \P[\text{one of the connected components of }\Sigma\text{ is unbouded}]=1\}.
\end{align*}
Let us show that the four sets are equal.
As $K$ is connected with non empty interior, one of the connected component of $\Sigma$ is unbounded
if and only if there exists $x \in \Q^d$ such that the connected component of $\Sigma$ containing $x$ is unbounded.
Thus $\Lambda \subset \Lambda_0$.
The inclusion $\Lambda_0 \subset \Lambda_r$ is straightforward.
The connected component of $\Sigma\cup rK$ that contains the origin is \eqref{e:pern}.
The connected components of $\Sigma$ are in similar correspondence with connected components of the graph $\chi$.
As moreover the degree of $0$ in the graph $\chi^0$ is finite we get $\Lambda_r \subset \Lambda'$.
Finally, $\Lambda' \subset \Lambda$ by ergodicity.
Thus the four sets are equal. 
In particular, $\Lambda_r=\Lambda$ and this proves \eqref{e:lambda_c_01}.

\def\cprime{$'$} \def\cprime{$'$}

\end{document}